\documentclass[a4paper]{article}
\usepackage{amssymb}
\usepackage{amsmath}
\usepackage{amsfonts}
\usepackage[margin=1.0in]{geometry}
\usepackage{graphicx}
\usepackage{tikz}

\numberwithin{equation}{section}
\usetikzlibrary{matrix}
\usetikzlibrary{calc}
\newtheorem{theorem}{Theorem}

\newtheorem{condition}[theorem]{Condition}

\newtheorem{corollary}[theorem]{Corollary}

\newtheorem{definition}[theorem]{Definition}
\newtheorem{example}[theorem]{Example}

\newtheorem{lemma}[theorem]{Lemma}
\newtheorem{notation}[theorem]{Notation}

\newtheorem{proposition}[theorem]{Proposition}
\newtheorem{remark}[theorem]{Remark}

\newenvironment{proof}[1][Proof]{\noindent\textbf{#1.} }{\ \rule{0.5em}{0.5em}}
\input{tcilatex}
\newcommand*{\Lcdot}
{\raisebox{-0.25ex}{\scalebox{2.0}{$\cdot$}}}
\newcommand*{\Bcdot}
{\raisebox{-0.5ex}{\scalebox{2.0}{$\cdot$}}}
\newcommand*{\Sbullet}
{\scalebox{0.6}{$\bullet$}}

\begin{document}

\title{Integration of time-varying cocyclic one-forms \\
against rough paths}
\author{Terry J. Lyons, \ Danyu Yang \thanks{
The authors would like to acknowledge the support of the Oxford-Man
Institute and the support provided by ERC advanced grant ESig (agreement no.
291244).}}
\maketitle

\begin{abstract}
We embed the rough integration in a larger geometrical/algebraic framework
of integrating one-forms against group-valued paths, and reduce the rough
integral to an inhomogeneous analogue of the classical Young integral. We
define dominated paths as integrals of one-forms, and demonstrate that they
are stable under basic operations.
\end{abstract}

\tableofcontents

\newpage

\section{Introduction}

\subsection{Overview}

In the early days of rough paths theory, and in the earlier work of Young 
\cite{young1936inequality}, it was understood that there is a natural
interplay between Lipschitz functions (or one-forms) and rough paths, with
the integration being the intermediary. There are a number of
characterizations for the Banach space of Lipschitz functions of degree $%
\gamma >0$, and in this article we particularly take the view of Stein \cite%
{stein1970singular}, that if $\mathcal{K}$ is a subset of an affine space $%
\mathcal{S}$ then a Lipschitz function on $\mathcal{K}$ of degree $\gamma $
is a continuous map $f:x\mapsto p\left( x\right) \left( \cdot \right) $
taking values in polynomial functions\footnote{%
A polynomial function of degree (at most) $n$ is a globally defined function
whose $(n+1)^{\text{th}}$ derivative exists and is identically zero.} on $%
\mathcal{S}$ of degree strictly less than $\gamma $. The Lipschitz degree of 
$f$ describes the varying speed of the polynomials: the higher the Lipschitz
degree the slower the changing speed of the polynomials. The idea is that $f$
prescribes a consistent family of global functions which are tangent to the
trace $\left\{ p(x)(x)|x\in \mathcal{K}\right\} $. It is the polynomials,
the norms on them, and interactions between them that are critical, and in
general the mapping $x\mapsto p(x)(x)$ does not carry nearly enough
information about $f$ unless $\mathcal{K}$ is open.

A key point about Lipschitz functions we use here is somehow
counter-intuitive. If $\mathcal{K}$ is bounded, open, and connected, and $p$
is a polynomial function, then one should think of the function $f:x\mapsto
p(\cdot )$ as a \emph{constant} Lipschitz function. Polynomial functions are
in this sense constant functions in the space of Lipschitz functions. The
view of polynomials as basic ingredients in the larger space of Lipschitz
functions, or more generally the view of closed (cocyclic) one-forms as the
basic ingredients in the space of time-varying one-forms, is at the heart of
the framework introduced in this paper.

In particular, we reinterpret the rough line integral as the integral of a
slowly-varying polynomial one-form against a rough path, where there is
neither a given point nor a power series associated with a polynomial, and
the customary view as a power series around a point on the path somewhat
clouds this understanding because of the erratic movement of the point as
the path evolves.

The original integration in the theory of rough paths defines an integral 
\begin{equation*}
\int_{u\in \left[ s,t\right] }\alpha \left( X_{u}\right) dX_{u}
\end{equation*}%
for a $\limfunc{Lip}(\gamma )$ one-form $\alpha $ against a $p$-rough path $%
X $ for $p<\gamma +1$. A crucial restrictive assumption was that the
one-form $\alpha $ depended on $X_{u}$ but not on $u$ directly. The
assumption, standard in the It\^{o} theory, that $\alpha \in L^{2}\left(
du\right) $ is far too permissive for a deterministic theory. On the other
hand, it was shown in \cite{lyons1998differential} that if $X$ is a $p$%
-rough path and $h$ is a continuous path of finite $q$-variation then $%
\left( X,h\right) $ is canonically a rough path providing $p^{-1}+q^{-1}>1$
(see also \cite{lejay2006p, friz2010multidimensional, hairer2015geometric}).
Letting $h\left( t\right) =t$ allows one to consider%
\begin{equation*}
\int_{u\in \left[ s,t\right] }\alpha \left( X_{u},u\right) dX_{u}
\end{equation*}%
with appropriate smoothness assumptions on $\alpha $. L. G. Gyurko, in his
thesis, gave mixed smoothness conditions that ensure the integral is well
defined. In this paper we introduce a geometrically richer class of
integrands (needed to get the algebraic closure) and provide a stronger and
more intrinsic approach to the integration of one-forms.

Starting with a geometric rough path $X$, one sees from \cite%
{gubinelli2004controlling} (equation $\left( 18\right) $, and the next
section about integration) that a path controlled by $X$ is a path $Y$ that
can formally be identified with the integral of a time varying one-form $G\ $%
against $X$ and modified by (or identified up to) a path of bounded $q$%
-variation where $p^{-1}+q^{-1}>1$. The smoothness condition assumed for $G$
in \cite{gubinelli2004controlling} is less restrictive than ours and does
not guarantee the existence of the integral against $X$. On the other hand,
geometrically the condition on $G$ in \cite{gubinelli2004controlling} is
more restrictive because $G$ is a one-form on the base Banach space whereas
we consider one-forms on the group. This allows us to prove an algebra
property for integrals and also makes the dominated rough path a function of
the integrand giving a simple linear parameterisation of the space of
controlled rough paths. We could also consider the semi-martingale like
spaces when one adds a $q$-variation perturbation. All integrals make sense
as was observed in \cite{lyons1998differential} page 259. However, it seems
interesting to understand the integrals as a class in its own right before
considering such perturbations. This way we introduce a natural and
approximately dense space of functions on rough path space. Adding the $q$%
-variation perturbations would remove that clarity.

The initial goal of the theory of rough paths is to tackle the
non-closability of the integral map for paths of low regularity. Lyons \cite%
{lyons1998differential} observed that the integral map becomes continuous
(and so closable), if one lifts the original integral in a Banach space to a
consistent integral in a topological group. This lift is essentially
nonlinear due to the nonlinearity of the group and is provably necessary
even to have an integral defined deterministically for almost all Brownian
sample paths. The integrals of a fixed rough path are jointly a rough path
so collectively they have a linear structure, which is important in the
proof of the unique existence of solutions to rough differential equations 
\cite{lyons1998differential}. The linear structure is captured in a
beautiful way by Gubinelli \cite{gubinelli2004controlling,
gubinelli2010ramification} and he defined weakly controlled paths as a class
of paths whose local behavior is comparable to a given rough path. For a
fixed reference rough path, the space of controlled paths is linear, and
there exists a canonical enhancement of a controlled path to a group-valued
path (when $2\leq p<3$). The linearity of space and the existence of
canonical enhancement are nice properties that general rough paths can not
have, and they give considerable convenience e.g. when solving a rough
differential equation. In \cite{gubinelli2010ramification}, Gubinelli
defined branched rough paths, and established the relationship between the
evolution of a branched rough path and the Connes-Kremier Hopf algebra \cite%
{connes1998hopf} (see also Butcher group \cite{butcher1972algebraic}). He
defined weakly controlled paths for branched rough paths, and defined the
integration of a weakly controlled path for $p\geq 1$. More recently, Friz
and Hairer \cite{friz2014course} summarized key theorems in the theory of
rough paths by employing Gubinelli's approach, and combined it with a brief
introduction to the recent breakthrough made by the theory of regularity
structures \cite{hairer2014theory}. In particular they defined controlled
rough paths as functions taking values in tensor algebra and defined the
integration of a controlled path accordingly. The theory of rough paths has
a wealth of literature, and there are many other formulations, e.g. \cite%
{Davie07differentialequations, friz2008euler, feyel2006curvilinear,
hu2009rough} etc. For a range of more detailed expositions, see \cite%
{lyons2002system, lejay2003introduction, lejay2009yet,
lyons2007differential, friz2010multidimensional}.

We used the graded algebraic structure in proving the existence of the
integral and in defining the set of dominated paths, so our setting is not
far from the tensor algebra used in \cite%
{lyons1998differential,gubinelli2004controlling} and the Connes-Kremier Hopf
algebra used in \cite{gubinelli2010ramification}. The barrier between tensor
algebra and Connes-Kremier Hopf algebra is not rigid. In \cite%
{hairer2015geometric} Hairer and Kelly proved that branched rough paths can
equally be defined as H\"{o}lder paths taking values in some Lie group, and
that every branched rough path can be encoded in a geometric rough path via
a graded morphism of Hopf algebras so that solving a differential equation
driven by a branched rough path is equivalent to solving an extended
differential equation driven by a geometric rough path. In this paper, we
identify structural properties of a Banach algebra and its associated
topological group that enable basic operations, and construct an algebraic
framework that subsumes tensor algebra and Connes-Kremier Hopf algebra. The
theory of rough paths provides a natural framework to integrate group-valued
paths, and is the incentive of this paper.

Popular approaches to rough integration use a representation of the group in
the truncated tensor algebra to linearize the group-valued path, and treat a
rough path as the collection of several Banach-space valued paths with
certain algebraic structure. We emphasize here an alternative
algebraic/geometrical approach, and develop an integration directly for
one-forms against group-valued paths. The generalization is needed to get
the algebraic closure of integrals. Indeed, suppose $w$ is Brownian motion,
and $\gamma ^{i}$, $i=1,2$, are suitable integrands. We want to find $\gamma 
$ that satisfies%
\begin{equation*}
\int_{0}^{\cdot }\gamma _{u}dw_{u}\overset{?}{=}\int_{0}^{\cdot }\gamma
_{u}^{1}dw_{u}\int_{0}^{\cdot }\gamma _{u}^{2}dw_{u}\text{.}
\end{equation*}%
Based on It\^{o}'s lemma there exists a drift term in the product that can
not be represented in the form of an integral against $w$, so such $\gamma $
does not exist (see Theorem 1 \cite{friz2012doob} for a pathwise
generalization). A generalized integral is therefore needed to get the
algebraic closure (even for almost all Brownian sample paths) that is
important for proving density in paths space. The integral developed here is
rich enough to handle the product structure of rough integrals, and the
multiplication is in fact a continuous operation in the space of one-forms
(Proposition \ref{Proposition Algebra}). In particular, suppose $y^{i}$, $%
i=1,2$, solves the It\^{o} differential equation $dy^{i}=f\left(
y^{i}\right) dw$, $y_{0}^{i}=\xi ^{i}$, for Brownian motion $w$. By using
the integral developed in this paper, $y^{i}$ and their product $y^{1}y^{2}$
solve the same type of equation, and the product $\left( y^{1},y^{2}\right)
\mapsto y^{1}y^{2}$ is a continuous operation. The algebra structure is
compatible with the filtration generated on paths space, and the product of
two previsible integrable one-forms is another previsible integrable
one-form. This is the same property that was exploited in the Martingale
Representation Theorem.

By introducing a family of closed one-forms on a group, we construct the
integral of one-forms against group-valued paths, and provide a simple and
unified interpretation of the extension theorem and the theories of
integration in rough paths theory \cite{lyons1998differential,
gubinelli2004controlling, gubinelli2010ramification}. We are able to allow
the one-form on the group to vary with time. As a consequence, the integral
is not restricted to use the same cotangent vector at distinct times of
self-intersection.

We identity structural properties of the group that enable basic operations,
e.g. rough integration and iterated integration. Condition \ref{Condition g
satisfies differential equation}' is for rough integration, and encodes the
information of how to integrate monomials against degree-one monomial on
paths space (Corollary \ref{Example rough integral}, Remark \ref{Remark
polynomial one-form Butcher group}). Condition \ref{Condition g satisfies
differential equation} is for iterated integration that encodes the
information of how to integrate monomials against monomials (not only
degree-one monomial) on paths space (Proposition \ref{Proposition
enhancement}, Corollary \ref{Example weakly controlled path}). These
algebraic conditions are structural assumptions on the group, and can be
viewed as counterparts to Chen's identity (that is about paths evolution) in
paths integration.

The process of rough integration can be split into two (independent) steps:
(1) integrating polynomials and increasing the regularity of the Lipschitz
function (2) constructing a group homomorphism from the Lipschitz function.
The integration we constructed (Theorem \ref{Theorem integrating slow
varying cocyclic one forms}) is about the homomorphism, and is not
(directly) related to increasing the regularity of the Lipschitz function.
The regularity of the Lipschitz function is increased when we specify the
one-form (Corollary \ref{Example rough integral}, Remark \ref{Remark
polynomial one-form Butcher group}).

By using one-forms, basic operations --- such as iterated integration,
multiplication and composition with regular functions --- are continuous
operations (Section \ref{Section stableness of dominated paths}). The
continuity gives considerable analytical robustness. In particular, the
enhancement to a group-valued path is a continuous operation. That is
applicable when the enhancement involves basic operations e.g. taking values
in nilpotent Lie group or Butcher group.

The approach is employed in \cite{lyons2015theory} to extend an argument of
Schwartz \cite{schwartz1989convergence} to rough differential equations, and
give a short proof of the global unique solvability and stability of the
solution that is applicable to geometric rough paths and branched rough
paths. Consider the rough differential equation%
\begin{equation*}
dy=f\left( y\right) dx\text{, }y_{0}=\xi \text{,}
\end{equation*}%
with Picard iterations $y_{\cdot }^{n}:=\xi +\int_{0}^{\cdot }f\left(
y^{n-1}\right) dx$, $n\geq 1$, $y_{\cdot }^{0}\equiv \xi $. When $f$ is $%
\limfunc{Lip}\left( \gamma \right) $ for $\gamma >p$, there exist one-forms $%
\left( \beta ^{n}\right) _{n}$ such that $y_{\cdot }^{n}=\xi
+\int_{0}^{\cdot }\beta ^{n}\left( g\right) dg$ and $\left( \beta
^{n+1}-\beta ^{n}\right) _{n}$ decay factorially in operator norm (Theorem
22 \cite{lyons2015theory}). More explicitly, there is a constant $%
C=C(p,\gamma ,\left\Vert f\right\Vert _{Lip\left( \gamma \right) },\omega
\left( 0,T\right) )$ such that, with $\left\Vert \cdot \right\Vert _{\theta
}^{\omega }$ in $\left( \ref{Definition of operator norm}\right) $, control $%
\omega :=\left\Vert g\right\Vert _{p-var}^{p}$ and $\theta :=\frac{\gamma
\wedge \left( \left[ p\right] +1\right) }{p}>1$, we have 
\begin{equation*}
\left\Vert \beta ^{n+1}-\beta ^{n}\right\Vert _{\theta }^{\omega }\leq \frac{%
C^{n-\left[ p\right] }}{\left( \frac{n-\left[ p\right] }{p}\right) !}\text{, 
}n\geq \left[ p\right] +1\text{.}
\end{equation*}%
Since the enhancement to a group-valued path is a continuous operation in
the space of one-forms and the indefinite integral is a continuous operation
from one-forms to paths, we have the convergence of Picard iterations and
their group-valued enhancements.

The basic idea of the integral in this paper can be summarized as follows.
Suppose $x$ is a continuous bounded variation path taking values in $%
\mathbb{R}
^{d}$ and $\alpha $ is a one-form. Let $G^{(n)}(%
\mathbb{R}
^{d})$ denote the step-$n$ free nilpotent Lie group over $%
\mathbb{R}
^{d}$ (see \cite{lyons1998differential}) with projection $\pi $ onto $%
\mathbb{R}
^{d}$. Then we may lift $x$ up to a path $g$ taking values in $G^{(n)}(%
\mathbb{R}
^{d})$\ (the signature of $x$), pull $\alpha $ up to $\alpha ^{\ast }$ using 
$\pi $, and get%
\begin{equation*}
\int \alpha \left( x\right) dx=\int \alpha ^{\ast }\left( g\right) dg\text{.}
\end{equation*}%
This equality holds because $x=\pi g$ and has little to do with $g_{t}\in
G^{(n)}(%
\mathbb{R}
^{d})$. Since the dimension of the tangent space to $g$ in $G^{(n)}(%
\mathbb{R}
^{d})$ is much larger than that of $\mathcal{%
\mathbb{R}
}^{d}$, there are other choices of one-forms that give the same integral.\
In particular, when $\alpha =p$ is a polynomial one-form of degree $(n-1)$,
there is a unique closed (in fact cocyclic) one-form $P$ on $G^{(n)}(%
\mathbb{R}
^{d})$ (as defined in Example \ref{Example polynomial cocyclic one-form}
below) that only depends on $p$ and satisfies%
\begin{equation}
\int p\left( x\right) dx=\int P\left( g\right) dg\text{.}
\label{integral of p}
\end{equation}%
Because this statement holds exactly rather than infinitesimally, it gives
considerable analytical flexibility. As we suggested before, Lipschitz
functions are continuous functions taking values in polynomial functions.
Since each polynomial one-form can be lifted to a closed one-form on the
nilpotent Lie group, a Lipschitz one-form can be lifted to a continuous
function taking values in closed one-forms on the nilpotent Lie group. In
particular for Lipschitz function $\alpha $, if we denote the lift of $%
\alpha $ by $\beta $, then for any continuous bounded variation path $x$
with lift $g$,%
\begin{equation}
\int \alpha \left( x\right) dx=\int \beta \left( g\right) dg\text{.}
\label{integral of alpha}
\end{equation}%
When $x$ is not necessarily of bounded variation, the integrals $\int
p\left( x\right) dx$ in $\left( \ref{integral of p}\right) $ and $\int
\alpha \left( x\right) dx$ in $\left( \ref{integral of alpha}\right) $ may
not be meaningfully defined. Since $P$ is a closed one-form, the integral $%
\int P\left( g\right) dg$ is well defined for any continuous path $g$. Since 
$\beta $ is a function taking values in closed one-forms in the form of $P$,
the integral $\int \beta \left( g\right) dg$ should still make sense when
the closed one-form varies slowly. For example, in the extreme case that $%
\beta $ is a constant closed one-form $P$, $\int \beta \left( g\right) dg$
coincides with $\int P\left( g\right) dg$ so is well defined. Indeed, it can
be proved that when $\beta $ is the lift of a $Lip\left( \gamma \right) $
function and $g$ is a continuous path with finite $p$-variation for $\frac{%
\gamma +1}{p}>1$, the integral $\int \beta \left( g\right) dg$ is well
defined and coincides with the rough integral in \cite{lyons1998differential}%
. The existence of the integral $\int \beta \left( g\right) dg$ relies on
that $\beta $ consists a family of well-behaved one-forms (the $P$s) that
vary slowly along the trajectory of $g$ (captured by the Lipschitz degree).
Closed one-forms are well-behaved because they put little assumption on the
path for the integral to make sense, and they can serve as basic ingredients
in the space of more general one-forms (like constant one-forms in the space
of continuous one-forms in the classical integration). Generally, the
integral $\int_{t\in \left[ 0,T\right] }\beta _{t}\left( g_{t}\right) dg_{t}$
is well defined, when $\beta _{t}$ is a continuous path taking values in
closed one-forms and the two dual paths $\beta _{t}$ and $g_{t}$ satisfy a
generalized Young condition (see construction of the integral in Section \ref%
{Subsection existence of integral}).

Polynomial one-forms are\ basic ingredients for the rough integration in 
\cite{lyons1998differential}, and serve as the primary example in this paper.

\subsection{Cocyclic one-forms}

Cocyclic one-forms are closed one-forms on a topological group. They can be
integrated against \textit{any} continuous path taking values in the group,
and the value of the integral only depends on the path through end points.

Suppose $\mathcal{A}$ and $\mathcal{B}$ are two Banach algebras and $%
\mathcal{G}$ is a topological group in $\mathcal{A}$. We denote by $L\left( 
\mathcal{A},\mathcal{B}\right) $ the set of continuous linear mappings from $%
\mathcal{A}$ to $\mathcal{B}$, and denote by $C\left( \mathcal{G},L\left( 
\mathcal{A},\mathcal{B}\right) \right) $ the set of continuous mappings from 
$\mathcal{G}$ to $L\left( \mathcal{A},\mathcal{B}\right) $.

\begin{definition}[Cocyclic One-Form]
\label{Definition of cocyclic one-form}We say $\beta \in C\left( \mathcal{G}%
,L\left( \mathcal{A},\mathcal{B}\right) \right) $ is a cocyclic one-form, if
there exists a topological group $\mathcal{H}$ in $\mathcal{B}$ such that $%
\beta \left( a,b\right) \in \mathcal{H}$, $\forall a,b\in \mathcal{G}$, and 
\begin{equation}
\beta \left( a,b\right) \beta \left( ab,c\right) =\beta \left( a,bc\right) 
\text{, }\forall a,b,c\in \mathcal{G}\text{.}  \label{definition of cocycle}
\end{equation}%
We denote the set of cocyclic one-forms by $B\left( \mathcal{G},\mathcal{H}%
\right) $ (or $B\left( \mathcal{G}\right) $).
\end{definition}

Equation $\left( \ref{definition of cocycle}\right) $ represents the exact
equality between the one-step and two-steps estimates that characterizes
closed one-forms. Intuitively, if we start from point $a$ and go in the
direction of $bc$, then it is equivalent to start from point $a$, go in the
direction of $b$, and start again from point $ab$ and then go in the
direction of $c$.

Cocyclic one-forms are of specific form but abundant; they are fundamental
in integration. A simple example is the constant one-form on a Banach space,
which we use implicitly in the classical integration. Another example is the
polynomial one-form in rough paths theory \cite{lyons1998differential}.

Recall that a polynomial function of degree (at most) $n$ is a globally
defined function whose $(n+1)^{\text{th}}$ derivative exists and is
identically zero.

\begin{definition}[Polynomial One-Form]
For Banach spaces $\mathcal{V}$ and $\mathcal{U}$, we say $p\in C\left( 
\mathcal{V},L\left( \mathcal{V},\mathcal{U}\right) \right) $ is a polynomial
one-form of degree $n$ if $p$ is a polynomial function of degree $n$ taking
values in $L\left( \mathcal{V},\mathcal{U}\right) $.
\end{definition}

In particular, by using Taylor's theorem, we have (with $\otimes $ denoting
the tensor product) 
\begin{equation}
p\left( z\right) \left( v\right) =\sum_{l=0}^{n}\left( D^{l}p\right) \left(
y\right) \frac{\left( z-y\right) ^{\otimes l}}{l!}\left( v\right) \text{, }%
\forall z,v,y\in \mathcal{V}\text{,}  \label{polynomial one form}
\end{equation}%
where $\left( D^{l}p\right) \left( y\right) \in L\left( \mathcal{V}^{\otimes
l},L\left( \mathcal{V},\mathcal{U}\right) \right) $ denotes the value at $y$
of the $l$-th derivative of $p$. Like polynomial functions, there is neither
a given point nor a power series associated with a polynomial one-form. One
can choose different representations of a polynomial one-form by choosing
the point $y$ in $\left( \ref{polynomial one form}\right) $, and the value
of $p\ $does not depend on the choice of $y$.

Then we enhance a polynomial one-form to a cocyclic one-form on a group. Let 
$G^{(n)}\left( \mathcal{V}\right) $ denote the step-$n$ free nilpotent Lie
group over Banach space $\mathcal{V}$, which is canonically embedded in the
Banach algebra $T^{(n)}\left( \mathcal{V}\right) =%
\mathbb{R}
\oplus \mathcal{V}\oplus \cdots \oplus \mathcal{V}^{\otimes n}$, and let $%
\pi _{k}$ denote the projection of $T^{(n)}\left( \mathcal{V}\right) $ to $%
\mathcal{V}^{\otimes k}$. For integers $l_{i}\geq 1$, $i=1,2,\dots ,k$, we
let 
\begin{equation*}
OS\left( l_{1},l_{2},\dots ,l_{k}\right)
\end{equation*}%
denote the ordered shuffles of $k$ stacks of cards with $l_{1},l_{2},\dots
,l_{k}$ cards respectively (p.73, \cite{lyons2007differential}).

\begin{example}[Polynomial Cocyclic One-Form]
\label{Example polynomial cocyclic one-form}Suppose $p\in C\left( \mathcal{V}%
,L\left( \mathcal{V},\mathcal{U}\right) \right) $ is a polynomial one-form
of degree $\left( n-1\right) $ for some integer $n\geq 1$. We define $P\in
C(G^{(n^{2})}\left( \mathcal{V}\right) ,L(T^{(n^{2})}\left( \mathcal{V}%
\right) ,T^{(n)}\left( \mathcal{U}\right) ))$ by, for $a\in
G^{(n^{2})}\left( \mathcal{V}\right) $ and $v\in T^{(n^{2})}\left( \mathcal{V%
}\right) $,%
\begin{equation}
P\left( a,v\right) :=1+\sum_{k=1}^{n}\sum_{l_{i}\in \left\{ 0,1,\dots
,n-1\right\} }\left( \left( D^{l_{1}}p\right) \otimes \cdots \otimes \left(
D^{l_{k}}p\right) \right) \left( \pi _{1}\left( a\right) \right) \sum_{\rho
\in OS\left( l_{1}+1,\dots ,l_{k}+1\right) }\rho ^{-1}\left( \pi
_{l_{1}+\cdots +l_{k}+k}\left( v\right) \right) \text{.}
\label{expression of polynomial cocyclic one-form}
\end{equation}%
Then $P$ is a cocyclic one-form, i.e. for $a,b\in G^{(n^{2})}\left( \mathcal{%
V}\right) $, $P\left( a,b\right) \in G^{(n)}(\mathcal{U})$, and%
\begin{equation*}
P\left( a,b\right) P\left( ab,c\right) =P\left( a,bc\right) \text{, }\forall
a,b,c\in G^{(n^{2})}\left( \mathcal{V}\right) \text{.}
\end{equation*}
\end{example}

For an explanation of the mathematical expression, we suppose $x$ is a
continuous bounded variation path on $\left[ 0,T\right] $ taking values in $%
\mathcal{V}$, and let $S_{n}\left( x\right) $ denote the step-$n$ Signature
of $x$: 
\begin{equation}
S_{n}\left( x\right) _{s,t}:=1+\sum_{k=1}^{n}x_{s,t}^{k}\text{ with }%
x_{s,t}^{k}:=\idotsint\nolimits_{s<u_{1}<\cdots <u_{k}<t}dx_{u_{1}}\otimes
\cdots \otimes dx_{u_{k}}\text{, }\forall 0\leq s\leq t\leq T\text{.}
\label{definition of signature}
\end{equation}%
Based on Chen \cite{chen2001iterated}, $S_{n}\left( x\right) $ takes values
in the step-$n$ nilpotent Lie group $G^{(n)}\left( \mathcal{V}\right) $, and
satisfies: 
\begin{equation*}
\text{\textbf{(Chen's Identity)} }S_{n}\left( x\right) _{s,u}S_{n}\left(
x\right) _{u,t}=S_{n}\left( x\right) _{s,t}\text{, }\forall 0\leq s\leq
u\leq t\leq T\text{,}
\end{equation*}%
where the multiplication on the l.h.s. is in $G^{(n)}\left( \mathcal{V}%
\right) $. In particular, $\left[ 0,T\right] \ni t\mapsto S_{n}\left(
x\right) _{0,t}\in G^{(n)}\left( \mathcal{V}\right) $ is a group-valued path
satisfying $S_{n}\left( x\right) _{0,s}^{-1}S_{n}\left( x\right)
_{0,t}=S_{n}\left( x\right) _{s,t}$, $\forall 0\leq s\leq t\leq T$. Since $%
\left( D^{l}p\right) \left( x_{s}\right) \in L\left( \mathcal{V}^{\otimes
l},L\left( \mathcal{V},\mathcal{U}\right) \right) $ is symmetric in $%
\mathcal{V}^{\otimes l}$ and the projection of $x_{s,t}^{l}$ to the space of
symmetric tensors is $\left( l!\right) ^{-1}\left( x_{t}-x_{s}\right)
^{\otimes l}$ (see \cite{lyons1998differential}), we have 
\begin{equation}
\left( D^{l}p\right) \left( x_{s}\right) \frac{\left( x_{t}-x_{s}\right)
^{\otimes l}}{l!}\left( v\right) =\left( D^{l}p\right) \left( x_{s}\right)
\left( x_{s,t}^{l}\right) \left( v\right) \text{, }\forall v\in \mathcal{V}%
\text{, }\forall 0\leq s\leq t\leq T\text{.}
\label{polynomial one-form symmetric tensor}
\end{equation}%
Then, based on the expressions $\left( \ref{polynomial one form}\right) $
and $\left( \ref{polynomial one-form symmetric tensor}\right) $, and by
using $x_{s,t}^{l+1}=\int_{s}^{t}x_{s,r}^{l}\otimes dx_{r}$, $l=0,\dots ,n-1$%
, we have 
\begin{equation*}
\int_{s}^{t}p\left( x_{r}\right) dx_{r}=\sum_{l=0}^{n-1}\left( D^{l}p\right)
\left( x_{s}\right) \int_{s}^{t}\frac{\left( x_{r}-x_{s}\right) ^{\otimes l}%
}{l!}\otimes dx_{r}=\sum_{l=0}^{n-1}\left( D^{l}p\right) \left( x_{s}\right)
\int_{s}^{t}x_{s,r}^{l}\otimes dx_{r}=\sum_{l=0}^{n-1}\left( D^{l}p\right)
\left( x_{s}\right) x_{s,t}^{l+1}\text{,}
\end{equation*}%
where $\left( D^{l}p\right) \left( x_{s}\right) $ contracts with $%
x_{s,t}^{l+1}$ because $\left( D^{l}p\right) \left( x_{s}\right) \in L\left( 
\mathcal{V}^{\otimes l},L\left( \mathcal{V},\mathcal{U}\right) \right) $ so $%
\left( D^{l}p\right) \left( x_{s}\right) \left( x_{s,r}^{l}\right) \in
L\left( \mathcal{V},\mathcal{U}\right) $ and $\left( D^{l}p\right) \left(
x_{s}\right) \int_{s}^{t}x_{s,r}^{l}\otimes dx_{r}=\left( D^{l}p\right)
\left( x_{s}\right) x_{s,t}^{l+1}\in \mathcal{U}$. Hence, for $0\leq s\leq
u\leq t\leq T$, by using $\left( \ref{polynomial one form}\right) $ and that 
$\int_{u}^{t}=\int_{s}^{t}-\int_{s}^{u}$, we have%
\begin{eqnarray*}
\sum_{l=0}^{n-1}\left( D^{l}p\right) \left( x_{u}\right) x_{u,t}^{l+1}
&=&\sum_{l=0}^{n-1}\left( D^{l}p\right) \left( x_{u}\right)
\int_{u}^{t}x_{u,r}^{l}\otimes dx_{r}=\int_{u}^{t}p\left( x_{r}\right)
dx_{r}=\sum_{l=0}^{n-1}\left( D^{l}p\right) \left( x_{s}\right)
\int_{u}^{t}x_{s,r}^{l}\otimes dx_{r} \\
&=&\sum_{l=0}^{n-1}\left( D^{l}p\right) \left( x_{s}\right)
x_{s,t}^{l+1}-\sum_{l=0}^{n-1}\left( D^{l}p\right) \left( x_{s}\right)
x_{s,u}^{l+1}\text{.}
\end{eqnarray*}%
As a result, if we define path $y:\left[ 0,T\right] \rightarrow \mathcal{U}$
by%
\begin{equation*}
y_{t}:=\int_{0}^{t}p\left( x_{r}\right) dx_{r}=\sum_{l=0}^{n-1}\left(
D^{l}p\right) \left( x_{0}\right) x_{0,t}^{l+1}\text{, \ \ }\forall 0\leq
t\leq T\text{,}
\end{equation*}%
then 
\begin{equation}
y_{t}-y_{s}=\sum_{l=0}^{n-1}\left( D^{l}p\right) \left( x_{s}\right)
x_{s,t}^{l+1}\text{, \ \ }\forall 0\leq s\leq t\leq T\text{.}
\label{increment of y}
\end{equation}%
Based on the definition of the Signature in $\left( \ref{definition of
signature}\right) $ and the representation of $y$ in $\left( \ref{increment
of y}\right) $, we have 
\begin{equation}
S_{n}\left( y\right) _{s,t}=1+\sum_{k=1}^{n}\sum_{l_{i}\in \left\{ 0,1,\dots
,n-1\right\} }\left( \left( D^{l_{1}}p\right) \otimes \cdots \otimes \left(
D^{l_{k}}p\right) \right) \left( x_{s}\right)
\idotsint\limits_{s<u_{1}<\cdots <u_{k}<t}dx_{s,u_{1}}^{l_{1}+1}\otimes
\cdots \otimes dx_{s,u_{k}}^{l_{k}+1}\text{.}
\label{representation of step-m siganture of y}
\end{equation}%
By using the ordered shuffle product (p. 73, \cite{lyons2007differential}), $%
\tidotsint\nolimits_{s<u_{1}<\cdots <u_{k}<t}dx_{s,u_{1}}^{l_{1}+1}\otimes
\cdots \otimes dx_{s,u_{k}}^{l_{k}+1}$ can be rewritten as a universal
continuous linear function of $S_{n^{2}}\left( x\right) _{s,t}$ that is
independent of $x$:%
\begin{equation}
\tidotsint\nolimits_{s<u_{1}<\cdots <u_{k}<t}dx_{s,u_{1}}^{l_{1}+1}\otimes
\cdots \otimes dx_{s,u_{k}}^{l_{k}+1}=\tsum\nolimits_{\rho \in OS\left(
l_{1}+1,\dots ,l_{k}+1\right) }\rho ^{-1}\left( \pi _{l_{1}+\cdots
+l_{k}+k}(S_{n^{2}}\left( x\right) _{s,t})\right) \text{.}
\label{representation of step-m siganture of y 2}
\end{equation}%
As we mentioned before, if we denote $g_{t}^{n^{2}}:=S_{n^{2}}\left(
x\right) _{0,t}$, $\forall t\in \left[ 0,T\right] $, then $g^{n^{2}}$ is a
group-valued path taking values in the step-$n^{2}$ free nilpotent Lie group 
$G^{(n^{2})}\left( \mathcal{V}\right) $. Based on the representations $%
\left( \ref{representation of step-m siganture of y}\right) $ and $\left( %
\ref{representation of step-m siganture of y 2}\right) $, if we define $P$
as in $\left( \ref{expression of polynomial cocyclic one-form}\right) $,
then, with $g_{s,t}^{n^{2}}:=(g_{s}^{n^{2}})^{-1}g_{t}^{n^{2}}$,%
\begin{equation}
P\left( g_{s}^{n^{2}},g_{s,t}^{n^{2}}\right) =S_{n}\left( y\right) _{s,t}%
\text{, }\forall 0\leq s\leq t\leq T\text{.}  \label{definition of P}
\end{equation}%
Based on $\left( \ref{definition of P}\right) $, $P$ takes values in the
step-$n$ nilpotent Lie group $G^{(n)}\left( \mathcal{U}\right) $ and
satisfies Chen's identity:%
\begin{equation}
P\left( g_{s}^{n^{2}},g_{s,u}^{n^{2}}\right) P\left(
g_{u}^{n^{2}},g_{u,t}^{n^{2}}\right) =P\left(
g_{s}^{n^{2}},g_{s,t}^{n^{2}}\right) \text{, }\forall 0\leq s\leq u\leq
t\leq T\text{.}  \label{polynomial one form is cocyclic}
\end{equation}%
Comparing $\left( \ref{polynomial one form is cocyclic}\right) $ with the
cocyclic property defined at $\left( \ref{definition of cocycle}\right) $,
we have that, $P$ is a cocyclic one-form on group $G^{(n^{2})}\left( 
\mathcal{V}\right) $ taking values in another group $G^{(n)}\left( \mathcal{U%
}\right) $.

That polynomial one-forms can be lifted to closed (cocyclic) one-forms is
not a coincidence, and follows from properties of polynomials on paths
space. A polynomial is a finite linear combination of monomials of an
indeterminate. A polynomial on paths space in our setting is a finite linear
combination of monomials of a path, i.e. its iterated integrals (see Chen 
\cite{chen1957integration} for geometrical interpretation of monomials on
paths space). (Following from algebraic properties of the group, classical
polynomials correspond naturally to polynomials on paths space, see
Corollary \ref{Example rough integral} and Remark \ref{Remark polynomial
one-form Butcher group}.) The space of paths (modulus translations,
reparametrisations and tree-like equivalence \cite{hambly2010uniqueness,
boedihardjo2014signature}) forms a group when equipped with the product
given by concatenation \cite{lyons1998differential}. The step-$n$ signature $%
S_{n}\left( x\right) $ can be viewed as the graded composition of the first $%
n$ monomials of the path $x$. Thus Chen's identity encodes the change of the
representation of a polynomial on paths space as the reference point (a
path) changes. In particular, for a continuous linear mapping $A:T^{\left(
n\right) }\left( \mathcal{V}\right) \rightarrow \mathcal{U}$, $p\left(
x\right) :=AS_{n}\left( x\right) $ is the representation of a polynomial
when the reference point is the constant path. Based on Chen's identity, the
representation of $p$ at a path $y$ is $p\left( x\right) =AS_{n}\left(
y\right) S_{n}\left( \overleftarrow{y}\sqcup x\right) $ (with $%
\overleftarrow{y}\sqcup x$ denotes the concatenation of $y$ backwards with $%
x $). From this perspective, rough integration (for geometric rough paths
and branched rough paths alike) can be viewed as the sewing process of
constructing a group homomorphism. Indeed, for a given rough path, rough
integration constructs from a family of polynomials (indexed by the
evolution of the rough path) a homomorphism (from paths space to another
group) with prescribed local behavior (the homomorphism and the Lipschitz
function are locally equivalent). That the family of polynomials on paths
space is stable under the signature mapping is based on algebraic properties
of the group --- the existence of the mapping $\mathcal{I}$ in Condition \ref%
{Condition g satisfies differential equation}. The cocyclic one-form
associated with the signature is simple: $\beta \left( a,b\right) =b$, $%
\forall a,b\in \mathcal{G}$, so $\beta \left( a,b\right) \beta \left(
ab,c\right) =bc=\beta \left( a,bc\right) $, $\forall a,b,c\in \mathcal{G}$.
The equality $\left( \ref{polynomial one form is cocyclic}\right) $ is hence
a synthesis of Chen's identity about changing the reference point of a
polynomial and the stability of polynomials under the signature mapping,
encodes the change of a cocyclic one-form under the change of the base
group-valued path, and is a transitive property (see also Proposition \ref%
{Proposition Transitivity}).

The following diagram illustrates the relationship between paths and
one-forms:

\begin{center}
\begin{tikzpicture}
  \matrix (m) [matrix of math nodes,row sep=3em,column sep=0em,minimum width=2em, column 1/.style={anchor=base east}, column 7/.style={anchor=base west}]
  {
   C\left(\left[0,T\right],G^{(n)}(\mathcal{V})\right)& \ni  & g^n     & + & \beta &\in & C\left(\left[0,T\right],
   B\left(G^{(n)}(\mathcal{V}),G^{(n)}(\mathcal{U})\right)\right)\\
   C(\left[0,T\right],G^{(n^2)}(\mathcal{V}))& \ni & g^{n^2} & + & P  & \in & B(G^{(n^2)}(\mathcal{V}),G^{(n)}(\mathcal{U}))\\
   C\left(\left[0,T\right],\mathcal{V}\right)      &\ni      & x       & + & p  &\in   & C\left(\mathcal{V},L\left(\mathcal{V},\mathcal{U}\right)\right)\\};
  \path[-stealth]
    ($(m-3-3.north)+(+0.05,0)$) edge [dashed] node [right] {signature}($(m-2-3.south)+(+0.05,0)$)
    ($(m-2-3.south)+(-0.05,0)$) edge node [left] {projection} ($(m-3-3.north)+(-0.05,0)$)
    ($(m-2-3.north)+(-0.05,0)$) edge node [left] {truncation}($(m-1-3.south)+(-0.05,0)$)
    ($(m-1-3.south)+(+0.05,0)$) edge node [right] {extension} ($(m-2-3.north)+(+0.05,0)$);
\end{tikzpicture}
\end{center}

The lower half-diagram summarizes the polynomial one-form and the polynomial
cocyclic one-form we discussed above. Suppose $x$ is a continuous bounded
variation path on $\left[ 0,T\right] $ taking values in $\mathcal{V}$ and $%
p\in C\left( \mathcal{V},L\left( \mathcal{V},\mathcal{U}\right) \right) \ $%
is a polynomial one-form. For integer $n\geq 1$, we can enhance $x$ to its
step-$n^{2}$ signature $g^{n^{2}}$, which is a continuous path taking values
in the step-$n^{2}$ free nilpotent Lie group $G^{(n^{2})}\left( \mathcal{V}%
\right) $. (We put a dashed arrow because, when $x$ is less regular, e.g. a
Brownian sample path, the enhancement exists but may not be unique, see \cite%
{lyons2007extension}.) We can define the polynomial cocyclic one-form $P\ $%
associated with $p$ on $G^{(n^{2})}\left( \mathcal{V}\right) $ taking values
in $G^{(n)}\left( \mathcal{U}\right) $ as in $\left( \ref{expression of
polynomial cocyclic one-form}\right) $. Then based on $\left( \ref%
{definition of P}\right) $,%
\begin{equation}
P\left( g_{s}^{n^{2}},g_{s,t}^{n^{2}}\right) =S_{n}\left( \int_{0}^{\cdot
}p\left( x_{u}\right) dx_{u}\right) _{s,t}\text{, }\forall 0\leq s<t\leq T%
\text{.}  \label{polynomial cocyclic one-form}
\end{equation}

The upper half-diagram is about truncation and keeping the homogeneity of
the degree of the groups so that it would not explode after several
compositions. The lifting to degree-$n^{2}$ is necessary to keep a closed
expression: although polynomials on paths space form an algebra, polynomials
of a specific degree do not form one --- the degree-$n$ polynomial of a
degree-$n$ polynomial is a degree-$n^{2}$ polynomial. If we would like to
define a one-form on $G^{(n)}\left( \mathcal{V}\right) $ (instead of $%
G^{(n^{2})}\left( \mathcal{V}\right) $), we can truncate $g^{n^{2}}$ to $%
g^{n}$, and define $\beta $ based on the truncation of $P$ as a time-varying
cocyclic one-form on $G^{(n)}(\mathcal{V})$ (see Corollary \ref{Example
rough integral}). When $g^{n}$ is of finite $p$-variation for some $\left[ p%
\right] \leq n$, $\beta $ is integrable against $g^{n}$ and satisfies%
\begin{equation}
\int_{s}^{t}\beta _{u}\left( g_{u}^{n}\right) dg_{u}^{n}=P\left(
g_{s}^{n^{2}},g_{s,t}^{n^{2}}\right) \text{, }\forall 0\leq s<t\leq T\text{.}
\label{slow-varying cocyclic polynomial one-form on n group}
\end{equation}%
Due to the truncation of higher leveled terms, the $\beta $ here is no
longer a constant cocyclic one-form as $P$ is, and the equality $\left( \ref%
{slow-varying cocyclic polynomial one-form on n group}\right) $ follows from
the comparison of local expansions. Actually there is no canonical way of
enhancing a polynomial one-form to a cocyclic one-form, and one could as
well define the cocyclic one-form obtained after truncation as the
polynomial cocyclic one-form (denoted by $P^{\prime }$). Then the
homogeneity of the group is preserved, and $\beta _{t}$ is a continuous path
taking values in polynomial cocyclic one-forms in the form of $P^{\prime }$.
In that case, the equality $\left( \ref{polynomial cocyclic one-form}\right) 
$ (with $P(g_{s}^{n^{2}},g_{s,t}^{n^{2}})$ replaced by $P^{\prime
}(g_{s}^{n},g_{s,t}^{n})$) holds not exactly but approximately with a small
error which will disappear in the process of integration, and the equality $%
\int_{s}^{t}\beta _{u}\left( g_{u}^{n}\right) dg_{u}^{n}=S_{n}\left(
\int_{0}^{\cdot }p\left( x_{u}\right) dx_{u}\right) _{s,t}$ still holds.\
The arrow from $g^{n}$ to $g^{n^{2}}$ is about the extension theorem
(Theorem \ref{Example extension}). Indeed, if $g^{n}$ is of finite $p$%
-variation for some $\left[ p\right] \leq n$, then there exists a unique $%
g^{n^{2}}$ which extends $g^{n}$, and $g^{n^{2}}$ can be represented as the
integral of a slow-varying cocyclic one-form against $g^{n}$.

By lifting a polynomial one-form on a Banach space to a polynomial cocyclic
one-form on the corresponding nilpotent Lie group, we linearize the one-form
by working with a larger affine space. The enlargement of the space is
sufficient and necessary for developing robust calculus for paths of low
regularity. More importantly, we view polynomial one-forms, which are by no
means closed one-forms in the classical sense, as closed one-forms on the
lifted group (see $\left( \ref{polynomial cocyclic one-form}\right) $). The
closedness of the one-forms relies on the stability of polynomials on paths
space, and ultimately on the algebraic structure of the group. The
closedness of the one-forms makes the analysis considerably simpler because
the integration of a closed one-form puts little restriction on the path. As
in the classical integration, where we implicitly work with slowly-varying
constant one-forms, we can develop integration of one-forms on a group by
slowly-varying the closed (cocyclic) one-forms on the group. In light of the
nice approximating properties of polynomials, by slowly-varying the
polynomial one-form and by taking closure of the polynomial one-forms
w.r.t.\ an appropriate norm, we can work with a large family of one-forms
(e.g. Lipschitz one-forms as in Corollary \ref{Example rough integral} and
time-varying Lipschitz one-forms as in Remark \ref{Remark integrating
time-varying Lipschitz functions}).

\subsection{Time-varying cocyclic one-forms}

As we suggested before, the integration of time-varying cocyclic one-forms
reduces to the integration of closed one-forms together with pasting closed
one-forms in a consistent way. Before proceeding to the definition of the
integral, we take a look at a little lemma, which shares the same spirit
with our group-valued integration.

\begin{definition}[Finite Partition]
$D=\left\{ t_{k}\right\} _{k=0}^{n}$ is a finite partition of $\left[ 0,T%
\right] $, if $0=t_{0}<t_{1}<\cdots <t_{n}=T$. We denote $\left\vert
D\right\vert :=\max_{k}\left\vert t_{k+1}-t_{k}\right\vert $.
\end{definition}

\begin{lemma}
For a differentiable manifold $M$, we suppose $x$ is a continuous path on $%
\left[ 0,T\right] $ taking values in $M$, and $\alpha $ is a path on $\left[
0,T\right] $ taking values in real-valued closed one-forms on $M$. If there
exists $\theta >1$ such that 
\begin{equation}
\left\vert \int_{u}^{t}\left( \alpha _{s}-\alpha _{u}\right) \left(
x_{r}\right) dx_{r}\right\vert \leq \left\vert t-s\right\vert ^{\theta }%
\text{, }\forall 0\leq s<u<t\leq T\text{,}
\label{condition one-form on manifold}
\end{equation}%
then the integral defined by%
\begin{equation}
\int_{0}^{T}\alpha _{r}\left( x_{r}\right) dx_{r}:=\lim_{\left\vert
D\right\vert \rightarrow 0,D=\left\{ t_{k}\right\} _{k=0}^{n}\subset \left[
s,t\right] }\sum_{k=0}^{n-1}\int_{t_{k}}^{t_{k+1}}\alpha _{t_{k}}\left(
x_{r}\right) dx_{r}\text{, }\forall 0\leq s<t\leq T\text{,}
\label{definition of integral of time-varying one-form on manifold}
\end{equation}%
exists, and satisfies, for some constant $C_{\theta }>0$,%
\begin{equation}
\left\vert \int_{s}^{t}\left( \alpha _{r}-\alpha _{s}\right) \left(
x_{r}\right) dx_{r}\right\vert \leq C_{\theta }\left\vert t-s\right\vert
^{\theta }\text{, }\forall 0\leq s<t\leq T\text{.}
\label{local estimate one-form on manifold}
\end{equation}
\end{lemma}

For each $u\in \left[ 0,T\right] $, $\alpha _{u}$ is a closed one-form, so
the integral of $\alpha _{u}$ against any continuous path on $M$ is simple
and only depends on end points of the path. The compensated regularity
condition $\left( \ref{condition one-form on manifold}\right) $ between
paths $\alpha $ and $x$ guarantees that we can change from $\alpha _{s}$ to $%
\alpha _{u}$ at $x_{u}$ without a big error. Then, as in the case of Young
integral \cite{young1936inequality}, we can sequentially remove partition
points and the integral exists with the local estimate $\left( \ref{local
estimate one-form on manifold}\right) $. The idea is the same if we try to
integrate a slow-varying cocyclic one-form against a group-valued path,
because cocyclic one-forms are closed one-forms on a group.

Suppose $g$ is a continuous path on $\left[ 0,T\right] $ taking values in
group $\mathcal{G}$, and $\beta $ is a time-varying cocyclic one-form (or
say, a continuous path taking values in cocyclic one-forms on $\mathcal{G}$).

\begin{definition}[Integral]
\label{Definition of integral Definition}Let $g\in C\left( \left[ 0,T\right]
,\mathcal{G}\right) $ and $\beta \in C\left( \left[ 0,T\right] ,B\left( 
\mathcal{G},\mathcal{H}\right) \right) $. If the limit exists 
\begin{equation}
\lim_{\left\vert D\right\vert \rightarrow 0,D=\left\{ t_{k}\right\}
_{k=0}^{n}\subset \left[ 0,T\right] }\beta _{0}\left(
g_{0},g_{0,t_{1}}\right) \beta _{t_{1}}\left(
g_{t_{1}},g_{t_{1},t_{2}}\right) \cdots \beta _{t_{n-1}}\left(
g_{t_{n-1}},g_{t_{n-1},T}\right) \text{, with }g_{s,t}:=g_{s}^{-1}g_{t}\text{%
,}  \label{Definition of integral}
\end{equation}%
then we define the limit to be the integral $\int_{0}^{T}\beta _{u}\left(
g_{u}\right) dg_{u}$.
\end{definition}

For each $u\in \left[ 0,T\right] $, $\beta _{u}$ is a cocyclic one-form. The
integral of $\beta _{u}$ against any continuous path $g$ only depends on the
end points, and satisfies $\int_{s}^{t}\beta _{u}\left( g_{r}\right)
dg_{r}=\beta _{u}\left( g_{s},g_{s,t}\right) $, $\forall s<t$. Hence, $%
\left( \ref{Definition of integral}\right) $ can be rewritten as%
\begin{equation*}
\lim_{\left\vert D\right\vert \rightarrow 0,D=\left\{ t_{k}\right\}
_{k=0}^{n}\subset \left[ 0,T\right] }\tint\nolimits_{0}^{t_{1}}\beta
_{0}\left( g_{u}\right) dg_{u}\tint\nolimits_{t_{1}}^{t_{2}}\beta
_{t_{1}}\left( g_{u}\right) dg_{u}\cdots \tint\nolimits_{t_{n-1}}^{T}\beta
_{t_{n-1}}\left( g_{u}\right) dg_{u}\text{,}
\end{equation*}%
which is similar to $\left( \ref{definition of integral of time-varying
one-form on manifold}\right) $. The proof of the existence of the integral
is also similar. We will assume a generalized Young condition (Condition \ref%
{Condition integrable condition}), and get a local estimate in the form of $%
\left( \ref{local estimate one-form on manifold}\right) $ (Theorem \ref%
{Theorem integrating slow varying cocyclic one forms}). Since cocyclic
one-forms take values in another Banach algebra, if we integrate a
time-varying cocyclic one-form, we would need some graded structure of the
target Banach algebra for the integral to make sense (specified in Section %
\ref{Section Algebraic formulation}, see \cite{feyel2008non} for a general
condition for the limit to exist) and these assumptions set out the basis on
which dominated paths are defined. In particular, we assume that the Banach
algebra coincides with $%
\mathbb{R}
\oplus \mathcal{V}\oplus \cdots \oplus \mathcal{V}^{\otimes n}$ as a Banach
space and that the multiplication in the Banach algebra is induced by the
comultiplication of a family of graded projective mappings. Since a Banach
space can be canonically embedded in a graded Banach algebra, our
formulation includes Banach-space valued one-forms as special cases (e.g.
dominated paths as in Definition \ref{Definition dominated path}). For
polynomial cocyclic one-forms, we can vary it with time to incorporate
Lipschitz one-forms as in \cite{lyons1998differential} (Corollary \ref%
{Example rough integral}) and also\ incorporate time-varying Lipschitz
one-forms (Remark \ref{Remark integrating time-varying Lipschitz functions}%
). In particular, the rough integral $\int_{u\in \left[ s,t\right] }\alpha
\left( X_{u},u\right) dX_{u}$ with an inhomogeneous degree of smoothness
assumption on $\alpha $ can be treated as a special example of integrating
time-varying Lipschitz one-forms (Remark \ref{Remark rough integral with
inhomogeneous degree}).

In proving the existence of the integral and in defining the dominated
paths, we would need to compare cocyclic one-forms. Suppose that we want to
switch from one one-form to another at a point (say $a$). Since they are
cocyclic, if they are close at $a$, then they will be close on the whole
group pointwisely (by which we mean that if a sequence of cocyclic one-forms
converge at a specific point then they converge on the whole group
pointwisely). However, if we want to identify a fairly sharp regularity
condition on the one-forms to integrate against a given group-valued path,
then it is reasonable to compare these two one-forms only around $a$,
because the difference between two one-forms will propagate based on the
structure of the group. Take the polynomial one-form as an example. Suppose $%
p$ and $q$ are two degree $\left( n-1\right) $ polynomial one-forms, which
are close at $0$. Then for $l=0,1,\dots ,n-1$,%
\begin{equation}
\left( D^{l}p\right) \left( z\right) -\left( D^{l}q\right) \left( z\right)
=\sum_{k=0}^{n-1-l}\left( \left( D^{l+k}p\right) \left( 0\right) -\left(
D^{l+k}q\right) \left( 0\right) \right) \frac{z^{\otimes k}}{k!}\text{, \ }%
\forall z\in \mathcal{V}\text{.}
\label{polynomial one form inhomogeneous translocation}
\end{equation}%
For $z\in \mathcal{V}$ satisfying $\delta \leq \left\Vert z\right\Vert \leq
\delta ^{-1}$ for some $\delta \in \left( 0,1\right) $, based on the
expression $\left( \ref{polynomial one form inhomogeneous translocation}%
\right) $, we have the estimate $|\hspace{-0.01in}|(D^{l}p)\left( z\right)
-(D^{l}q)\left( z\right) |\hspace{-0.01in}|\lesssim \max_{k=0}^{n-1-l}|%
\hspace{-0.01in}|(D^{l+k}p)\left( 0\right) -(D^{l+k}q)\left( 0\right) |%
\hspace{-0.01in}|$ (rather than $|\hspace{-0.01in}|(D^{l}p)\left( z\right)
-(D^{l}q)\left( z\right) |\hspace{-0.01in}|\lesssim |\hspace{-0.01in}%
|(D^{l}p)\left( 0\right) -(D^{l}q)\left( 0\right) |\hspace{-0.01in}|$). In
the theory of rough paths we use an inhomogeneous distance between two
one-forms to compensate the inhomogeneous norm on the group where the paths
live. Based on $\left( \ref{polynomial one form inhomogeneous translocation}%
\right) $, this inhomogeneous distance will not be preserved if we compare
these two one-forms at a point that is far from our reference point. Hence,
we compare these two one-forms around $a$ as two continuous linear operators
on a graded Banach-space, equipped with an inhomogeneous norm. In fact, the
definition of the norm (homogeneous or inhomogeneous) is not essential in
the construction of integral. As long as the two paths --- one takes values
in one-forms and the other takes values in the group --- satisfy a
compensated regularity condition (Condition \ref{Condition integrable
condition}), the integral is well-defined and is reduced to an analogue of
the classical Young integral.

We introduce cocyclic one-forms as a family of closed one-forms on a
topological group, and recast the integration in the theory of rough paths
as an example of integrating a time-varying cocyclic one-form against a
group-valued path. Under a compensated regularity condition between the
one-form and the path, the integral exists, and the indefinite integral
obtained is another group-valued path. The integral recovers and extends the
theories of integration for geometric rough paths and for branched rough
paths \cite{lyons1998differential, gubinelli2004controlling,
gubinelli2010ramification}. As an application, we prove the extension
theorem that there exists a unique extension of a given group-valued path of
finite $p$-variation, and we represent the extended path as the integral of
a time-varying cocyclic one-form against the original group-valued path.

For a group-valued path, we define dominated paths as a family of
Banach-space valued paths that can be represented as integrals of
time-varying cocyclic one-forms against the given group-valued path. Under
some structural assumptions on the group, the set of dominated paths is both
a linear space and an algebra, has a canonical enhancement to a group-valued
path, is stable under composition with regular functions and satisfies a
transitive property. There are minor differences between dominated paths and
controlled paths as defined in \cite{gubinelli2004controlling,
gubinelli2010ramification}. Some discussion about their relationship can be
found in Section \ref{Subsection definition of dominated path}. In
particular, dominated paths are defined from and determined by one-forms,
and problems about dominated paths can be reformulated in terms of
one-forms. Working with one-forms has the benefit that basic operations
(iterated integration, multiplication, composition, transitivity) are
continuous in the space of one-forms (Section \ref{Section stableness of
dominated paths}). For example, the solution to a rough differential
equation can equally be formulated as a fixed point problem in the space of
integrable one-forms. The enhancement into a group-valued path is a
continuous operation in the space of one-forms (based on $\left( \ref%
{continuity of one-forms iterated integral}\right) \left( \ref{continuity of
one-forms in algebra}\right) $) and integration is a continuous operation
from one-forms to paths (based on $\left( \ref{Lipschitz continuity of the
integral in terms of one-forms}\right) $), so if the one-forms associated
with a sequence of dominated paths converge then their group-valued
enhancements also converge. In \cite{lyons2015theory}, we give an accessible
overview of the one-form approach developed here. As an application, we
extend an argument of Schwartz \cite{schwartz1989convergence} to rough
differential equations, and give a short proof of the global unique
existence and continuity of the solution by using one-forms.

\section{Construction of the integral}

Recall in Definition \ref{Definition of cocyclic one-form} that $\beta \in
C\left( \mathcal{G},L\left( \mathcal{A},\mathcal{B}\right) \right) $ is
called a cocyclic one-form if there exists a topological group $\mathcal{H}$
in $\mathcal{B}$ such that $\beta \left( a,b\right) \in \mathcal{H}$, $%
\forall a,b\in \mathcal{G}$, and 
\begin{equation}
\beta \left( a,bc\right) =\beta \left( a,b\right) \beta \left( ab,c\right) 
\text{, }\forall a,b,c\in \mathcal{G}\text{.}  \label{cocyclic property}
\end{equation}%
Based on $\left( \ref{cocyclic property}\right) $, we have $\beta \left(
a,1_{\mathcal{G}}\right) =1_{\mathcal{H}}$, $\forall a\in \mathcal{G}$, and $%
\beta \left( a,b\right) ^{-1}=\beta \left( ab,b^{-1}\right) $, $\forall
a,b\in \mathcal{G}$. Moreover, the one-form at $a\in \mathcal{G}$ only
depends on the one-form at $1_{\mathcal{G}}$ and the structure of the group:%
\begin{equation*}
\beta \left( a,b\right) =\beta \left( 1_{\mathcal{G}},a\right) ^{-1}\beta
\left( 1_{\mathcal{G}},ab\right) \text{, }\forall a,b\in \mathcal{G}\text{.}
\end{equation*}%
The cocyclic one-form is a purely algebraic object; and the topology is
coming in when we want to vary it with time.

\begin{proposition}
\label{Proposition equivalent definition of cocyclic one-form}Let $\beta \in
C\left( \mathcal{G},L\left( \mathcal{A},\mathcal{B}\right) \right) $. Then $%
\beta $ is a cocyclic one-form if and only if there exist a topological
group $\mathcal{H}$ in $\mathcal{B}$ and $\alpha \in L\left( \mathcal{A},%
\mathcal{B}\right) $ satisfying $\alpha \left( \mathcal{G}\right) \subseteq 
\mathcal{H}$, such that%
\begin{equation*}
\beta \left( a,b\right) =\alpha \left( a\right) ^{-1}\alpha \left( ab\right) 
\text{, }\forall a,b\in \mathcal{G}\text{.}
\end{equation*}
\end{proposition}

\begin{proof}
$\Leftarrow $ is clear. For $\Rightarrow $, set $\alpha \left( b\right)
:=\beta \left( 1_{\mathcal{G}},b\right) $, $\forall b\in \mathcal{G}$. Then $%
\beta \left( a,b\right) =\beta \left( 1_{\mathcal{G}},a\right) ^{-1}\beta
\left( 1_{\mathcal{G}},ab\right) =\alpha \left( a\right) ^{-1}\alpha \left(
ab\right) $, $\forall a,b\in \mathcal{G}$.
\end{proof}

Proposition \ref{Proposition equivalent definition of cocyclic one-form} is
simple, but is useful for constructing a cocyclic one-form.

Recall the definition of the integral in Definition \ref{Definition of
integral Definition}, i.e. for $\beta \in C\left( \left[ 0,T\right] ,B\left( 
\mathcal{G}\right) \right) $ and $g\in C\left( \left[ 0,T\right] ,\mathcal{G}%
\right) $,%
\begin{equation}
\int_{s}^{t}\beta _{u}\left( g_{u}\right) dg_{u}:=\lim_{\left\vert
D\right\vert \rightarrow 0,D=\left\{ t_{k}\right\} _{k=0}^{n}\subset \left[
s,t\right] }\beta _{0}\left( g_{0},g_{0,t_{1}}\right) \beta _{t_{1}}\left(
g_{t_{1}},g_{t_{1},t_{2}}\right) \cdots \beta _{t_{n-1}}\left(
g_{t_{n-1}},g_{t_{n-1},t}\right) \text{, }\forall s<t\text{,}
\label{integral}
\end{equation}%
provided the limit exists. When $\beta \in C\left( \left[ 0,T\right]
,B\left( \mathcal{G}\right) \right) $ is a constant cocyclic one-form, i.e. $%
\beta _{t}\equiv \beta _{0}\in B\left( \mathcal{G}\right) $, we know how to
integrate $\beta $ against $g\in C\left( \left[ 0,T\right] ,\mathcal{G}%
\right) $, because based on the cocyclic property in $\left( \ref{cocyclic
property}\right) $ and the definition of integral in $\left( \ref{integral}%
\right) $,%
\begin{equation*}
\tint\nolimits_{s}^{t}\beta _{u}\left( g_{u}\right)
dg_{u}=\tint\nolimits_{s}^{t}\beta _{0}\left( g_{u}\right) dg_{u}=\beta
_{0}\left( g_{s},g_{s,t}\right) ,\forall s<t\text{.}
\end{equation*}%
Then, when $t\mapsto \beta _{t}$ varies slowly (to be quantified), the
integral of $\beta $ against $g$ should still exist. In that case, the
behavior of $\beta _{t}$ will depend on the behavior of $g_{t}$. In
particular, if $g_{t}\equiv g_{0}$, then for any $\beta :\left[ 0,T\right]
\rightarrow B\left( \mathcal{G},\mathcal{H}\right) $, we have $%
\int_{0}^{T}\beta _{u}\left( g_{u}\right) dg_{u}=1_{\mathcal{H}}$ (based on
the definition of integral and using that $\beta _{s}\left( a,1_{\mathcal{G}%
}\right) =1_{\mathcal{H}}$, $\forall a\in \mathcal{G}$, $\forall s$). More
generally, when $\beta $ and $g$ have compensated regularities, the integral 
$\int \beta \left( g\right) dg$ exists. We give a sufficient condition in
Condition \ref{Condition integrable condition}.

\subsection{Algebraic assumptions}

\label{Section Algebraic formulation}

The algebraic structure formulated in this section will be assumed
throughout the paper. Our structure is similar to that used in \cite%
{lyons1998differential, lyons2002system, lyons2007differential,
gubinelli2004controlling, gubinelli2010ramification}.

Following Def 1.25 \cite{lyons2007differential}, we equip the tensor powers
of $\mathcal{V}$ with admissible norms.

\begin{definition}[Admissible Norms]
\label{Definition admissible norms}Let $\mathcal{V}$ be a Banach space. We
say the tensor powers of $\mathcal{V}$ are equipped with admissible norms,
if the following conditions hold ($Sym\left( k\right) $ denotes the
symmetric group of degree $k$) 
\begin{eqnarray*}
\left\Vert \varrho v\right\Vert &=&\left\Vert v\right\Vert \text{, }\forall
v\in \mathcal{V}^{\otimes k}\text{, }\forall \varrho \in Sym\left( k\right) 
\text{, } \\
\left\Vert u\otimes v\right\Vert &=&\left\Vert u\right\Vert \left\Vert
v\right\Vert \text{, }\forall u\in \mathcal{V}^{\otimes k}\text{, }\forall
v\in \mathcal{V}^{\otimes j}\text{.}
\end{eqnarray*}
\end{definition}

\begin{notation}[Triple $(T^{\left( n\right) }\left( \mathcal{V}\right) ,%
\mathcal{G}_{n},\mathcal{P}_{n})$]
\label{Notation triple} \ \ \ 

$\left( 1\right) $ We assume that $T^{\left( n\right) }\left( \mathcal{V}%
\right) $ is a unital associative Banach algebra, which, as a Banach space,
coincides with $\mathcal{%
\mathbb{R}
\oplus V\oplus \cdots \oplus V}^{\otimes n}$ with the norm ($\pi _{k}$
denotes the projection to $\mathcal{V}^{\otimes k}$)%
\begin{equation*}
\left\Vert v\right\Vert :=\tsum\nolimits_{k=0}^{n}\left\Vert \pi _{k}\left(
v\right) \right\Vert \text{, }\forall v\in T^{\left( n\right) }\left( 
\mathcal{V}\right) \text{.}
\end{equation*}

$\left( 2\right) $ The multiplication on $T^{\left( n\right) }\left( 
\mathcal{V}\right) $ is induced by the comultiplication on a finite set of
graded projective mappings: 
\begin{equation*}
\mathcal{P}_{n}\mathcal{=}\left\{ \sigma |\sigma \in L\left( T^{\left(
n\right) }\left( \mathcal{V}\right) ,\mathcal{V}^{\otimes \left\vert \sigma
\right\vert }\right) ,\left\vert \sigma \right\vert =0,1,\dots ,n\right\} 
\text{.}
\end{equation*}%
More specifically,

$\left( 2.a\right) $ If we denote by $\sigma _{0}$ the projection of $%
T^{\left( n\right) }\left( \mathcal{V}\right) $ to $\mathcal{%
\mathbb{R}
}$, then $\sigma _{0}\in \mathcal{P}_{n}$.

$\left( 2.b\right) $ Each $\sigma \in \mathcal{P}_{n}$ is a continuous
linear mapping satisfying $\sigma \circ \sigma =\sigma $, and $a=\mathcal{%
\sum }_{\sigma \in \mathcal{P}_{n}}\sigma \left( a\right) $ for any $a\in
T^{\left( n\right) }\left( \mathcal{V}\right) $.

$\left( 2.c\right) $ Let $\bigtriangleup :\mathcal{P}_{n}\rightarrow 
\mathcal{P}_{n}\otimes \mathcal{P}_{n}$ denote the comultiplication on $%
\mathcal{P}_{n}$. Then 
\begin{equation*}
\bigtriangleup \sigma _{0}=\sigma _{0}\otimes \sigma _{0}\text{ .}
\end{equation*}%
For any $\sigma \in \mathcal{P}_{n}$, $\left\vert \sigma \right\vert \geq 1$%
, there exist an integer $N\left( \sigma \right) $ and $\left\{ \sigma
^{j,i}|j=1,2,i=1,\dots ,N\left( \sigma \right) \right\} \subseteq \mathcal{P}%
_{n}$, $\left\vert \sigma ^{j,i}\right\vert \geq 1$, $\left\vert \sigma
^{1,i}\right\vert +\left\vert \sigma ^{2,i}\right\vert =\left\vert \sigma
\right\vert $, $\forall i$, such that 
\begin{equation}
\bigtriangleup \sigma =\sigma \otimes \sigma _{0}+\sigma _{0}\otimes \sigma
+\tsum\nolimits_{i=1}^{N\left( \sigma \right) }\sigma ^{1,i}\otimes \sigma
^{2,i}\text{ .}  \label{comultiplication on projection mappings}
\end{equation}%
When $\left\vert \sigma \right\vert =1$, we set $N\left( \sigma \right) =0$
and $\bigtriangleup \sigma =\sigma \otimes \sigma _{0}+\sigma _{0}\otimes
\sigma $.

$\left( 2.d\right) $ The multiplication on $T^{\left( n\right) }\left( 
\mathcal{V}\right) $ is induced by the comultiplication on $\mathcal{P}_{n}$%
, i.e. for any $a,b\in T^{\left( n\right) }\left( \mathcal{V}\right) $, 
\begin{gather}
\sigma _{0}\left( ab\right) =\sigma _{0}\left( a\right) \sigma _{0}\left(
b\right) \text{,}  \notag \\
\sigma \left( ab\right) =\sigma \left( a\right) \sigma _{0}\left( b\right)
+\sigma _{0}\left( a\right) \sigma \left( b\right)
+\tsum\nolimits_{i=1}^{N\left( \sigma \right) }\sigma ^{1,i}\left( a\right)
\otimes \sigma ^{2,i}\left( b\right) \text{, }\sigma \in \mathcal{P}_{n}%
\text{, }\left\vert \sigma \right\vert \geq 1\text{.}
\label{multiplication is induced by comultiplication}
\end{gather}

$\left( 3\right) $ $\mathcal{G}_{n}$ is a closed connected topological group
in $T^{\left( n\right) }\left( \mathcal{V}\right) $ satisfying $\sigma
_{0}\left( a\right) =1$ for any $a\in \mathcal{G}_{n}$.
\end{notation}

We assume that $\mathcal{G}_{n}$ and $T^{\left( n\right) }\left( \mathcal{V}%
\right) $ have consistent unit, multiplication and topology, but $\mathcal{G}%
_{n}$ may be equipped with a different norm.

\begin{notation}[Consistent Triples]
\label{Notation consistent triples}We say that $\{(T^{\left( n\right)
}\left( \mathcal{V}\right) ,\mathcal{G}_{n},\mathcal{P}_{n})\}_{n=0}^{\infty
}$ is a consistent family, if

$\left( 1\right) $ Each $(T^{\left( n\right) }\left( \mathcal{V}\right) ,%
\mathcal{G}_{n},\mathcal{P}_{n})$ is a triple as in Notation \ref{Notation
triple}.

$\left( 2\right) $ For $m\geq n\geq 1$, $\mathcal{P}_{n}=\left\{ \sigma
|\sigma \in \mathcal{P}_{m},\left\vert \sigma \right\vert \leq n\right\} $,
and the mapping $1_{n}:T^{\left( m\right) }\left( \mathcal{V}\right)
\rightarrow T^{\left( n\right) }\left( \mathcal{V}\right) $ defined by $%
1_{n}\left( a\right) =\sum_{\sigma \in \mathcal{P}_{n}}\sigma \left(
a\right) $, $\forall a\in T^{\left( m\right) }\left( \mathcal{V}\right) $,
is an algebra homomorphism, and induces a group homomorphism from $\mathcal{G%
}_{m}$ to $\mathcal{G}_{n}$ satisfying $1_{n}\left( \mathcal{G}_{m}\right) =%
\mathcal{G}_{n}$.
\end{notation}

With $N\left( \sigma \right) $ in $\left( \ref{comultiplication on
projection mappings}\right) $, we denote the the integer%
\begin{equation}
N\left( n\right) :=\left( \#\mathcal{P}_{n}\right) \vee \left(
\max\nolimits_{\sigma \in \mathcal{P}_{n}}N\left( \sigma \right) \right) 
\text{.}  \label{Bounded by N(m)}
\end{equation}

The nilpotent Lie group and Butcher group are the two main examples in this
paper.

\begin{example}[Nilpotent Lie Group]
Let $\mathcal{G}_{n}$ be the step-$n$ nilpotent Lie group over Banach space $%
\mathcal{V}$. Then $\mathcal{P}_{n}\mathcal{\ }$is the set of projective
mappings $\left\{ \pi _{k}\right\} _{k=0}^{n}$ with $\pi _{k}$ denoting the
projection of $T^{\left( n\right) }\left( \mathcal{V}\right) $ to $\mathcal{V%
}^{\otimes k}$ and $\bigtriangleup \pi _{k}=\sum_{j=0}^{k}\pi _{j}\otimes
\pi _{k-j}$, $k=0,1,\dots ,n$. In this case, $a\in \mathcal{G}_{n}$ if and
only if%
\begin{equation*}
\log _{n}\left( a\right) :=\tsum\nolimits_{k=1}^{n}\left( -1\right) ^{k+1}%
\frac{1}{k}\left( a-1\right) ^{k}\in \mathcal{V}\oplus \lbrack \mathcal{V},%
\mathcal{V]}\oplus \lbrack \mathcal{V},\left[ \mathcal{V},\mathcal{V}\right] 
\mathcal{]}\oplus \cdots \oplus \lbrack \mathcal{V},\cdots \lbrack \mathcal{V%
},\mathcal{V}\overset{n}{\overbrace{]\cdots ]}},
\end{equation*}%
where the multiplication in $\left( a-1\right) ^{k}$ is in $T^{\left(
n\right) }\left( \mathcal{V}\right) $ and $[\mathcal{V},\cdots \lbrack 
\mathcal{V},\mathcal{V}\overset{k}{\overbrace{]\cdots ]}}$ is the subspace
of $\mathcal{V}^{\otimes k}$ spanned by $\left[ v_{1},\cdots ,\left[
v_{k-1},v_{k}\right] \right] $, $v_{i}\in \mathcal{V}$, $i=1,\dots ,k\,$,$\ $%
with $\left[ u,v\right] :=u\otimes v-v\otimes u$.
\end{example}

Based on a classical theorem in free Lie algebras (Theorem 3.2 \cite%
{reutenauer2003free}), the nilpotent Lie group can be equivalently
characterized by using the shuffle product (as exploited in \cite%
{lyons2007differential}).

\begin{example}[Butcher Group]
Let $\mathcal{G}_{n}$ be the Butcher group over $\mathcal{%
\mathbb{R}
}^{d}$. Then $\mathcal{P}_{n}=\left\{ \sigma |\left\vert \sigma \right\vert
\leq n\right\} $ is the set of labelled forests of degree less or equal to $%
n $, and $\bigtriangleup $ is the comultiplication in Connes-Kreimer Hopf
algebra. In particular, for a labelled tree $\tau $, 
\begin{equation*}
\bigtriangleup \tau =1\otimes \tau +\tau \otimes
1+\tsum\nolimits_{c}P^{c}\left( \tau \right) \otimes R^{c}\left( \tau
\right) \text{,}
\end{equation*}%
where the sum is over all non-trivial admissible cuts of $\tau $. For a
labelled forest $\tau _{1}\tau _{2}\cdots \tau _{n}$, where $\tau _{i}$ are
labelled trees, we have 
\begin{equation*}
\bigtriangleup \left( \tau _{1}\tau _{2}\cdots \tau _{n}\right)
=\bigtriangleup \tau _{1}\bigtriangleup \tau _{2}\cdots \bigtriangleup \tau
_{n}\text{.}
\end{equation*}%
In this case, $a\in \mathcal{G}_{n}$ if and only if%
\begin{equation}
\left( \sigma _{1}\sigma _{2}\right) \left( a\right) =\sigma _{1}\left(
a\right) \sigma _{2}\left( a\right) \text{, }\forall \sigma _{i}\in \mathcal{%
P}_{n}\text{, }\left\vert \sigma _{1}\right\vert +\left\vert \sigma
_{2}\right\vert \leq n\text{.}
\label{relationship satisfied by elements in Butcher group}
\end{equation}%
For more detailed explanations and concrete examples, please refer to \cite%
{connes1998hopf} and \cite{gubinelli2010ramification}.
\end{example}

When $\mathcal{G}_{n}$ is the Butcher group, for $a\in \mathcal{G}_{n}$ and $%
\sigma \in \mathcal{P}_{n}$, instead of $\sigma \left( a\right) \in (%
\mathcal{%
\mathbb{R}
}^{d})^{\otimes \left\vert \sigma \right\vert }$, we assume $\sigma \left(
a\right) \in 
\mathbb{R}
$, $\sigma \left( a\right) e_{\sigma }\in (\mathcal{%
\mathbb{R}
}^{d})^{\otimes \left\vert \sigma \right\vert }$ and $a=\sum_{\sigma \in 
\mathcal{P}_{n}}\sigma \left( a\right) e_{\sigma }$ (with $e_{\sigma }$
denotes the tensor coordinate corresponding to $\sigma $). In this case, all
the components in $\sigma \left( ab\right) =\sigma \left( a\right) \sigma
_{0}\left( b\right) +\sigma _{0}\left( a\right) \sigma \left( b\right)
+\tsum\nolimits_{i=1}^{N\left( \sigma \right) }\sigma ^{1,i}\left( a\right)
\sigma ^{2,i}\left( b\right) $ and $\left( \sigma _{1}\sigma _{2}\right)
\left( a\right) =\sigma _{1}\left( a\right) \sigma _{2}\left( a\right) $ are
real numbers.

\subsection{Existence of the integral}

\label{Subsection existence of integral}

\begin{definition}[Control]
The function $\omega :0\leq s\leq t\leq T\rightarrow \overline{%
\mathbb{R}
^{+}}$ is called a control if $\omega $ is continuous, vanishes on the
diagonal $\{(s,t)|0\leq s=t\leq T\}$ and is super-additive: $\omega \left(
s,u\right) +\omega \left( u,t\right) \leq \omega \left( s,t\right) $, $%
\forall 0\leq s\leq u\leq t\leq T$.
\end{definition}

Recall the triple $\left( T^{\left( n\right) }\left( \mathcal{V}\right) ,%
\mathcal{G}_{n},\mathcal{P}_{n}\right) $ in Notation \ref{Notation triple}
with the structural constant $N\left( n\right) $ defined in $\left( \ref%
{Bounded by N(m)}\right) $. Assume $\mathcal{B}$ is a Banach algebra and $%
\mathcal{H}$ is a topological group in $\mathcal{B}$. (We do not assume that 
$\mathcal{B}$ and $\mathcal{H}$ satisfy the conditions in Section \ref%
{Section Algebraic formulation}.) Recall the notation of cocyclic one-forms $%
B\left( \mathcal{H},\mathcal{G}_{n}\right) $ as in Definition \ref%
{Definition of cocyclic one-form}. Based on Definition \ref{Definition of
integral Definition}, the integral of $\beta \in C\left( \left[ 0,T\right]
,B\left( \mathcal{H},\mathcal{G}_{n}\right) \right) $ against $g\in C\left( %
\left[ 0,T\right] ,\mathcal{H}\right) $ is defined by%
\begin{equation*}
\tint\nolimits_{0}^{T}\beta _{u}\left( g_{u}\right) dg_{u}:=\lim_{\left\vert
D\right\vert \rightarrow 0,D=\left\{ t_{k}\right\} _{k=0}^{n}\subset \left[
0,T\right] }\beta _{0}\left( g_{0},g_{0,t_{1}}\right) \beta _{t_{1}}\left(
g_{t_{1}},g_{t_{1},t_{2}}\right) \cdots \beta _{t_{n-1}}\left(
g_{t_{n-1}},g_{t_{n-1},T}\right) \text{,}
\end{equation*}%
provided the limit exists.

\begin{condition}[Integrable Condition]
\label{Condition integrable condition}Let $g\in C\left( \left[ 0,T\right] ,%
\mathcal{H}\right) $ and $\beta \in C\left( \left[ 0,T\right] ,B\left( 
\mathcal{H},\mathcal{G}_{n}\right) \right) $. Then $g$ and $\beta $ are said
to satisfy the integrable condition, if there exist $M>0$, control $\omega $
and $\theta >1$ such that 
\begin{gather}
\max_{\sigma \in \mathcal{P}_{n}}\sup_{0\leq s<t\leq T}\left\Vert \sigma
\left( \beta _{s}\left( g_{s},g_{s,t}\right) \right) \right\Vert \leq M\text{%
,}  \label{condition uniformly bounded linear operator} \\
\max_{\sigma \in \mathcal{P}_{n}}\left\Vert \sigma \left( \left( \beta
_{u}-\beta _{s}\right) \left( g_{u},g_{u,t}\right) \right) \right\Vert \leq
\omega \left( s,t\right) ^{\theta }\text{, }\forall 0\leq s<u<t\leq T\text{.}
\label{condition regularity of time-varying constant one-form}
\end{gather}
\end{condition}

The proof of Theorem \ref{Theorem integrating slow varying cocyclic one
forms} is in the spirit of Young \cite{young1936inequality} and Lyons \cite%
{lyons1998differential}. In \cite{feyel2008non} the authors proved the
existence of a unique multiplicative function for every almost
multiplicative function taking values in an associative monoide using dyadic
partitions, giving an estimate that is a generalization of $\left( \ref%
{error between integral and approximation}\right) $.

\begin{theorem}[Existence of Integral]
\label{Theorem integrating slow varying cocyclic one forms}Suppose $g\in
C\left( \left[ 0,T\right] ,\mathcal{H}\right) $ and $\beta \in C\left( \left[
0,T\right] ,B\left( \mathcal{H},\mathcal{G}_{n}\right) \right) $ satisfy
Condition \ref{Condition integrable condition}. Then $\int_{0}^{\cdot }\beta
_{u}\left( g_{u}\right) dg_{u}\in C\left( \left[ 0,T\right] ,\mathcal{G}%
_{n}\right) $ exists, and there exists a constant $C_{n,\theta ,M,\omega
\left( 0,T\right) }$, such that%
\begin{equation}
\max_{\sigma \in \mathcal{P}_{n}}\left\Vert \sigma \left( \int_{s}^{t}\beta
_{u}\left( g_{u}\right) dg_{u}\right) -\sigma \left( \beta _{s}\left(
g_{s},g_{s,t}\right) \right) \right\Vert \leq C_{n,\theta ,M,\omega \left(
0,T\right) }\omega \left( s,t\right) ^{\theta }\text{, }\forall 0\leq
s<t\leq T\text{.}  \label{error between integral and approximation}
\end{equation}
\end{theorem}

\begin{remark}
Condition $\left( \ref{condition regularity of time-varying constant
one-form}\right) $ is a generalized Young's condition \cite%
{young1936inequality}, representing the compensated regularity between $%
\beta $ and $g$.
\end{remark}

\begin{remark}
\label{Remark continuity}The integral $\int \beta \left( g\right) dg$ is
continuous in the norm%
\begin{equation*}
\max_{\sigma \in \mathcal{P}_{n}}\sup_{0\leq s<t\leq T}\left\Vert \sigma
\left( \beta _{s}\left( g_{s},g_{s,t}\right) \right) \right\Vert
+\max_{\sigma \in \mathcal{P}_{n}}\sup_{0\leq s<u<t\leq T}\frac{\left\Vert
\sigma \left( \left( \beta _{u}-\beta _{s}\right) \left(
g_{u},g_{u,t}\right) \right) \right\Vert }{\omega \left( s,t\right) ^{\theta
}}\text{.}
\end{equation*}
\end{remark}

\begin{proof}
We first prove a uniform bound over all finite partitions. Then we prove the
limit exists.

Since $\beta \in C\left( \left[ 0,T\right] ,B\left( \mathcal{H},\mathcal{G}%
_{n}\right) \right) $ and $\sigma _{0}\left( \mathcal{G}_{n}\right) =1$, we
have%
\begin{equation}
\sigma _{0}\left( \beta _{s}\left( a,b\right) \right) =1\text{, }\forall
a,b\in \mathcal{H}\text{, }\forall s\in \left[ 0,T\right] \text{.}
\label{condition zero level is 1}
\end{equation}%
For $\left[ s,t\right] \subseteq \left[ 0,T\right] $ and finite partition $%
D=\left\{ t_{j}\right\} _{j=0}^{l}$ of $\left[ s,t\right] $, i.e. $%
s=t_{0}<t_{1}<\cdots <t_{l}=t$, we denote 
\begin{equation*}
\beta _{t_{j_{1}},t_{j_{2}}}^{D}:=\beta _{t_{j_{1}}}\left(
g_{t_{j_{1}}},g_{t_{j_{1}},t_{j_{1}+1}}\right) \cdots \beta
_{t_{j_{2}-1}}\left( g_{t_{j_{2}-1}},g_{t_{j_{2}-1},t_{j_{2}}}\right) \text{%
, \ }\forall 0\leq j_{1}\leq j_{2}\leq l\text{.}
\end{equation*}

By using mathematical induction, we first prove that%
\begin{equation}
\max_{\sigma \in \mathcal{P}_{n}}\sup_{D,D\subset \left[ s,t\right]
}\left\Vert \sigma \left( \beta _{s,t}^{D}\right) -\sigma \left( \beta
_{s}\left( g_{s},g_{s,t}\right) \right) \right\Vert \leq C_{n,\theta
,M,\omega \left( 0,T\right) }\omega \left( s,t\right) ^{\theta }\text{, }%
\forall 0\leq s<t\leq T\text{.}
\label{inner uniform boundness in convergence}
\end{equation}%
Using $\left( \ref{condition zero level is 1}\right) $, $\left( \ref{inner
uniform boundness in convergence}\right) $ holds for $\sigma _{0}$. Suppose $%
\left( \ref{inner uniform boundness in convergence}\right) $ holds for $%
\{\sigma |\left\vert \sigma \right\vert \leq k,\sigma \in \mathcal{P}_{n}%
\mathcal{\}}$ for some $k=0,1,\dots ,n-1$. Then, using $\left( \ref%
{condition uniformly bounded linear operator}\right) $, we have%
\begin{equation}
M_{k}:=\max_{\left\vert \sigma \right\vert \leq k,\sigma \in \mathcal{P}%
_{n}}\sup_{0\leq s<t\leq T}\sup_{D,D\subset \left[ s,t\right] }\left\Vert
\sigma \left( \beta _{s,t}^{D}\right) \right\Vert \leq C_{n,\theta ,M,\omega
\left( 0,T\right) }\text{.}  \label{inner definition of Mk}
\end{equation}

Fix $\sigma \in \mathcal{P}_{n}$, $\left\vert \sigma \right\vert =k+1$, $%
\left[ s,t\right] \subseteq \left[ 0,T\right] $ and a finite partition $%
D=\left\{ t_{j}\right\} _{j=0}^{l}\subset \left[ s,t\right] $. By using that 
$\beta _{t_{j-1}}$ is a cocyclic one-form, we have 
\begin{equation*}
\beta _{s,t}^{D}-\beta _{s,t}^{D\backslash \left\{ t_{j}\right\} }=\beta
_{s,t_{j}}^{D}\left( \beta _{t_{j}}-\beta _{t_{j-1}}\right) \left(
g_{t_{j}},g_{t_{j},t_{j+1}}\right) \beta _{t_{j+1},t}^{D}\text{ .}
\end{equation*}%
The multiplication on $\mathcal{G}_{n}$ is induced by the comultiplication
on $\mathcal{P}_{n}$. Then with $I_{d}$ denoting the identity function on $%
\mathcal{P}_{n}$, if 
\begin{equation}
\left( \left( \bigtriangleup \otimes I_{d}\right) \circ \bigtriangleup
\right) \sigma =\tsum\nolimits_{i}\sigma ^{1,i}\otimes \sigma ^{2,i}\otimes
\sigma ^{3,i}\text{,}  \label{inner comultiplication of tau}
\end{equation}%
then%
\begin{equation*}
\sigma \left( \beta _{s,t}^{D}\right) -\sigma \left( \beta
_{s,t}^{D\backslash \left\{ t_{j}\right\} }\right) =\sum_{i,\left\vert
\sigma ^{2,i}\right\vert \geq 1}\sigma ^{1,i}\left( \beta
_{s,t_{j}}^{D}\right) \otimes \sigma ^{2,i}\left( \left( \beta
_{t_{j}}-\beta _{t_{j-1}}\right) \left( g_{t_{j}},g_{t_{j},t_{j+1}}\right)
\right) \otimes \sigma ^{3,i}\left( \beta _{t_{j+1},t}^{D}\right) \text{,}
\end{equation*}%
where $\left\vert \sigma ^{2,i}\right\vert \geq 1$ because $\sigma
_{0}((\beta _{t_{j}}-\beta _{t_{j-1}})(g_{t_{j}},g_{t_{j},t_{j+1}}))=0$.
Since $\left\vert \sigma ^{1,i}\right\vert +\left\vert \sigma
^{2,i}\right\vert +\left\vert \sigma ^{3,i}\right\vert =\left\vert \sigma
\right\vert $, $\forall i$, and $\left\vert \sigma ^{2,i}\right\vert \geq 1$%
, we have $\left\vert \sigma ^{1,i}\right\vert \vee \left\vert \sigma
^{3,i}\right\vert \leq \left\vert \sigma \right\vert -1=k$. Using the
definition of $N\left( n\right) $ in $\left( \ref{Bounded by N(m)}\right) $
and the inductive hypothesis $\left( \ref{inner definition of Mk}\right) $,
we have, for $\left\vert \sigma \right\vert =k+1$, 
\begin{eqnarray*}
\left\Vert \sigma \left( \beta _{s,t}^{D}\right) -\sigma \left( \beta
_{s,t}^{D\backslash \left\{ t_{j}\right\} }\right) \right\Vert &\leq
&C_{n}M_{k}^{2}\max_{\left\vert \sigma ^{\prime }\right\vert \leq k+1,\sigma
^{\prime }\in \mathcal{P}_{n}}\left\Vert \sigma ^{\prime }\left( \left(
\beta _{t_{j}}-\beta _{t_{j-1}}\right) \left(
g_{t_{j}},g_{t_{j},t_{j+1}}\right) \right) \right\Vert \\
&\leq &C_{n,\theta ,M,\omega \left( 0,T\right) }\max_{\left\vert \sigma
^{\prime }\right\vert \leq k+1,\sigma ^{\prime }\in \mathcal{P}%
_{n}}\left\Vert \sigma ^{\prime }\left( \left( \beta _{t_{j}}-\beta
_{t_{j-1}}\right) \left( g_{t_{j}},g_{t_{j},t_{j+1}}\right) \right)
\right\Vert \text{.}
\end{eqnarray*}%
Then, combined with $\left( \ref{condition regularity of time-varying
constant one-form}\right) $, we have%
\begin{equation*}
\left\Vert \sigma \left( \beta _{s,t}^{D}\right) -\sigma \left( \beta
_{s,t}^{D\backslash \left\{ t_{j}\right\} }\right) \right\Vert \leq
C_{n,\theta ,M,\omega \left( 0,T\right) }\omega \left(
t_{j-1},t_{j+1}\right) ^{\theta }\text{.}
\end{equation*}%
For finite partition $D=\left\{ t_{j}\right\} _{k=0}^{l}$ of $\left[ s,t%
\right] $, (as in Theorem 1.16 \cite{lyons2007differential}) we select an
interval $\left[ t_{j-1},t_{j+1}\right] $ that satisfies 
\begin{equation}
\omega \left( t_{j-1},t_{j+1}\right) \leq \frac{2}{l-1}\omega \left(
s,t\right) \text{.}  \label{inner tj to be removed}
\end{equation}%
By recursively removing partition point $t_{j}$ that satisfies $\left( \ref%
{inner tj to be removed}\right) $ (removing the middle point when $l=2$), we
have%
\begin{equation*}
\left\Vert \sigma \left( \beta _{s,t}^{D}\right) -\sigma \left( \beta
_{s}\left( g_{s},g_{s,t}\right) \right) \right\Vert \leq 2^{\theta }\zeta
\left( \theta \right) C_{n,\theta ,M,\omega \left( 0,T\right) }\omega \left(
s,t\right) ^{\theta }=C_{n,\theta ,M,\omega \left( 0,T\right) }\omega \left(
s,t\right) ^{\theta }\text{.}
\end{equation*}%
Hence, $\left( \ref{inner uniform boundness in convergence}\right) $ holds
for $\sigma \in \mathcal{P}_{n}$, $\left\vert \sigma \right\vert =k+1$, and
the induction is complete.

As a consequence, we have 
\begin{equation}
M_{n}:=\max_{\sigma \in \mathcal{P}_{n}}\sup_{0\leq s<t\leq
T}\sup_{D,D\subset \left[ s,t\right] }\left\Vert \sigma \left( \beta
_{s,t}^{D}\right) \right\Vert \leq C_{n,\theta ,M,\omega \left( 0,T\right) }%
\text{.}  \label{inner definition of Mm}
\end{equation}

Then we prove the existence of $\lim_{\left\vert D\right\vert \rightarrow
0,D\subset \left[ s,t\right] }\beta _{s,t}^{D}$. If the limit exists, then $%
\left( \ref{error between integral and approximation}\right) $ holds based
on $\left( \ref{inner uniform boundness in convergence}\right) $. Suppose $%
D^{\prime }=\left\{ s_{w}\right\} \subset \left[ s,t\right] \ $is a
refinement of $D=\left\{ t_{j}\right\} _{j=0}^{l}\subset \left[ s,t\right] $
i.e. for $j=0,1,\dots ,l$, there exists $w_{j}$ such that $s_{w_{j}}=t_{j}$.
Similar as above, for $\sigma \in \mathcal{P}_{n}$, using the linearity of $%
\sigma $ and the comultiplication of $\sigma $ as in $\left( \ref{inner
comultiplication of tau}\right) $, we have%
\begin{eqnarray*}
\sigma \left( \beta _{s,t}^{D}\right) -\sigma \left( \beta _{s,t}^{D^{\prime
}}\right) &=&\tsum\nolimits_{j=0}^{l-1}\sigma \left( \beta
_{s,t_{j}}^{D}\left( \beta _{t_{j},t_{j+1}}^{D}-\beta
_{t_{j},t_{j+1}}^{D^{\prime }}\right) \beta _{t_{j+1},t}^{D^{\prime }}\right)
\\
&=&\tsum\nolimits_{j=0}^{l-1}\tsum\nolimits_{i,\left\vert \sigma
^{2,i}\right\vert \geq 1}\sigma ^{1,i}\left( \beta _{s,t_{j}}^{D}\right)
\sigma ^{2,i}\left( \beta _{t_{j},t_{j+1}}^{D}-\beta
_{t_{j},t_{j+1}}^{D^{\prime }}\right) \sigma ^{3,i}\left( \beta
_{t_{j+1},t}^{D^{\prime }}\right) \text{.}
\end{eqnarray*}%
Then, using the definition of $N\left( n\right) $ in $\left( \ref{Bounded by
N(m)}\right) $ and the definition of $M_{n}$ in $\left( \ref{inner
definition of Mm}\right) $, together with the estimate in $\left( \ref{inner
uniform boundness in convergence}\right) $, we have (since $\beta
_{t_{j},t_{j+1}}^{D}=\beta _{t_{j}}\left( g_{t_{j}},g_{t_{j},t_{j+1}}\right) 
$, $\forall j$)%
\begin{eqnarray*}
\left\Vert \sigma \left( \beta _{s,t}^{D}\right) -\sigma \left( \beta
_{s,t}^{D^{\prime }}\right) \right\Vert &\leq
&C_{n}M_{n}^{2}\tsum\nolimits_{j=0}^{l-1}\max_{\sigma ^{\prime }\in \mathcal{%
P}_{n}}\left\Vert \sigma ^{\prime }\left( \beta _{t_{j}}\left(
g_{t_{j}},g_{t_{j},t_{j+1}}\right) \right) -\sigma ^{\prime }\left( \beta
_{t_{j},t_{j+1}}^{D^{\prime }}\right) \right\Vert \\
&\leq &C_{n,\theta ,M,\omega \left( 0,T\right)
}\tsum\nolimits_{j=0}^{l-1}\omega \left( t_{j},t_{j+1}\right) ^{\theta }\leq
C_{n,\theta ,M,\omega \left( 0,T\right) }\sup_{\left\vert v-u\right\vert
\leq \left\vert D\right\vert }\omega \left( u,v\right) ^{\theta -1}\text{.}
\end{eqnarray*}%
For two general partitions $D$ and $\widetilde{D}$ of $\left[ s,t\right] $,
we have 
\begin{eqnarray*}
\left\Vert \sigma \left( \beta _{s,t}^{D}\right) -\sigma \left( \beta
_{s,t}^{\widetilde{D}}\right) \right\Vert &\leq &\left\Vert \sigma \left(
\beta _{s,t}^{D}\right) -\sigma \left( \beta _{s,t}^{D\cup \widetilde{D}%
}\right) \right\Vert +\left\Vert \sigma \left( \beta _{s,t}^{\widetilde{D}%
}\right) -\sigma \left( \beta _{s,t}^{D\cup \widetilde{D}}\right) \right\Vert
\\
&\leq &C_{n,\theta ,M,\omega \left( 0,T\right) }\sup_{\left\vert
v-u\right\vert \leq \left( \left\vert D\right\vert \vee \left\vert 
\widetilde{D}\right\vert \right) }\omega \left( u,v\right) ^{\theta -1}\text{%
.}
\end{eqnarray*}%
Since $\theta >1$, we have that%
\begin{equation*}
\int_{s}^{t}\beta _{u}\left( g_{u}\right) dg_{u}:=\lim_{\left\vert
D\right\vert \rightarrow 0,D\subset \left[ s,t\right] }\beta _{s,t}^{D}\text{
\ exists.}
\end{equation*}
\end{proof}

\subsection{Extension theorem}

As an application of the integration constructed in Section \ref{Subsection
existence of integral}, we prove the extension theorem (as in Theorem 2.2.1 
\cite{lyons1998differential}, Proposition 9 \cite{gubinelli2004controlling}
and Theorem 7.3 \cite{gubinelli2010ramification}). We represent the extended
group-valued path as the integral of a time-varying cocyclic one-form
against the original group-value path.

Recall the consistent family of triples $\{(T^{\left( n\right) }\left( 
\mathcal{V}\right) ,\mathcal{G}_{n},\mathcal{P}_{n})\}_{n=0}^{\infty }$ in
Notation \ref{Notation consistent triples} and that $\sigma _{0}\in \mathcal{%
P}_{n}$ denotes the projection of $T^{\left( n\right) }\left( \mathcal{V}%
\right) $ to $%
\mathbb{R}
$.

\begin{notation}[Group $\mathcal{T}_{n}$]
For integer $n\geq 0$,\ we denote by $\mathcal{T}_{n}$ the closed
topological group in $T^{\left( n\right) }\left( \mathcal{V}\right) $
consisting all $a\in T^{\left( n\right) }\left( \mathcal{V}\right) $
satisfying $\sigma _{0}\left( a\right) =1$.
\end{notation}

For $a\in \mathcal{T}_{n}$, we define $a^{-1}$ as the algebraic polynomial $%
1+\sum_{k=1}^{n}\left( -1\right) ^{k}\left( a-1\right) ^{k}$, where the
multiplication in $\left( a-1\right) ^{k}$ is in the algebra $T^{\left(
n\right) }\left( \mathcal{V}\right) $. Since we assumed that $\mathcal{G}%
_{n} $ is closed and $\sigma _{0}\left( \mathcal{G}_{n}\right) =1$, $%
\mathcal{G}_{n}$ is a closed subgroup of $\mathcal{T}_{n}$.

\begin{notation}[{$\left\Vert g\right\Vert _{p-var,\left[ 0,T\right] }$}]
\label{Definition p-var of group-valued path}We equip $\mathcal{T}_{n}$ (so $%
\mathcal{G}_{n}$) with the norm%
\begin{equation}
\left\vert a\right\vert :=\tsum\nolimits_{\sigma \in \mathcal{P}%
_{n},\left\vert \sigma \right\vert \geq 1}\left\Vert \sigma \left( a\right)
\right\Vert ^{\frac{1}{\left\vert \sigma \right\vert }}\text{, }\forall a\in 
\mathcal{T}_{n}\text{.}  \label{Definition of norm on G}
\end{equation}%
For $p\geq 1$ and $g\in C\left( \left[ 0,T\right] ,\mathcal{T}_{n}\right) $,
we define the $p$-variation of $g$ by 
\begin{equation*}
\left\Vert g\right\Vert _{p-var,\left[ 0,T\right] }:=\sup_{D,D\subset \left[
0,T\right] }\left( \tsum\nolimits_{k,t_{k}\in D}\left\vert
g_{t_{k},t_{k+1}}\right\vert ^{p}\right) ^{\frac{1}{p}}\text{, }%
g_{s,t}:=g_{s}^{-1}g_{t}\text{.}
\end{equation*}
\end{notation}

\begin{notation}[{$C^{p-var}\left( \left[ 0,T\right] ,\mathcal{T}_{n}\right) $%
}]
We denote the set of continuous paths of finite $p$-variation on $\left[ 0,T%
\right] $ taking values in $\mathcal{T}_{n}$ by $C^{p-var}\left( \left[ 0,T%
\right] ,\mathcal{T}_{n}\right) $ (similarly denote $C^{p-var}\left( \left[
0,T\right] ,\mathcal{G}_{n}\right) $).
\end{notation}

Let $m\geq n\geq 1$ be integers. We recall the algebra homomorphism $1_{n}$
from $T^{\left( m\right) }\left( \mathcal{V}\right) $ to $T^{\left( n\right)
}\left( \mathcal{V}\right) $\ in Notation \ref{Notation consistent triples}
defined by $1_{n}\left( a\right) =\sum_{\sigma \in \mathcal{P}_{n}}\sigma
\left( a\right) $, $\forall a\in T^{\left( m\right) }\left( \mathcal{V}%
\right) $. The algebra homomorphism $1_{n}$ induces a group homomorphism
from $\mathcal{G}_{m}$ to $\mathcal{G}_{n}$ that satisfies $1_{n}\left( 
\mathcal{G}_{m}\right) =\mathcal{G}_{n}$. For $p\geq 1$, $\left[ p\right] $
denotes the largest integer that is less or equal to $p$.

Let $g\in C^{p-var}\left( \left[ 0,T\right] ,\mathcal{G}_{\left[ p\right]
}\right) $ and integer $n\geq \left[ p\right] +1$. We prove that the step-$n$
extension of $g$ exists uniquely and can be represented as the integral of a
slow-varying cocyclic one-form against $g$. In general the extended path
lives in $\mathcal{T}_{n}$ but not in $\mathcal{G}_{n}$. To guarantee that
the extended path takes values in $\mathcal{G}_{n}$, we further assume that $%
\mathcal{G}_{n}$ is \textquotedblleft \thinspace large\thinspace
\textquotedblright\ enough to accommodate the extended path. More
specifically, we assume that

\begin{condition}
\label{Condition extended path taking value in group Gn}For any integer $%
n\geq 1$, there exists a constant $C_{n}>0$ such that, for any $a\in 
\mathcal{G}_{n}$ there exists $\widetilde{a}\in \mathcal{G}_{n+1}$
satisfying $1_{n}\left( \widetilde{a}\right) =a$ and $\left\vert \widetilde{a%
}\right\vert \leq C_{n}\left\vert a\right\vert $ (with $\left\vert \cdot
\right\vert $ defined in $\left( \ref{Definition of norm on G}\right) $).
\end{condition}

Suppose $a\in \mathcal{G}_{n}$. When $\mathcal{G}_{n}$ is the step-$n$
nilpotent Lie group, we can let $\widetilde{a}:=\exp _{n+1}\left( \log
_{n}a\right) $ with $\log $ and $\exp $ defined by algebraic series and the
lower index $n$ indicates the level of truncation. When $\mathcal{G}_{n}$ is
the step-$n$ Butcher group with $\mathcal{P}_{n}^{\prime }=\left\{ \tau
|\left\vert \tau \right\vert \leq n\right\} $ denoting the set of labelled
trees of degree less or equal to $n$, we can let $\widetilde{a}%
:=a+\sum_{k=2}^{n+1}\sum_{\tau _{i}\in \mathcal{P}_{n}^{\prime },\left\vert
\tau _{1}\right\vert +\cdots +\left\vert \tau _{k}\right\vert =n+1}\tau
_{1}\left( a\right) \cdots \tau _{k}\left( a\right) e_{\tau _{1}}\otimes
\cdots \otimes e_{\tau _{k}}$ with $e_{\tau _{i}}$ denoting the tensor
coordinate corresponding to $\tau _{i}$. Then it can be checked that,
Condition \ref{Condition extended path taking value in group Gn} holds in
both cases.

\begin{theorem}[Extension]
\label{Example extension}Let $p\geq 1$ be a real number. \textit{For} $g\in
C^{p-var}\left( \left[ 0,T\right] ,\mathcal{T}_{\left[ p\right] }\right) $
and integer $n\geq \left[ p\right] +1$, there exists a unique $g^{n}\in
C^{p-var}\left( \left[ 0,T\right] ,\mathcal{T}_{n}\right) $ satisfying $%
g_{0}^{n}=1\in \mathcal{T}_{n}$ and $1_{\left[ p\right] }\left(
g_{t}^{n}\right) =g_{0,t}$, $\forall t\in \left[ 0,T\right] $. Moreover,
there exists $\beta \in C\left( \left[ 0,T\right] ,B\left( \mathcal{T}_{%
\left[ p\right] },\mathcal{T}_{n}\right) \right) $ such that $\beta $ and $g$
satisfy the integrable condition (Condition \ref{Condition integrable
condition}) and%
\begin{equation*}
g_{t}^{n}=\int_{0}^{t}\beta _{u}\left( g_{u}\right) dg_{u}\text{, }\forall
t\in \left[ 0,T\right] \text{.}
\end{equation*}%
There exists a constant $C_{n,p}$ (which only depends on $n$ and $p$) such
that%
\begin{equation}
\left\Vert g^{n}\right\Vert _{p-var,\left[ s,t\right] }\leq
C_{n,p}\left\Vert g\right\Vert _{p-var,\left[ s,t\right] }\text{, }\forall
0\leq s\leq t\leq T\text{.}  \label{boundedness of p-var}
\end{equation}%
If we further assume that $g$ takes values in $\mathcal{G}_{\left[ p\right]
} $ and Condition \ref{Condition extended path taking value in group Gn}
holds, then $g^{n}$ takes value in $\mathcal{G}_{n}$.
\end{theorem}

\begin{proof}
\textit{Uniqueness}. Suppose $h^{1}$ and $h^{2}$ are two extensions of $g$
in $\mathcal{T}_{n}$ with finite $p$-variation. Then for $\sigma \in 
\mathcal{P}_{\left[ p\right] }$, $\sigma \left( h^{1}\right) =\sigma \left(
h^{2}\right) $. For $\sigma \in \mathcal{P}_{n}\backslash \mathcal{P}_{\left[
p\right] }$, if suppose $\left( \left( \bigtriangleup \otimes I_{d}\right)
\circ \bigtriangleup \right) \sigma =\sum_{i}\sigma ^{1,i}\otimes \sigma
^{2,i}\otimes \sigma ^{3,i}$, then 
\begin{eqnarray}
&&\left\Vert \sigma \left( h_{s,t}^{1}\right) -\sigma \left(
h_{s,t}^{2}\right) \right\Vert  \label{inner extension1} \\
&\leq &\lim_{\left\vert D\right\vert \rightarrow 0,D\subset \left[ s,t\right]
}\sum\nolimits_{k,t_{k}\in D}\sum\nolimits_{i,\left\vert \sigma
^{2,i}\right\vert \geq \left[ p\right] +1}\left\Vert \sigma ^{1,i}\left(
h_{s,t_{k}}^{1}\right) \right\Vert \left\Vert \sigma ^{2,i}\left(
h_{t_{k},t_{k+1}}^{1}\right) -\sigma ^{2,i}\left(
h_{t_{k},t_{k+1}}^{2}\right) \right\Vert \left\Vert \sigma ^{3,i}\left(
h_{t_{k+1},t}^{2}\right) \right\Vert  \notag \\
&\leq &C_{n,p,\left\Vert h^{1}\right\Vert _{p-var,\left[ 0,T\right]
},\left\Vert h^{2}\right\Vert _{p-var,\left[ 0,T\right] }}\lim_{\left\vert
D\right\vert \rightarrow 0,D\subset \left[ s,t\right] }\sum\limits_{k,t_{k}%
\in D}\left( \left\Vert h^{1}\right\Vert _{p-var,\left[ t_{k},t_{k+1}\right]
}^{\left[ p\right] +1}+\left\Vert h^{2}\right\Vert _{p-var,\left[
t_{k},t_{k+1}\right] }^{\left[ p\right] +1}\right) =0\text{.}  \notag
\end{eqnarray}

\textit{Existence.} We prove by mathematical induction. Denote $g^{\left[ p%
\right] }:=g$. For $m=\left[ p\right] ,\dots ,n-1$, we assume $g^{m}\in
C^{p-var}\left( \left[ 0,T\right] ,\mathcal{T}_{m}\right) $, which holds
when $m=\left[ p\right] $, and define $\beta ^{m}\in C\left( \left[ 0,T%
\right] ,B\left( \mathcal{T}_{m},\mathcal{T}_{m+1}\right) \right) $ by 
\begin{equation*}
\beta _{s}^{m}\left( a,b\right) :=\left( 1_{m}\left( \left( g_{s}^{m}\right)
^{-1}a\right) \right) ^{-1}1_{m}\left( \left( g_{s}^{m}\right)
^{-1}ab\right) \text{, }\forall a,b\in \mathcal{T}_{m}\text{, }\forall s\in %
\left[ 0,T\right] \text{,}
\end{equation*}%
where we used the implicit identification of $\mathcal{T}_{m}$ as a subset
of $\mathcal{T}_{m+1}$, and all operations are in $\mathcal{T}_{m+1}$ except
the multiplication between $a$ and $b$ is in $\mathcal{T}_{m}$. (That $\beta
_{s}^{m}$ is a cocyclic one-form follows from Proposition \ref{Proposition
equivalent definition of cocyclic one-form}.) For $s<u<t$, we have%
\begin{eqnarray*}
\left( \beta _{u}^{m}-\beta _{s}^{m}\right) \left(
g_{u}^{m},g_{u,t}^{m}\right) &=&1_{m}\left( g_{u,t}^{m}\right) -\left(
1_{m}\left( g_{s,u}^{m}\right) \right) ^{-1}1_{m}\left( g_{s,t}^{m}\right) \\
&=&\left( 1_{m}\left( g_{s,u}^{m}\right) \right) ^{-1}\left( 1_{m}\left(
g_{s,u}^{m}\right) 1_{m}\left( g_{u,t}^{m}\right) -1_{m}\left(
g_{s,t}^{m}\right) \right) \\
&=&\left( 1_{m}\left( g_{s,u}^{m}\right) \right) ^{-1}\left(
\tsum\nolimits_{\sigma _{i}\in \mathcal{P}_{\left[ p\right] },\left\vert
\sigma _{1}\right\vert +\left\vert \sigma _{2}\right\vert =m+1}\sigma
_{1}\left( g_{s,u}^{m}\right) \sigma _{2}\left( g_{u,t}^{m}\right) \right) 
\text{.}
\end{eqnarray*}%
We assumed that $g^{m}$ is of finite $p$-variation. Then since $\beta
_{s}^{m}\left( g_{s}^{m},g_{s,t}^{m}\right) =1_{m}\left( g_{s,t}^{m}\right) $%
, we have%
\begin{gather*}
\max_{\sigma \in \mathcal{P}_{m+1}}\sup_{0\leq s<t\leq T}\left\Vert \sigma
\left( \beta _{s}^{m}\left( g_{s}^{m},g_{s,t}^{m}\right) \right) \right\Vert
\leq 1\vee \left\Vert g^{m}\right\Vert _{p-var,\left[ 0,T\right] }^{m}\text{,%
} \\
\max_{\sigma \in \mathcal{P}_{m+1}}\left\Vert \sigma \left( \left( \beta
_{u}^{m}-\beta _{s}^{m}\right) \left( g_{u}^{m},g_{u,t}^{m}\right) \right)
\right\Vert \leq C_{m+1,\left\Vert g^{m}\right\Vert _{p-var,\left[ 0,T\right]
}}\left\Vert g^{m}\right\Vert _{p-var,\left[ s,t\right] }^{m+1}\text{, }%
\forall s<t\text{.}
\end{gather*}%
Since $m+1\geq \left[ p\right] +1>p$, $\left( \beta ^{m},g^{m}\right) $
satisfies the integrable condition. Then by using that $\sigma \left( \beta
_{s}^{m}\left( g_{s}^{m},g_{s,t}^{m}\right) \right) $ equals $\sigma
(\tint\nolimits_{s}^{t}\beta _{u}^{m}\left( g_{u}^{m}\right) dg_{u}^{m})$
when $\sigma \in \mathcal{P}_{m}$ and $\sigma \left( \beta _{s}^{m}\left(
g_{s}^{m},g_{s,t}^{m}\right) \right) $ equals zero when $\sigma \in \mathcal{%
P}_{m+1}\backslash \mathcal{P}_{m}$, and combined with the estimate of the
integral in Theorem \ref{Theorem integrating slow varying cocyclic one forms}%
, we have%
\begin{eqnarray*}
\max_{\sigma \in \mathcal{P}_{m+1}\backslash \mathcal{P}_{m}}\left\Vert
\sigma \left( \tint\nolimits_{s}^{t}\beta _{u}^{m}\left( g_{u}^{m}\right)
dg_{u}^{m}\right) \right\Vert &=&\max_{\sigma \in \mathcal{P}%
_{m+1}}\left\Vert \sigma \left( \tint\nolimits_{s}^{t}\beta _{u}^{m}\left(
g_{u}^{m}\right) dg_{u}^{m}\right) -\sigma \left( \beta _{s}^{m}\left(
g_{s}^{m},g_{s,t}^{m}\right) \right) \right\Vert \\
&\leq &C_{m+1,p,\left\Vert g^{m}\right\Vert _{p-var,\left[ 0,T\right]
}}\left\Vert g^{m}\right\Vert _{p-var,\left[ s,t\right] }^{m+1}\text{.}
\end{eqnarray*}%
As a result, if we define%
\begin{equation*}
g_{t}^{m+1}:=\int_{0}^{t}\beta _{u}^{m}\left( g_{u}^{m}\right) dg_{u}^{m}%
\text{, }t\in \left[ 0,T\right] \text{,}
\end{equation*}%
then%
\begin{equation*}
\left\Vert g^{m+1}\right\Vert _{p-var,\left[ 0,T\right] }\leq
C_{m+1,p,\left\Vert g^{m}\right\Vert _{p-var,\left[ 0,T\right] }}\left\Vert
g^{m}\right\Vert _{p-var,\left[ 0,T\right] }
\end{equation*}%
which holds inductively for $m=\left[ p\right] ,\dots ,n-1$. Since the
constant could be chosen to be monotone in $m+1$ and $\left\Vert
g^{m}\right\Vert _{p-var,\left[ 0,T\right] }$, we have (with $g^{\left[ p%
\right] }:=g$)%
\begin{equation}
\left\Vert g^{n}\right\Vert _{p-var,\left[ 0,T\right] }\leq
C_{n,p,\left\Vert g\right\Vert _{p-var,\left[ 0,T\right] }}\left\Vert
g\right\Vert _{p-var,\left[ 0,T\right] }<\infty \text{.}
\label{inner extension3}
\end{equation}%
Since $1_{\left[ p\right] }\left( g^{n}\right) =1_{\left[ p\right] }\left(
g^{n-1}\right) =\cdots =g$, $g^{n}$ is an extension of $g$ in $\mathcal{T}%
_{n}$. Combined with the uniqueness, $g^{n}$ is the unique step-$n$
extension of $g$ with finite $p$-variation.

\textit{Representation.} Define $\beta \in C\left( \left[ 0,T\right]
,B\left( \mathcal{T}_{\left[ p\right] },\mathcal{T}_{n}\right) \right) $ by 
\begin{equation*}
\beta _{s}\left( a,b\right) :=\left( 1_{\left[ p\right] }\left(
g_{s}^{-1}a\right) \right) ^{-1}1_{\left[ p\right] }\left(
g_{s}^{-1}ab\right) \text{, }\forall a,b\in \mathcal{T}_{\left[ p\right] }%
\text{, }\forall s\in \left[ 0,T\right] \text{.}
\end{equation*}%
As in the case of $\beta ^{m}$ and $g^{m}$, $\beta $ is integrable against $%
g $, and the integral satisfies%
\begin{equation}
\max_{\sigma \in \mathcal{P}_{n}}\left\Vert \sigma \left(
\tint\nolimits_{s}^{t}\beta _{u}\left( g_{u}\right) dg_{u}\right) -\sigma
\left( g_{s,t}\right) \right\Vert \leq C_{n,p,\left\Vert g\right\Vert
_{p-var,\left[ 0,T\right] }}\left\Vert g\right\Vert _{p-var,\left[ s,t\right]
}^{\left[ p\right] +1}\text{, }\forall 0\leq s\leq t\leq T\text{.}
\label{inner extension2}
\end{equation}%
By using $1_{\left[ p\right] }\left( \int \beta \left( g\right) dg\right)
=1_{\left[ p\right] }\left( g^{n}\right) $, using $\left( \ref{inner
extension2}\right) $ and by following similar argument as in $\left( \ref%
{inner extension1}\right) $, we have%
\begin{equation*}
g_{t}^{n}=\int_{0}^{t}\beta _{u}\left( g_{u}\right) dg_{u},\text{ }\forall
t\in \left[ 0,T\right] .
\end{equation*}

The constant in $\left( \ref{inner extension3}\right) $ can be chosen to be
independent of $\left\Vert g\right\Vert _{p-var,\left[ 0,T\right] }$. For $%
c>0$, denote by $\delta _{c}$ the dilation operator i.e. $\delta
_{c}a=\sum_{\sigma }c^{\left\vert \sigma \right\vert }\sigma \left( a\right) 
$. Without loss of generality, we assume $\left\Vert g\right\Vert _{p-var,%
\left[ 0,T\right] }>0$ and denote $c:=\left\Vert g\right\Vert _{p-var,\left[
0,T\right] }^{-1}$. Then $\left\Vert \delta _{c}g\right\Vert _{p-var,\left[
0,T\right] }=1$, and for any $s<t$, 
\begin{eqnarray*}
c\left\Vert \tint_{0}^{\cdot }\beta _{u}\left( g_{u}\right)
dg_{u}\right\Vert _{p-var,\left[ s,t\right] } &=&\left\Vert \delta
_{c}\left( \tint_{0}^{\cdot }\beta _{u}\left( g_{u}\right) dg_{u}\right)
\right\Vert _{p-var,\left[ s,t\right] }=\left\Vert \tint_{0}^{\cdot }\beta
_{u}\left( \left( \delta _{c}g\right) _{u}\right) d\left( \delta
_{c}g\right) _{u}\right\Vert _{p-var,\left[ s,t\right] } \\
&\leq &C_{p,n}\left\Vert \delta _{c}g\right\Vert _{p-var,\left[ s,t\right]
}=cC_{p,n}\left\Vert g\right\Vert _{p-var,\left[ s,t\right] }\text{,}
\end{eqnarray*}%
where we used the uniqueness of extension because $1_{\left[ p\right]
}\left( \delta _{c}\left( \tint \beta \left( g\right) dg\right) \right)
=\delta _{c}g=1_{\left[ p\right] }\left( \int \beta \left( \delta
_{c}g\right) d\delta _{c}g\right) $.

Then we check that, when $g\in C^{p-var}\left( \left[ 0,T\right] ,\mathcal{G}%
_{\left[ p\right] }\right) $ and Condition \ref{Condition extended path
taking value in group Gn} holds, $\int \beta \left( g\right) dg=g^{n}$ takes
values in $\mathcal{G}_{n}$. For $m=\left[ p\right] ,\dots ,n-1$, suppose
that $g^{m}$ takes values in $\mathcal{G}_{m}$, which holds when $m=\left[ p%
\right] $. Based on Condition \ref{Condition extended path taking value in
group Gn}, there exists a constant $C_{m}>0$ such that for any $0\leq
s<t\leq T$ there exists $a_{m+1}^{s,t}\in \mathcal{G}_{m+1}$ such that 
\begin{equation*}
1_{m}\left( a_{m+1}^{s,t}\right) =g_{s,t}^{m}\text{ and }\tsum\nolimits_{%
\sigma \in \mathcal{P}_{m+1}\backslash \mathcal{P}_{m}}\left\Vert \sigma
\left( a_{m+1}^{s,t}\right) \right\Vert \leq C_{m}\left\Vert
g^{m}\right\Vert _{p-var,\left[ s,t\right] }^{m+1}\text{.}
\end{equation*}%
Then we have (with multiplications in $\mathcal{T}_{m+1}$)%
\begin{eqnarray*}
&&|\hspace{-0.01in}|1_{m}\left( g_{t_{0},t_{1}}^{m}\right) \cdots
1_{m}\left( g_{t_{l-1},t_{l}}^{m}\right)
-a_{m+1}^{t_{0},t_{1}}a_{m+1}^{t_{1},t_{2}}\cdots a_{m+1}^{t_{l-1},t_{l}}|%
\hspace{-0.01in}| \\
&=&\left\Vert \tsum\nolimits_{j=0}^{l-1}\tsum\nolimits_{\sigma \in \mathcal{P%
}_{m+1}\backslash \mathcal{P}_{m}}\sigma \left( a^{t_{j},t_{j+1}}\right)
\right\Vert \leq C_{m}\tsum\nolimits_{j=0}^{l-1}\left\Vert g^{m}\right\Vert
_{p-var,\left[ t_{j},t_{j+1}\right] }^{m+1} \\
&\leq &C_{m}\left\Vert g\right\Vert _{p-var,\left[ 0,T\right]
}^{p}\sup\nolimits_{\left\vert v-u\right\vert \leq \left\vert D\right\vert
}\left\Vert g^{m}\right\Vert _{p-var,\left[ u,v\right] }^{m+1-p}\rightarrow 0%
\text{ as }\left\vert D\right\vert \rightarrow 0\text{ \ (since }m+1\geq %
\left[ p\right] +1>p\text{).}
\end{eqnarray*}%
As a result,%
\begin{eqnarray*}
g_{t}^{m+1} &=&\tint_{0}^{t}\beta _{u}^{m}\left( g_{u}^{m}\right)
dg_{u}^{m}=\lim_{\left\vert D\right\vert \rightarrow 0,D=\left\{
t_{j}\right\} _{j=0}^{l}\subset \left[ 0,t\right] }1_{m}\left(
g_{t_{0},t_{1}}^{m}\right) \cdots 1_{m}\left( g_{t_{l-1},t_{l}}^{m}\right) \\
&=&\lim_{\left\vert D\right\vert \rightarrow 0,D\subset \left[ 0,t\right]
}a_{m+1}^{t_{0},t_{1}}a_{m+1}^{t_{1},t_{2}}\cdots a_{m+1}^{t_{l-1},t_{l}}%
\text{, }\forall t\in \left[ 0,T\right] \text{.}
\end{eqnarray*}%
Since $a_{m+1}^{s,t}\in \mathcal{G}_{m+1}$ and $\mathcal{G}_{m+1}$ is
closed, $g^{m+1}\ $takes values in $\mathcal{G}_{m+1}$.
\end{proof}

\section{Dominated paths}

Let $\mathcal{V}$ be a Banach space and suppose that $(T^{(n)}(\mathcal{V)},%
\mathcal{G}_{n},\mathcal{P}_{n})\ $is a triple as in Notation \ref{Notation
triple}. For Banach spaces $E$ and $F$, let $L(E,F\mathcal{)}$ denote the
set of continuous linear mappings from $E$ to $F$.

\subsection{Structural assumptions on the group}

\label{Section dominated path}

Dominated paths are Banach-space valued paths that can be represented as
integrals of time-varying cocyclic one-forms against a given group-valued
path. We would like the set of dominated paths to be stable under some basic
operations, which imposes some structural conditions on the group.

\begin{condition}
\label{Condition smallest Banach space}$T^{\left( n\right) }\left( \mathcal{V%
}\right) $ is the smallest Banach space that includes $\mathcal{G}_{n}$, in
the sense that, for Banach space $\mathcal{U}$ and $\alpha \in L\left(
T^{\left( n\right) }\left( \mathcal{V}\right) ,\mathcal{U}\right) $, if $%
\alpha \left( g\right) =0$, $\forall g\in \mathcal{G}_{n}$, then $\alpha
\left( v\right) =0$, $\forall v\in \mathcal{V}^{\otimes k}$, $k=0,\dots ,n$.
\end{condition}

\begin{condition}
\label{Condition P is an algebra}For $\sigma _{i}\in \mathcal{P}_{n}$, $%
i=1,\dots ,k$, satisfying $\left\vert \sigma _{1}\right\vert +\left\vert
\sigma _{2}\right\vert +\cdots +\left\vert \sigma _{k}\right\vert \leq n$,
there exists $\sigma _{1}\ast \sigma _{2}\ast \cdots \ast \sigma _{k}\in L(%
\mathcal{V}^{\otimes \left( \left\vert \sigma _{1}\right\vert +\cdots
+\left\vert \sigma _{k}\right\vert \right) },\mathcal{V}^{\otimes \left\vert
\sigma _{1}\right\vert }\otimes \cdots \otimes \mathcal{V}^{\otimes
\left\vert \sigma _{k}\right\vert })$ such that 
\begin{equation}
\left( \sigma _{1}\ast \sigma _{2}\ast \cdots \ast \sigma _{k}\right) \left(
a\right) =\sigma _{1}\left( a\right) \otimes \sigma _{2}\left( a\right)
\otimes \cdots \otimes \sigma _{k}\left( a\right) \text{, }\forall a\in 
\mathcal{G}_{n}\text{.}  \label{linear map sigma1*sigma2}
\end{equation}
\end{condition}

It is always possible to extend the algebra (and group) by adding in
monomials of projective mappings so that Condition \ref{Condition P is an
algebra} holds. The product $\ast $ induces a coproduct on the Banach
algebra $T^{\left( n\right) }\left( \mathcal{V}\right) $, and the algebraic
structure corresponds naturally to a Hopf algebra \cite%
{gubinelli2010ramification, hairer2014theory}.

We assume that $T^{\left( n\right) }\left( \mathcal{V}\right) ^{\otimes 2}$
is another Banach algebra, equipped with an admissible norm (see Definition %
\ref{Definition admissible norms}) and with a multiplication, which is a
continuous bilinear operator from $T^{\left( n\right) }\left( \mathcal{V}%
\right) ^{\otimes 2}\times T^{\left( n\right) }\left( \mathcal{V}\right)
^{\otimes 2}$ to $T^{\left( n\right) }\left( \mathcal{V}\right) ^{\otimes 2}$
satisfying%
\begin{equation*}
\left( a_{1}\otimes b_{1}\right) \left( a_{2}\otimes b_{2}\right) =\left(
a_{1}a_{2}\right) \otimes \left( b_{1}b_{2}\right) ,\forall a_{i},b_{i}\in
T^{\left( n\right) }\left( \mathcal{V}\right) \text{.}
\end{equation*}

\begin{condition}
\label{Condition g satisfies differential equation}There exists a continuous
linear mapping $\mathcal{I}$ from $T^{\left( n\right) }\left( \mathcal{V}%
\right) $ to $T^{\left( n\right) }\left( \mathcal{V}\right) ^{\otimes 2}$
satisfying%
\begin{equation}
\mathcal{I}\left( 1\right) =\mathcal{I}\left( \mathcal{V}\right) =0\text{, }%
\mathcal{I}\left( \mathcal{V}^{\otimes k}\right) \subseteq
\tsum\nolimits_{j_{1}+j_{2}=k,j_{i}\geq 1}\mathcal{V}^{\otimes j_{1}}\otimes 
\mathcal{V}^{\otimes j_{2}}\text{, }k=2,\dots ,n\text{.}
\label{property of I   II}
\end{equation}%
In addition, let $1_{n,2}$ denote the projection of $T^{\left( n\right)
}\left( \mathcal{V}\right) ^{\otimes 2}$ to $\tsum\nolimits_{j_{1}+j_{2}\leq
n,j_{i}\geq 1}\mathcal{V}^{\otimes j_{1}}\otimes \mathcal{V}^{\otimes j_{2}}$%
, then%
\begin{equation}
\mathcal{I}\left( ab\right) =\mathcal{I}\left( a\right) +1_{n,2}\left(
\left( a\otimes a\right) \mathcal{I}\left( b\right) \right) +1_{n,2}\left(
\left( a-1\right) \otimes \left( a\left( b-1\right) \right) \right) ,\text{ }%
\forall a,b\in \mathcal{G}_{n}\text{.}  \label{property of I}
\end{equation}
\end{condition}

For a potential choice of the mapping $\mathcal{I}$, if for any $g\in
C\left( \left[ 0,T\right] ,\mathcal{G}_{n}\right) $, the \textquotedblleft
\thinspace formal\thinspace \textquotedblright\ integral%
\begin{equation*}
1_{n,2}(\tiint\nolimits_{0<u_{1}<u_{2}<T}\delta g_{0,u_{1}}\otimes \delta
g_{0,u_{2}})
\end{equation*}%
is well-defined and can be represented as a universal continuous linear
function of $g_{0,T}$, then define%
\begin{equation}
\mathcal{I}\left( a\right) :=1_{n,2}\left(
\tiint\nolimits_{0<u_{1}<u_{2}<T}\delta \left( g_{0,u_{1}}\right) \otimes
\delta \left( g_{0,u_{2}}\right) \right) \text{, \ }g\in C\left( \left[ 0,T%
\right] ,\mathcal{G}_{n}\right) \text{,\ }g_{0,T}=a\text{, }\forall a\in 
\mathcal{G}_{n}\text{,}  \label{Definition of I}
\end{equation}%
which extends linearly to $T^{\left( n\right) }\left( \mathcal{V}\right) $.
By \textquotedblleft universal\textquotedblright , we mean that $\mathcal{I}$
is independent of the selection of $g$ and independent of $a$. In this
formal definition, \textquotedblleft \thinspace $\delta $\thinspace
\textquotedblright\ is comparable to the differential operator and
\textquotedblleft \thinspace $\int $\thinspace \textquotedblright\ is
comparable to the integral operator. Normally, $\left( \ref{property of I}%
\right) $ follows from $\tiint\nolimits_{s<u_{1}<u_{2}<t}=\tiint%
\nolimits_{s<u_{1}<u_{2}<u}+\tiint\nolimits_{u<u_{1}<u_{2}<t}+%
\int_{s<u_{1}<u}\int_{u<u_{2}<t}$ for $s<u<t$, and $\left( \ref{property of
I II}\right) $ holds if $\delta \left( 1\right) =0$. Yet, both $\left( \ref%
{property of I II}\right) $ and $\left( \ref{property of I}\right) $ have to
be checked rigorously for a specific choice of the group.

The existence of the mapping $\mathcal{I}$ imposes a stronger structural
assumption on the group than that is needed for the rough integration. In
rough integration, knowing how to integrate monomials against the degree-one
monomial on paths space, we know how to integrate sufficiently smooth
one-forms against the path \cite{lyons1998differential,
gubinelli2004controlling, gubinelli2010ramification}. The information needed
for rough integration is encoded in the mapping $\mathcal{I}^{\prime }$
below.

\bigskip

\noindent \textbf{Condition 25'} \ \textit{There exists a continuous linear
mapping }$\mathcal{I}^{\prime }$\textit{\ from }$T^{\left( n\right) }\left( 
\mathcal{V}\right) $\textit{\ to }$T^{\left( n\right) }\left( \mathcal{V}%
\right) ^{\otimes 2}$\textit{\ satisfying}%
\begin{equation*}
\mathcal{I}^{\prime }\left( 1\right) =\mathcal{I}^{\prime }\left( \mathcal{V}%
\right) =0\text{, }\mathcal{I}^{\prime }\left( \mathcal{V}^{\otimes
k}\right) \subseteq \mathcal{V}^{\otimes \left( k-1\right) }\otimes \mathcal{%
V}\text{, }k=2,\dots ,n\text{.}
\end{equation*}%
\textit{In addition, let }$1_{n,2}^{\prime }$\textit{\ denote the projection
of }$T^{\left( n\right) }\left( \mathcal{V}\right) ^{\otimes 2}$\textit{\ to 
}$\tsum\nolimits_{k=2}^{n}\mathcal{V}^{\otimes \left( k-1\right) }\otimes 
\mathcal{V}$\textit{, then}%
\begin{equation*}
\mathcal{I}^{\prime }\left( ab\right) =\mathcal{I}^{\prime }\left( a\right)
+1_{n,2}^{\prime }\left( \left( a\otimes a\right) \mathcal{I}^{\prime
}\left( b\right) \right) +1_{n,2}^{\prime }\left( \left( a-1\right) \otimes
\left( a\left( b-1\right) \right) \right) ,\text{ }\forall a,b\in \mathcal{G}%
_{n}\text{.}
\end{equation*}

\bigskip

The mapping $\mathcal{I}^{\prime }$ represents the formal integral $%
1_{n,2}(\tiint\nolimits_{0<u_{1}<u_{2}<T}\delta g_{0,u_{1}}\otimes \delta
x_{0,u_{2}}^{1})$ with $x^{1}:=\pi _{1}\left( g\right) $ and contains part
of the information of the mapping $\mathcal{I}$. The mapping $\mathcal{I}%
^{\prime }$\ contains the information of how to integrate monomials against
the degree-one monomial on paths space: $\tiint\nolimits_{0<u_{1}<u_{2}<T}%
\delta x_{0,u_{1}}^{k}\otimes \delta x_{0,u_{2}}^{1}$, $x^{k}:=\pi
_{k}\left( g\right) $, and $\mathcal{I}^{\prime }$ is sufficient and
necessary to define rough integration (Lemma 11 \cite{lyons2015theory}). In
the same manner, the mapping $\mathcal{I}$ encodes the integration of a
monomial against another monomial (not only the degree-one monomial): $%
\tiint\nolimits_{0<u_{1}<u_{2}<T}\delta x_{0,u_{1}}^{k}\otimes \delta
x_{0,u_{2}}^{j}$, and is sufficient and necessary to define the iterated
integration for dominated paths (Proposition \ref{Proposition enhancement})
and for controlled paths (Corollary \ref{Example weakly controlled path}).

The mapping $\mathcal{I}$ resp. $\mathcal{I}^{\prime }$ are important
because they identify algebraic properties of the group needed to define
rough integration resp. iterated integration. They can be seen as
counterparts of Chen's identity (that encodes paths evolution) in paths
integration. The mapping $\mathcal{I}$ exists for step-$n$ nilpotent Lie
group $n\geq 1$ and step-$2$ Butcher group; the mapping $\mathcal{I}^{\prime
}$ exists for step-$n$ nilpotent Lie group and step-$n$ Butcher group $n\geq
1$. In \cite{lyons2015theory}, the mapping $\mathcal{I}^{\prime }$ is
employed to define Picard iterations for rough differential equations and
prove the unique existence and continuity of the solution when the driving
path lives in step-$n$ nilpotent Lie group or step-$n$ Butcher group for $%
n\geq 1$.

In Section \ref{Section stableness of dominated paths}, we prove that
dominated paths are stable under $\left( 1\right) $ iterated integration, $%
\left( 2\right) \ $multiplication, $\left( 3\right) \ $composition with
regular functions, and $\left( 4\right) \ $is a transitive property.
Condition \ref{Condition smallest Banach space} is used in all four
properties; Condition \ref{Condition P is an algebra} is used in $\left(
2\right) $ and $\left( 3\right) $; Condition \ref{Condition g satisfies
differential equation} is used in $\left( 1\right) $ and $\left( 4\right) $.

\subsubsection{Example: nilpotent Lie group}

\label{Section structural property nilpotent Lie group}Conditions \ref%
{Condition smallest Banach space}, \ref{Condition P is an algebra} and \ref%
{Condition g satisfies differential equation} hold when $\mathcal{G}_{n}$ is
the step-$n$ nilpotent Lie group over Banach space $\mathcal{V}$.

For Condition \ref{Condition smallest Banach space}, suppose $\alpha \in
L(T^{\left( n\right) }\left( \mathcal{V}\right) ,\mathcal{U)}$ satisfies $%
\alpha \left( g\right) =0$, $\forall g\in \mathcal{G}_{n}$. Then for $%
k=1,\dots ,n$ and $v_{i}\in \mathcal{V}$, $i=1,\dots ,k$, by considering the
finite-dimensional space spanned by $\left\{ v_{i}\right\} _{i=1}^{k}$ and
by applying Poincar\'{e}-Birkhoff-Witt theorem, we have $\alpha \left(
v_{1}\otimes \cdots \otimes v_{k}\right) =0$, $\forall \left\{ v_{i}\right\}
_{i=1}^{k}\subset \mathcal{V}$, which implies $\alpha \left( v\right) =0$, $%
\forall v\in \mathcal{V}^{\otimes k}$ (since $\mathcal{V}^{\otimes k}$ is
the closure of the linear span of $\left\{ v_{1}\otimes \cdots \otimes
v_{k}|v_{i}\in \mathcal{V},i=1,\dots ,k\right\} $ and $\alpha $ is a
continuous linear mapping).

Condition \ref{Condition P is an algebra} is satisfied by using the shuffle
product (p36 \cite{lyons2007differential}). Indeed, for $\left(
k_{1},k_{2},\dots ,k_{l}\right) \in \left\{ 1,\dots ,n\right\} ^{l}$, 
\begin{equation*}
\pi _{k_{1}}\left( a\right) \otimes \pi _{k_{2}}\left( a\right) \otimes
\cdots \otimes \pi _{k_{l}}\left( a\right) =\tsum\nolimits_{\varrho \in 
\text{Shuffles}\left( k_{1},k_{2},\dots ,k_{l}\right) }\rho \left( \pi
_{k_{1}+k_{2}+\cdots +k_{l}}\left( a\right) \right) \text{,}
\end{equation*}%
where $\rho \in $Shuffles$\left( k_{1},k_{2},\dots ,k_{l}\right) $ induces a
continuous linear mapping from $\mathcal{V}^{\otimes \left( k_{1}+\cdots
+k_{l}\right) }$ to $\mathcal{V}^{\otimes k_{1}}\otimes \cdots \otimes 
\mathcal{V}^{\otimes k_{l}}$.

For Condition \ref{Condition g satisfies differential equation}, if we
assume that any $g\in C\left( \left[ 0,T\right] ,\mathcal{G}_{n}\right) $
satisfies the formal differential equations: $\delta \left( \pi _{k}\left(
g_{0,t}\right) \right) =\pi _{k-1}\left( g_{0,t}\right) \otimes \delta x_{t}$%
, $\forall t\in \left[ 0,T\right] $, $k=1,\dots ,n$, with $x:=\pi _{1}\left(
g\right) $, then%
\begin{eqnarray*}
&&1_{n,2}\left( \tiint\nolimits_{0<u_{1}<u_{2}<T}\delta \left(
g_{0,u_{1}}\right) \otimes \delta \left( g_{0,u_{2}}\right) \right) \\
&=&1_{n,2}\left( \tint\nolimits_{0}^{T}\left( g_{0,u}-1\right) \otimes
\delta \left( g_{0,u}\right) \right) =\tsum\nolimits_{k_{1}+k_{2}\leq
n,k_{i}\geq 1}\tint\nolimits_{0}^{T}\pi _{k_{1}}\left( g_{0,u}\right)
\otimes \delta \pi _{k_{2}}\left( g_{0,u}\right) \\
&=&\tsum\nolimits_{k_{1}+k_{2}\leq n,k_{i}\geq 1}\tint\nolimits_{0}^{T}\pi
_{k_{1}}\left( g_{0,u}\right) \otimes \pi _{k_{2}-1}\left( g_{0,u}\right)
\otimes \delta x_{u} \\
&=&\tsum\nolimits_{k_{1}+k_{2}\leq n,k_{i}\geq 1}\tsum\nolimits_{\varrho \in 
\text{Shuffles}\left( k_{1},k_{2}-1\right) }\left( \varrho ,1\right) \left(
\pi _{k_{1}+k_{2}}\left( g_{0,T}\right) \right) \text{.}
\end{eqnarray*}%
For $\varrho \in $Shuffles$\left( k_{1},k_{2}-1\right) $, $\left( \varrho
,1\right) $ denotes the continuous linear mapping from $\mathcal{V}^{\otimes
\left( k_{1}+k_{2}\right) }$ to $\mathcal{V}^{\otimes k_{1}}\otimes \mathcal{%
V}^{\otimes k_{2}}$ induced by $\left( \rho ,1\right) $ which is an element
of the symmetric group of order $\left( k_{1}+k_{2}\right) $ whose first $%
\left( k_{1}+k_{2}-1\right) $ elements coincide with $\rho $ with the last
element unchanged. Then using $\left( \ref{Definition of I}\right) $, we
obtain that%
\begin{equation}
\mathcal{I}\left( v\right) :=\tsum\nolimits_{k_{1}+k_{2}\leq n,k_{i}\geq
1}\tsum\nolimits_{\varrho \in \text{Shuffles}\left( k_{1},k_{2}-1\right)
}\left( \varrho ,1\right) \left( \pi _{k_{1}+k_{2}}\left( v\right) \right) 
\text{, }\forall v\in T^{\left( n\right) }\left( \mathcal{V}\right) \text{,}
\label{definition of I for nilpotent Lie group}
\end{equation}%
which is a universal continuous linear mapping. Using $\left( \varrho
,1\right) \in L(\mathcal{V}^{\otimes \left( k_{1}+k_{2}\right) },\mathcal{V}%
^{\otimes k_{1}}\otimes \mathcal{V}^{\otimes k_{2}})$, $\forall k_{i}\geq 1$%
, $k_{1}+k_{2}\leq n$, the mapping $\mathcal{I}$ satisfies $\left( \ref%
{property of I II}\right) $. Then we check that $\mathcal{I}$ satisfies $%
\left( \ref{property of I}\right) $. Since $\mathcal{G}_{n}$ is a closed
topological group in $T^{(n)}(\mathcal{V})$ and $T^{(n)}(\mathcal{V})$ is
the closure of the linear span of $\left\{ v_{1}\otimes \cdots \otimes
v_{k}|v_{i}\in \mathcal{V},k=1,\dots ,n\right\} $, for any $a,b\in \mathcal{G%
}_{n}$, there exist $v_{i}\in \mathcal{V}$, $i\geq 1$, and $a_{m},b_{m}\in 
\mathcal{G}_{n}(\limfunc{span}(\{v_{i}\}_{i=1}^{m}))$, $m\geq 1$, such that $%
\lim_{m\rightarrow \infty }a_{m}=a$ and $\lim_{m\rightarrow \infty }b_{m}=b$%
. If we prove that $\left( \ref{property of I}\right) $ holds for $a_{m}$
and $b_{m}$ for any $m\geq 1$, then by using continuity we can prove that $%
\left( \ref{property of I}\right) $ holds for $a$ and $b$. For this, fix $%
m\geq 1$, we treat $\{v_{i}\}_{i=1}^{m}$ as a basis of an $m$-dimensional
space. Then by using Chow-Rashevskii connectivity Theorem, there exist two
continuous bounded variation paths $x_{m}$ and $y_{m}$ on $\left[ 0,1\right] 
$ taking values in $\limfunc{span}(\{v_{i}\}_{i=1}^{m})$ such that $%
S_{n}\left( x_{m}\right) _{0,1}=a_{m}$ and $S_{n}\left( y_{m}\right)
_{0,1}=b_{m}$. Then $S_{n}\left( x_{m}\right) _{0,\cdot }$ and $S_{n}\left(
y_{m}\right) _{0,\cdot }$ are two differentiable paths taking values in $%
\mathcal{G}_{n}(\limfunc{span}(\{v_{i}\}_{i=1}^{m}))$ that satisfy%
\begin{eqnarray*}
\mathcal{I}\left( a_{m}\right) &=&1_{n,2}\left(
\tiint\nolimits_{0<u_{1}<u_{2}<1}dS_{n}\left( x_{m}\right) _{0,u_{1}}\otimes
dS_{n}\left( x_{m}\right) _{0,u_{2}}\right) \text{,} \\
\mathcal{I}\left( b_{m}\right) &=&1_{n,2}\left(
\tiint\nolimits_{0<u_{1}<u_{2}<1}dS_{n}\left( y_{m}\right) _{0,u_{1}}\otimes
dS_{n}\left( y_{m}\right) _{0,u_{2}}\right) \text{.}
\end{eqnarray*}%
Let $x_{m}\sqcup y_{m}:\left[ 0,2\right] \rightarrow \limfunc{span}%
(\{v_{i}\}_{i=1}^{m})$ denote the concatenation of $x_{m}$ and $y_{m}$. Then
by using Chen's identity, we have $S_{n}\left( x_{m}\sqcup y_{m}\right)
_{0,2}=a_{m}b_{m}$. By using the definition of the mapping $\mathcal{I}$, we
have 
\begin{eqnarray*}
\mathcal{I}\left( a_{m}b_{m}\right) &=&1_{n,2}\left(
\tiint\nolimits_{0<u_{1}<u_{2}<2}dS_{n}\left( x_{m}\sqcup y_{m}\right)
_{0,u_{1}}\otimes dS_{n}\left( x_{m}\sqcup y_{m}\right) _{0,u_{2}}\right) \\
&=&1_{n,2}\left(
\tiint\nolimits_{0<u_{1}<u_{2}<1}+\tiint\nolimits_{1<u_{1}<u_{2}<2}+\tint%
\nolimits_{0<u_{1}<1}\tint\nolimits_{1<u_{2}<2}\right) \\
&=&\mathcal{I}\left( a_{m}\right) +1_{n,2}\left( \left( a_{m}\otimes
a_{m}\right) \mathcal{I}\left( b_{m}\right) \right) +1_{n,2}\left( \left(
a_{m}-1\right) \otimes \left( a_{m}\left( b_{m}-1\right) \right) \right) 
\text{,}
\end{eqnarray*}%
so $\left( \ref{property of I}\right) $ holds for $a_{m}$ and $b_{m}$.

\subsubsection{Example: Butcher group}

\label{Section structural assumption Butcher group}Conditions \ref{Condition
smallest Banach space} and \ref{Condition P is an algebra} hold when $%
\mathcal{G}_{n}$ is the step-$n$ Butcher group over $\mathcal{%
\mathbb{R}
}^{d}$. Condition \ref{Condition g satisfies differential equation} holds
when $n=2$. For $n\geq 3$, it is hard to construct the mapping $\mathcal{I}$
in Condition \ref{Condition g satisfies differential equation}, but
Condition 25' holds and the mapping $\mathcal{I}^{\prime }$ exists. The
mapping $\mathcal{I}^{\prime }$ encodes the integration of monomials against
the degree-one monomial on paths space, and that is sufficient and necessary
to define the rough integration \cite{lyons1998differential,
gubinelli2004controlling, gubinelli2010ramification}.

Condition \ref{Condition smallest Banach space} holds for similar reasons as
for the nilpotent Lie group.

For labelled forests $\sigma _{i}\in \mathcal{P}_{n}$, $i=1,\dots ,k$,
satisfying $\left\vert \sigma _{1}\right\vert +\cdots +\left\vert \sigma
_{k}\right\vert \leq n$, $\left( \sigma _{1}\cdots \sigma _{k}\right) $ is
again an labelled forest of degree less or equal to $n$, so $\left( \sigma
_{1}\cdots \sigma _{k}\right) \in \mathcal{P}_{n}$ and $\sigma _{1}\left(
a\right) \cdots \sigma _{k}\left( a\right) =(\sigma ^{1}\cdots \sigma
^{k})\left( a\right) $ for any $a\in \mathcal{G}_{n}$ based on $\left( \ref%
{relationship satisfied by elements in Butcher group}\right) $. Hence,
Condition \ref{Condition P is an algebra} holds.

When $n=2$, we can prove that Condition \ref{Condition g satisfies
differential equation} holds. $\mathcal{P}_{2}$ is the set of labelled
forests with degree less or equal to $2$, i.e. $\mathcal{P}_{2}=\{\Lcdot_{i},%
\Lcdot_{i}\Lcdot_{j},\raisebox{0.0012in}{$|$}\hspace{-0.0572in}_{\Lcdot i}^{%
\Lcdot j}\left\vert i,j\right. \in \left\{ 1,\dots ,d\right\} \}$. The
property of $\mathcal{I}$ in $\left( \ref{property of I}\right) $ reduces to
($\{e_{i}\}_{i=1}^{d}$ a basis of $%
\mathbb{R}
^{d}$) 
\begin{equation}
\mathcal{I}\left( ab\right) =\mathcal{I}\left( a\right) +\mathcal{I}\left(
b\right) +\tsum\nolimits_{i,j=1}^{d}\left( \Lcdot_{j}\left( a\right) \right)
e_{j}\otimes \left( \Lcdot_{i}\left( b\right) \right) e_{i}\text{, }\forall
a,b\in \mathcal{G}_{2}\text{.}  \label{Butcher group p<3}
\end{equation}%
Then, let $\mathcal{I}\left( a\right) :=\sum_{i,j=1}^{d}%
\raisebox{0.0012in}{$|$}\hspace{-0.0572in}_{\Lcdot i}^{\Lcdot j}\left(
a\right) e_{j}\otimes e_{i}$, that is a universal continuous linear mapping
from $T^{(2)}(%
\mathbb{R}
^{d})$ to $%
\mathbb{R}
^{d}\otimes 
\mathbb{R}
^{d}$. Then the property $\left( \ref{Butcher group p<3}\right) $ holds
because, based on the rule of multiplication in the Butcher group, $%
\raisebox{0.0012in}{$|$}\hspace{-0.0572in}_{\Lcdot i}^{\Lcdot j}\left(
ab\right) =\raisebox{0.0012in}{$|$}\hspace{-0.0572in}_{\Lcdot i}^{\Lcdot %
j}\left( a\right) +\raisebox{0.0012in}{$|$}\hspace{-0.0572in}_{\Lcdot i}^{%
\Lcdot j}\left( b\right) +\left( \Lcdot_{j}\left( a\right) \right) \left( %
\Lcdot_{i}\left( b\right) \right) $, $\forall a,b\in \mathcal{G}_{2}$.
Equivalently, we could assume that any $g\in C\left( \left[ 0,T\right] ,%
\mathcal{G}_{2}\right) $ satisfies the formal differential equation $\delta (%
\raisebox{0.0012in}{$|$}\hspace{-0.0572in}_{\Lcdot i}^{\Lcdot j}\left(
g_{0,t}\right) )=(\Lcdot_{j}\left( g_{0,t}\right) )\delta x_{t}^{i}$ with $%
x_{t}^{i}:=\Lcdot_{i}\left( g_{0,t}\right) $. Then, for $a\in \mathcal{G}%
_{2} $ and $g\in C\left( \left[ 0,T\right] ,\mathcal{G}_{2}\right) $
satisfying $g_{0,T}=a$, we have%
\begin{eqnarray}
\mathcal{I}\left( a\right) &=&1_{n,2}\left(
\tiint\nolimits_{0<u_{1}<u_{2}<T}\delta \left( g_{0,u_{1}}\right) \otimes
\delta \left( g_{0,u_{2}}\right) \right)
\label{definition of the mapping I for Butcher group 1} \\
&=&\tsum\nolimits_{i,j=1}^{d}\tint\nolimits_{0}^{T}\left( \Lcdot_{j}\left(
g_{0,u}\right) \right) \delta x_{t}^{i}\,e_{j}\otimes
e_{i}=\tsum\nolimits_{i,j=1}^{d}\raisebox{0.0012in}{$|$}\hspace{-0.0572in}_{%
\Lcdot i}^{\Lcdot j}\left( g_{0,T}\right) e_{j}\otimes
e_{i}=\tsum\nolimits_{i,j=1}^{d}\raisebox{0.0012in}{$|$}\hspace{-0.0572in}_{%
\Lcdot i}^{\Lcdot j}\left( a\right) e_{j}\otimes e_{i}\text{.}  \notag
\end{eqnarray}%
The existence of $\mathcal{I}$ when $n=2$ gives an explanation to the
existence of the canonical enhancement of a controlled path when $2\leq p<3$
(see Theorem $1$ \cite{gubinelli2004controlling} and Corollary \ref{Example
weakly controlled path} below).

For $n\geq 3$, the problem is complicated and finding a mapping $\mathcal{I}$
satisfying Condition \ref{Condition g satisfies differential equation} is
difficult. Indeed, for $g\in C\left( \left[ 0,T\right] ,\mathcal{G}%
_{3}\right) $, it is hard to represent%
\begin{equation*}
\tint\nolimits_{0}^{T}\left( \Lcdot_{k}\left( g_{0,t}\right) \right) \delta
\left( (\Lcdot_{i}\Lcdot_{j})\left( g_{0,t}\right) \right) \text{, }\left(
i,j,k\right) \in \left\{ 1,\dots ,d\right\} ^{3}\text{,}
\end{equation*}%
as a linear functional of $g_{0,T}$, because the integration by parts
formula does not hold. For $n\geq 3$, Condition \ref{Condition g satisfies
differential equation}' holds. If we assume (as in Theorem 8.5 \cite%
{gubinelli2010ramification}) that any $g\in C\left( \left[ 0,T\right] ,%
\mathcal{G}_{n}\right) $ satisfies the formal differential equations that,
for any labelled forest $\sigma \in \mathcal{P}_{n-1}$, 
\begin{equation*}
\delta \left[ \sigma \right] _{i}\left( g_{0,t}\right) =\sigma \left(
g_{0,t}\right) \delta x_{t}^{i}\text{, \ }\forall t\in \left[ 0,T\right] 
\text{, with }x_{t}^{i}:=\Lcdot_{i}\left( g_{0,t}\right) \text{,}
\end{equation*}%
where $[\sigma ]_{i}$ denotes the labelled tree obtained by attaching $%
\sigma $ to a new root with label $i$. Then, we can define $\mathcal{I}%
^{\prime }\in L(T^{(n)}(%
\mathbb{R}
^{d}),T^{(n)}(%
\mathbb{R}
^{d})^{\otimes 2})$ by, for any $a\in \mathcal{G}_{n}$ and $g\in C\left( %
\left[ 0,T\right] ,\mathcal{G}_{n}\right) $ satisfying $g_{0,T}=a$, 
\begin{eqnarray}
\mathcal{I}^{\prime }\left( a\right) &:&=1_{n,2}\left(
\tint\nolimits_{0}^{T}\left( g_{0,u}-1\right) \otimes \delta x_{u}\right) 
\notag \\
&=&\tsum\nolimits_{i=1}^{d}\tsum\nolimits_{\sigma \in \mathcal{P}%
_{n},|\sigma |=1,\dots ,n-1}\tint\nolimits_{0}^{T}\sigma \left(
g_{0,u}\right) \delta x_{u}^{i}\,e_{\sigma }\otimes e_{i}  \notag \\
&=&\tsum\nolimits_{i=1}^{d}\tsum\nolimits_{\sigma \in \mathcal{P}%
_{n},|\sigma |=1,\dots ,n-1}\left[ \sigma \right] _{i}\left( a\right)
e_{\sigma }\otimes e_{i}\text{,}
\label{definition of the mapping I' for Butcher group 2}
\end{eqnarray}%
where $e_{\sigma }\in (\mathcal{%
\mathbb{R}
}^{d})^{\otimes |\sigma |}$ is the basis coordinate corresponding to $\sigma 
$ and $e_{\sigma }\otimes e_{i}$ is treated as an element in $(\mathcal{%
\mathbb{R}
}^{d})^{\otimes |\sigma |}\otimes \mathcal{%
\mathbb{R}
}^{d}\subset T^{(n)}(%
\mathbb{R}
^{d})^{\otimes 2}$. Hence, $\mathcal{I}^{\prime }$ is a universal continuous
linear mapping from $T^{(n)}(%
\mathbb{R}
^{d})$ to $T^{(n)}(%
\mathbb{R}
^{d})^{\otimes 2}$ that satisfies $\mathcal{I}^{\prime }\left( 1\right) =%
\mathcal{I}^{\prime }\left( 
\mathbb{R}
^{d}\right) =0$, $\mathcal{I}^{\prime }\left( (%
\mathbb{R}
^{d})^{\otimes k}\right) \subseteq (%
\mathbb{R}
^{d})^{\otimes \left( k-1\right) }\otimes 
\mathbb{R}
^{d}$, $k=2,\dots ,n$. Moreover, by using the multiplication in Butcher
group, we have, for $a,b\in \mathcal{G}_{n}$ and $\sigma \in \mathcal{P}%
_{n-1}$ satisfying $\bigtriangleup \sigma =\tsum\nolimits_{j}\sigma
^{1,j}\otimes \sigma ^{2,j}$,%
\begin{equation*}
\left[ \sigma \right] _{i}\left( ab\right) =\left[ \sigma \right] _{i}\left(
a\right) +\tsum\nolimits_{j}\sigma ^{1,j}\left( a\right) \left[ \sigma ^{2,j}%
\right] _{i}\left( b\right) \text{, }\forall a,b\in \mathcal{G}_{n}\text{.}
\end{equation*}%
This implies that,%
\begin{eqnarray}
&&\tsum\nolimits_{\sigma \in \mathcal{P}_{n},|\sigma |=0,\dots ,n-1}\left[
\sigma \right] _{i}\left( ab\right) e_{\sigma }\otimes e_{i}
\label{property of I' for Butcher group} \\
&=&\tsum\nolimits_{\sigma \in \mathcal{P}_{n},|\sigma |=0,\dots ,n-1}\left[
\sigma \right] _{i}\left( a\right) e_{\sigma }\otimes e_{i}+\left( a\otimes
1\right) \left( \tsum\nolimits_{\sigma \in \mathcal{P}_{n},|\sigma |=0,\dots
,n-1}\left[ \sigma \right] _{i}\left( b\right) e_{\sigma }\otimes
e_{i}\right) \text{,}  \notag
\end{eqnarray}%
where $e_{\sigma }\otimes e_{i}$ is treated as an element in $(\mathcal{%
\mathbb{R}
}^{d})^{\otimes |\sigma |}\otimes \mathcal{%
\mathbb{R}
}^{d}\subset T^{(n)}(%
\mathbb{R}
^{d})^{\otimes 2}$. Hence, if we let $1_{n,2}^{\prime }$ denote the
projection of $T^{(n)}(%
\mathbb{R}
^{d})^{\otimes 2}$ to $\tsum\nolimits_{k=2}^{n}(%
\mathbb{R}
^{d})^{\otimes \left( k-1\right) }\otimes 
\mathbb{R}
^{d}$, then based on $\left( \ref{definition of the mapping I' for Butcher
group 2}\right) $ and $\left( \ref{property of I' for Butcher group}\right) $
we have%
\begin{eqnarray*}
\mathcal{I}^{\prime }\left( ab\right) &=&\mathcal{I}^{\prime }\left(
a\right) +1_{n,2}^{\prime }\left( \left( a\otimes 1\right) \mathcal{I}%
^{\prime }\left( b\right) \right) +1_{n,2}^{\prime }\left( \left( a-1\right)
\otimes \left( \Lcdot\left( b\right) \right) \right) \\
&=&\mathcal{I}^{\prime }\left( a\right) +1_{n,2}^{\prime }\left( \left(
a\otimes a\right) \mathcal{I}^{\prime }\left( b\right) \right)
+1_{n,2}^{\prime }\left( \left( a-1\right) \otimes \left( a\left( b-1\right)
\right) \right) \text{, }\forall a,b\in \mathcal{G}_{n}\text{,}
\end{eqnarray*}%
where the term $1_{n,2}^{\prime }\left( \left( a-1\right) \otimes \left(
a\left( b-1\right) \right) \right) $ is caused by the different ranges of
summation of $\sigma \in \mathcal{P}_{n}$ in $\left( \ref{definition of the
mapping I' for Butcher group 2}\right) $ and $\left( \ref{property of I' for
Butcher group}\right) $.

It is hard to find a mapping $\mathcal{I}$ for step-$n$ Butcher group $n\geq
3$, because it is hard to differentiate a path taking values in the group.
Unlike the nilpotent Lie group where any group-valued path $g$ satisfies the
`formal' differential equation $\delta g=g\delta x$ with $x:=\pi _{1}\left(
g\right) $, in this case it is hard to find a continuous mapping $F$ such
that any path taking values in step-$n$ Butcher group $n\geq 3$ satisfies $%
\delta g=F\left( g\right) \delta x$. The mapping $\mathcal{I}$ exists for
step-$2$ Butcher group because we only need to define $\int x\delta x$ and
to integrate against the degree-one monomial. For step-$n$ Butcher group $%
n\geq 3$, to define the mapping $\mathcal{I}$, we need to define the
differentiation of products that may easily drift out of the group as in the
case of It\^{o} differential equations.

\subsection{Definition of dominated paths}

\begin{notation}
\label{Notation operator norm}Let $\mathcal{U}$ be a Banach space and $%
\alpha \in L\left( T^{\left( n\right) }\left( \mathcal{V}\right) ,\mathcal{U}%
\right) $. We denote 
\begin{eqnarray*}
\left\Vert \alpha \left( \cdot \right) \right\Vert &:&=\sup_{v\in T^{\left(
n\right) }\left( \mathcal{V}\right) ,\left\Vert v\right\Vert =1}\left\Vert
a\left( v\right) \right\Vert \text{,} \\
\left\Vert \alpha \left( \cdot \right) \right\Vert _{k} &:&=\sup_{v\in 
\mathcal{V}^{\otimes k},\left\Vert v\right\Vert =1}\left\Vert a\left(
v\right) \right\Vert \text{, }k=0,1,\dots ,n\text{.}
\end{eqnarray*}
\end{notation}

With $\left[ p\right] $ we denote the largest integer that is less or equal
to $p\geq 1$. We work with the triple $(T^{(\left[ p\right] )}(\mathcal{V)},%
\mathcal{G}_{\left[ p\right] },\mathcal{P}_{\left[ p\right] })$ and
continuous paths of finite $p$-variation taking values in $\mathcal{G}_{%
\left[ p\right] }$ i.e. $C^{p-var}\left( \left[ 0,T\right] ,\mathcal{G}_{%
\left[ p\right] }\right) $.

\begin{condition}[Slowly-Varying Condition]
\label{Condition of integrable beta}Suppose that $(T^{(\left[ p\right] )}(%
\mathcal{V)},\mathcal{G}_{\left[ p\right] },\mathcal{P}_{\left[ p\right] })$
satisfies Conditions \ref{Condition smallest Banach space}, \ref{Condition P
is an algebra} and \ref{Condition g satisfies differential equation}, $g\in
C^{p-var}\left( \left[ 0,T\right] ,\mathcal{G}_{\left[ p\right] }\right) $
and $\mathcal{U}$ is a Banach space.\ We say $\beta \in C\left( \left[ 0,T%
\right] ,B\left( \mathcal{G}_{\left[ p\right] },\mathcal{U}\right) \right) $
is slowly-varying, if there exist $M>0$, control $\omega $ and $\theta >1$
such that 
\begin{gather*}
\left\Vert \beta _{t}\left( g_{t},\cdot \right) \right\Vert \leq M\text{, }%
\forall t\in \left[ 0,T\right] \text{,} \\
\left\Vert (\beta _{t}-\beta _{s})\left( g_{t},\cdot \right) \right\Vert
_{k}\leq \omega \left( s,t\right) ^{\theta -\frac{k}{p}}\text{, }\forall
0\leq s<t\leq T\text{, }k=1,2,\dots ,\left[ p\right] \text{.}
\end{gather*}%
We define the operator norm of $\beta $ by 
\begin{equation}
\left\Vert \beta \right\Vert _{\theta }^{\omega }:=\sup_{t\in \left[ 0,T%
\right] }\left\Vert \beta _{t}\left( g_{t},\cdot \right) \right\Vert
+\max_{k=1,\dots ,\left[ p\right] }\sup_{0\leq s\leq t\leq T}\frac{%
\left\Vert \left( \beta _{t}-\beta _{s}\right) \left( g_{t},\cdot \right)
\right\Vert }{\omega \left( s,t\right) ^{\theta -\frac{k}{p}}}\text{.}
\label{Definition of operator norm}
\end{equation}
\end{condition}

The norm $\left\Vert \cdot \right\Vert _{\theta }^{\omega }$ is used in \cite%
{lyons2015theory} to quantify the convergence of one-forms associated with
Picard iterations for rough differential equations.

If $\beta $ satisfies the slowly-varying condition for $g$, then $\left(
\beta ,g\right) $ satisfies the integrable condition (Condition \ref%
{Condition integrable condition}). Indeed, for $s<u<t$,%
\begin{eqnarray*}
\left\Vert \left( \beta _{u}-\beta _{s}\right) \left( g_{u},g_{u,t}\right)
\right\Vert &\leq &\tsum\nolimits_{\sigma \in \mathcal{P}_{\left[ p\right]
}}\left\Vert (\beta _{u}-\beta _{s})\left( g_{u},\sigma \left(
g_{u,t}\right) \right) \right\Vert \\
&\leq &\tsum\nolimits_{\sigma \in \mathcal{P}_{\left[ p\right] }}\left\Vert
(\beta _{u}-\beta _{s})\left( g_{u},\cdot \right) \right\Vert _{\left\vert
\sigma \right\vert }\left\Vert \sigma \left( g_{u,t}\right) \right\Vert \\
&\leq &\tsum\nolimits_{\sigma \in \mathcal{P}_{\left[ p\right] }}\omega
\left( s,u\right) ^{\theta -\frac{\left\vert \sigma \right\vert }{p}%
}\left\Vert g\right\Vert _{p-var,\left[ u,t\right] }^{\left\vert \sigma
\right\vert }\leq C_{p}(\omega \left( s,t\right) +\left\Vert g\right\Vert
_{p-var,\left[ s,t\right] }^{p})^{\theta }\text{.}
\end{eqnarray*}

\begin{definition}[Dominated Paths]
\label{Definition dominated path}Suppose $g\in C^{p-var}\left( \left[ 0,T%
\right] ,\mathcal{G}_{\left[ p\right] }\right) $, and for a Banach space $%
\mathcal{U}$, $h\in C\left( \left[ 0,T\right] ,\mathcal{U}\right) $. If
there exists $\beta \in C\left( \left[ 0,T\right] ,B\left( \mathcal{G}_{%
\left[ p\right] },\mathcal{U}\right) \right) $ that is slowly-varying and
satisfies%
\begin{equation}
h_{t}=h_{0}+\int_{0}^{t}\beta _{u}\left( g_{u}\right) dg_{u}\text{, }\forall
t\in \left[ 0,T\right] \text{.}  \label{definition of dominated path}
\end{equation}%
Then we call $h$ a path dominated by $g$ with the one-form\ $\beta $, and
refer to the process of starting with the path $h$ and fixing a choice of
the one-form $\beta $ as coupling $h$ to $g$ via $\beta $ to make a
dominated path.
\end{definition}

Using Theorem \ref{Theorem integrating slow varying cocyclic one forms}, for
control $\hat{\omega}:=\omega +\left\Vert g\right\Vert _{p-var}^{p}$ and $%
\theta >1$, we have%
\begin{equation}
\left\Vert h_{t}-h_{s}-\beta _{s}\left( g_{s},g_{s,t}\right) \right\Vert
\leq C_{p,\theta ,\hat{\omega}\left( 0,T\right) }\left\Vert \beta
\right\Vert _{\theta }^{\omega }\hat{\omega}\left( s,t\right) ^{\theta }%
\text{, }\forall s<t\text{,}
\label{estimate of integral in terms of one-forms}
\end{equation}%
and the function $\left( \beta ,\left\Vert \cdot \right\Vert _{\theta
}^{\omega }\right) \mapsto (h,\left\Vert \cdot \right\Vert _{p-var,\left[ 0,T%
\right] })$ is a Lipschitz function:%
\begin{equation}
\left\Vert h\right\Vert _{p-var,\left[ 0,T\right] }\leq C_{p,\theta ,\hat{%
\omega}\left( 0,T\right) }\left\Vert \beta \right\Vert _{\theta }^{\omega }%
\text{.}  \label{Lipschitz continuity of the integral in terms of one-forms}
\end{equation}

\subsection{Weakly controlled paths and dominated paths}

\label{Subsection definition of dominated path}

Dominated paths bear some similarities to controlled paths \cite%
{gubinelli2004controlling, gubinelli2010ramification}. We quote the
definition of $\kappa $-weakly controlled paths in Def 8.1 \cite%
{gubinelli2010ramification}:

\begin{definition}[Weakly Controlled Paths, Gubinelli]
\label{Definition weakly controlled path Gubinelli}Let $X$ be a $\gamma $%
-BRP and let $n$ be the largest integer such that $n\gamma \leq 1$. For $%
\kappa \in (1/\left( n+1\right) ,\gamma ]$ a path $y$ is a $\kappa $-weakly
controlled by $X$ with values in $V$ if there exist paths $\{y^{\tau }\in 
\mathcal{C}_{1}^{\left\vert \tau \right\vert \kappa }\left( V\right) :\tau
\in \mathcal{F}_{\mathcal{L}}^{n-1}\}$ and remainders $\{y^{\#}\in \mathcal{C%
}_{2}^{n\kappa }\left( V\right) ,y^{\#,\tau }\in \mathcal{C}_{2}^{\left(
n-\left\vert \tau \right\vert \right) \kappa }\left( V\right) ,\tau \in 
\mathcal{F}_{\mathcal{L}}^{n-1}\}$ such that%
\begin{equation}
\delta y=\tsum\nolimits_{\tau \in \mathcal{F}_{\mathcal{L}}^{n-1}}X^{\tau
}y^{\tau }+y^{\#}  \label{Gubinelli weakly controlled path 1}
\end{equation}%
and for $\tau \in \mathcal{F}_{\mathcal{L}}^{n-1}$:%
\begin{equation}
\delta y^{\tau }=\tsum\nolimits_{\sigma \in \mathcal{F}_{\mathcal{L}%
}^{n-1}}\tsum\nolimits_{\rho }c^{\prime }\left( \sigma ,\tau ,\rho \right)
X^{\rho }y^{\sigma }+y^{\tau ,\#}  \label{Gubinelli weakly controlled path 2}
\end{equation}%
where we mean $\delta y^{\tau }=y^{\tau ,\#}$ when $\left\vert \tau
\right\vert =n-1$.
\end{definition}

Note that in a convenient abuse of language, although it is the coupling $y^{\tau }$ of $y$ to $X$ %
that is the weakly controlled path, it is customary to write as if the symbol $y$ alone was the weakly controlled path! It is a good
cautionary excercise to give examples of a non zero coupling of the zero path to $X$.
  
Translated to our language, Definition \ref{Definition weakly controlled
path Gubinelli} can be rewritten as follows. For $p\geq 1$, suppose $%
\mathcal{G}_{\left[ p\right] }$ is the step-$\left[ p\right] $ Butcher group
over $%
\mathbb{R}
^{d}$, $\mathcal{P}_{\left[ p\right] -1}$ denotes the set of labelled
forests of degree less or equal to $\left[ p\right] -1$ and $g\in
C^{p-var}\left( \left[ 0,T\right] ,\mathcal{G}_{\left[ p\right] }\right) $.
Then for Banach space $\mathcal{U}$, $\gamma \in C\left( \left[ 0,T\right] ,%
\mathcal{U}\right) $ is a path controlled by $g$, if there exist a family of
paths $\gamma ^{\sigma }\in C(\left[ 0,T\right] ,\mathcal{U))}$ indexed by $%
\sigma \in \mathcal{P}_{\left[ p\right] -1}$, $\left\vert \sigma \right\vert
\geq 1$, and constants $\theta >1$, $C>0$, such that, $\gamma $ satisfies 
\begin{equation}
\left\Vert \gamma _{t}-\gamma _{s}-\tsum\nolimits_{\sigma \in \mathcal{P}_{%
\left[ p\right] -1},\left\vert \sigma \right\vert \geq 1}\gamma _{s}^{\sigma
}\sigma \left( g_{s,t}\right) \right\Vert \leq C(\left\Vert g\right\Vert
_{p-var,\left[ s,t\right] }^{p})^{\theta -\frac{1}{p}}\text{, }\forall 0\leq
s<t\leq T\text{,}  \label{condition increment of weakly controlled path}
\end{equation}%
and $\gamma ^{\sigma }$, $\sigma \in \mathcal{P}_{\left[ p\right] -1}$, $%
\left\vert \sigma \right\vert \geq 1$, satisfies 
\begin{equation}
\left\Vert \gamma _{t}^{\sigma }-\gamma _{s}^{\sigma
}-\tsum\nolimits_{\sigma _{i}\in \mathcal{P}_{\left[ p\right] -1},\left\vert
\sigma _{i}\right\vert \geq 1}c^{\prime }\left( \sigma _{1},\sigma
_{2},\sigma \right) \gamma _{s}^{\sigma _{1}}\sigma _{2}\left(
g_{s,t}\right) \right\Vert \leq C(\left\Vert g\right\Vert _{p-var,\left[ s,t%
\right] }^{p})^{\theta -\frac{1+\left\vert \sigma \right\vert }{p}}\text{, }%
\forall 0\leq s<t\leq T\text{, }
\label{condition time-varying of weakly controlled path}
\end{equation}%
where $c^{\prime }\left( \sigma _{1},\sigma _{2},\sigma \right) $ counts the
number of $\sigma _{2}\otimes \sigma $ in the reduced comultiplication $%
\bigtriangleup ^{\prime }\sigma _{1}=\bigtriangleup \sigma _{1}-\sigma
_{0}\otimes \sigma _{1}-\sigma _{1}\otimes \sigma _{0}$ of $\sigma _{1}$
(with $\sigma _{0}$ denoting the projection to $%
\mathbb{R}
$).

We can redefine controlled paths by using time-varying cocyclic one-forms.

\begin{definition}[Weakly Controlled Paths]
\label{Definition weakly controlled path}Let $g\in C^{p-var}\left( \left[ 0,T%
\right] ,\mathcal{G}_{\left[ p\right] }\right) $ and let $\mathcal{U}$ be a
Banach space. We say that $\gamma :\left[ 0,T\right] \rightarrow \mathcal{U}$
is a path weakly controlled by $g$, if there exist control $\omega $ and $%
\beta :\left[ 0,T\right] \rightarrow B\left( \mathcal{G}_{\left[ p\right]
-1},\mathcal{U}\right) $ satisfying%
\begin{gather}
\left\Vert \gamma _{t}-\gamma _{s}-\beta _{s}\left( g_{s},g_{s,t}\right)
\right\Vert \leq \omega \left( s,t\right) ^{\theta -\frac{1}{p}}\text{, }%
\forall 0\leq s<t\leq T\text{,}  \label{condition weakly controlled 1} \\
\left\Vert \left( \beta _{t}-\beta _{s}\right) \left( g_{t},\cdot \right)
\right\Vert _{k}\leq \omega \left( s,t\right) ^{\theta -\frac{1+k}{p}}\text{%
, }\forall 0\leq s<t\leq T\text{, }k=1,\dots ,\left[ p\right] -1\text{.}
\label{condition weakly controlled 2}
\end{gather}
\end{definition}

If $\gamma $ is a controlled path in the sense of Definition \ref{Definition
weakly controlled path Gubinelli} then $\gamma $ is a controlled path in the
sense of Definition \ref{Definition weakly controlled path}. Indeed, we can
rewrite $\left( \ref{condition increment of weakly controlled path}\right) $
and $\left( \ref{condition time-varying of weakly controlled path}\right) $
in term of time-varying cocyclic one-forms. Define $\beta \in C\left( \left[
0,T\right] ,B\left( \mathcal{G}_{\left[ p\right] -1},\mathcal{U}\right)
\right) \,$by 
\begin{equation*}
\beta _{s}\left( a,b\right) :=\tsum\nolimits_{\sigma \in \mathcal{P}_{\left[
p\right] -1},\left\vert \sigma \right\vert \geq 1}\gamma _{s}^{\sigma
}\sigma \left( g_{s}^{-1}a\left( b-\sigma _{0}\left( b\right) \right)
\right) \text{, }\forall a,b\in \mathcal{G}_{\left[ p\right] -1}\text{, }%
\forall 0\leq s\leq T\text{.}
\end{equation*}%
Then $\left( \ref{condition increment of weakly controlled path}\right)
\Leftrightarrow \left( \ref{condition weakly controlled 1}\right) $, and $%
\left( \ref{condition time-varying of weakly controlled path}\right)
\Rightarrow \left( \ref{condition weakly controlled 2}\right) $. Indeed, $%
\left( \ref{condition time-varying of weakly controlled path}\right) $
implies $\left( \ref{condition weakly controlled 2}\right) $, because for
any $a\in \mathcal{G}_{\left[ p\right] -1}$ we have 
\begin{align}
\left( \beta _{t}-\beta _{s}\right) \left( g_{t},a\right) &
=\tsum\nolimits_{\sigma \in \mathcal{P}_{\left[ p\right] -1},\left\vert
\sigma \right\vert \geq 1}\gamma _{t}^{\sigma }\sigma \left( a-\sigma
_{0}\left( a\right) \right) -\tsum\nolimits_{\sigma \in \mathcal{P}_{\left[ p%
\right] -1},\left\vert \sigma \right\vert \geq 1}\gamma _{s}^{\sigma }\sigma
\left( g_{s,t}\left( a-\sigma _{0}\left( a\right) \right) \right)  \notag \\
& =\tsum\nolimits_{\sigma \in \mathcal{P}_{\left[ p\right] -1},\left\vert
\sigma \right\vert \geq 1}\left( \gamma _{t}^{\sigma }-\gamma _{s}^{\sigma
}-\tsum\nolimits_{\sigma _{i}\in \mathcal{P}_{\left[ p\right] -1},\left\vert
\sigma _{i}\right\vert \geq 1}c^{\prime }\left( \sigma _{1},\sigma
_{2},\sigma \right) \gamma _{s}^{\sigma _{1}}\sigma _{2}\left(
g_{s,t}\right) \right) \sigma \left( a\right) \text{,}
\label{inner equivalency of condition 2}
\end{align}%
where the constant $c^{\prime }\left( \sigma _{1},\sigma _{2},\sigma \right) 
$ is defined as in $\left( \ref{condition time-varying of weakly controlled
path}\right) $. Since both sides of $\left( \ref{inner equivalency of
condition 2}\right) $ are linear in $a$, based on Condition \ref{Condition
smallest Banach space} (we proved that it holds for Butcher group), $\left( %
\ref{inner equivalency of condition 2}\right) $ holds for any $v\in (%
\mathcal{%
\mathbb{R}
}^{d}\mathcal{)}^{\otimes k}$, $k=1,\dots ,\left[ p\right] -1$. Hence, 
\begin{eqnarray*}
\left\Vert \left( \beta _{t}-\beta _{s}\right) \left( g_{t},\cdot \right)
\right\Vert _{k} &\leq &\tsum\nolimits_{\sigma \in \mathcal{P}_{\left[ p%
\right] -1},\left\vert \sigma \right\vert =k}\left\Vert \left( \gamma
_{t}^{\sigma }-\gamma _{s}^{\sigma }-\tsum\nolimits_{\sigma _{i}\in \mathcal{%
P}_{\left[ p\right] -1},\left\vert \sigma _{i}\right\vert \geq 1}c^{\prime
}\left( \sigma _{1},\sigma _{2},\sigma \right) \gamma _{s}^{\sigma
_{1}}\sigma _{2}\left( g_{s,t}\right) \right) \sigma \left( \cdot \right)
\right\Vert _{k} \\
&\leq &C(\left\Vert g\right\Vert _{p-var,\left[ s,t\right] }^{p})^{\theta -%
\frac{1+k}{p}}\text{, }\forall 0\leq s<t\leq T\text{, }k=1,\dots ,\left[ p%
\right] -1\text{.}
\end{eqnarray*}

The space of controlled paths is a linear space and is preserved under
composition with regular functions. Moreover, when $2\leq p<3$, for paths $%
\gamma ^{1}$ and $\gamma ^{2}$ controlled by $g\in C^{p-var}\left( \left[ 0,T%
\right] ,\mathcal{G}_{\left[ p\right] }\right) $, the integral path $%
\int_{0}^{\cdot }\gamma _{u}^{1}\otimes d\gamma _{u}^{2}$ is canonically
defined and is again a path controlled by $g$ (Theorem $1$ \cite%
{gubinelli2004controlling}). When $p\geq 3$, for path $\gamma $ controlled
by $g\in C^{p-var}\left( \left[ 0,T\right] ,\mathcal{G}_{\left[ p\right]
}\right) $ with $x$ denoting the first level of $g$, the integral path $%
\int_{0}^{\cdot }\gamma _{u}\otimes dx_{u}$ is well defined and is again a
path controlled by $g$ (Theorem 8.5 \cite{gubinelli2010ramification},
Definition 4.17 \cite{friz2014course}). The existence of the canonical
integral of controlled paths when $2\leq p<3$ resp. $p\geq 3$ is closely
related to the mappings $\mathcal{I}$ in $\left( \ref{definition of the
mapping I for Butcher group 1}\right) $ resp. $\mathcal{I}^{\prime }$ in $%
\left( \ref{definition of the mapping I' for Butcher group 2}\right) $ (see
also Remark \ref{Remark integrating a weakly controlled path get a dominated
path} and Corollary \ref{Example weakly controlled path} below).

Comparing with the definition of dominated paths, we have

\begin{proposition}
If $\gamma $ is a path dominated by $g$, then $\gamma $ is a path weakly
controlled by $g$ as in Definition \ref{Definition weakly controlled path}.
\end{proposition}

The other direction is not necessarily true, and the one-form associated
with a controlled path can vary a little quicker than the one-form
associated with a dominated path. Roughly speaking, the relationship between
the controlled path and the dominated path is comparable to the relationship
between the integrand and the indefinite integral got after integration or
to the difference between a weak and a strong solution to a stochastic 
differential equation. The slowly varying 
one form in the coupling realises $\gamma$ as a function of $X$. %

In \cite{lyons2015theory} we proved that, when the mapping $\mathcal{I}^{\prime%
}$ exists (e.g. step-$n$ nilpotent Lie group or step-$n$ Butcher group $%
n\geq 1$), for controlled path $\gamma $ and $x:=\pi _{1}\left( g\right) $,
the indefinite integral $\int_{0}^{\cdot }\gamma _{u}\otimes dx_{u}$ is
well-defined as a dominated path. More generally, when the mapping $\mathcal{%
I}$ exists (e.g. step-$n$ nilpotent Lie group $n\geq 1$ or step-$2$ Butcher
group), for controlled path $\gamma ^{1}$ and dominated path $\gamma ^{2}$,
the integral path $\int_{0}^{\cdot }\gamma _{u}^{1}\otimes d\gamma _{u}^{2}$
is well defined and is another dominated path (Remark \ref{Remark
integrating a weakly controlled path get a dominated path}).

In the definition of dominated paths e.g. $h_{\cdot }=h_{0}+\int_{0}^{\cdot
}\beta _{u}\left( g_{u}\right) dg_{u}$ in $\left( \ref{definition of
dominated path}\right) $, $\beta $ is integrable against $g$ and $h$ is
determined by $\beta $. Indeed, dominated paths are all about integrable
one-forms, and the path is determined by the one-form. On the other hand,
based on $\left( \ref{condition weakly controlled 1}\right) $ and $\left( %
\ref{condition weakly controlled 2}\right) $, for a controlled path $\gamma $%
, $\beta $ does not necessarily satisfy the integrable condition, and $%
\gamma $ is not uniquely determined by $\beta $. Indeed, for $\beta $
satisfying $\left( \ref{condition weakly controlled 2}\right) $, there does
not necessarily exist a path $\gamma $ satisfies $\left( \ref{condition
weakly controlled 1}\right) $; if there exists a $\gamma $ satisfies $\left( %
\ref{condition weakly controlled 1}\right) $, then $\gamma +\eta $ also
satisfies $\left( \ref{condition weakly controlled 1}\right) $ for any $\eta
:\left[ 0,T\right] \rightarrow \mathcal{U}$ satisfying $\left\Vert \eta
_{t}-\eta _{s}\right\Vert \leq C|\hspace{-0.01in}|g|\hspace{-0.01in}|_{p-var,%
\left[ s,t\right] }^{\theta p-1}$, $\forall 0\leq s<t\leq T$. That the
time-varying one-form is not sufficiently integrable and that the path is
not uniquely determined by the one-form will always be there for a
controlled path, which make the existence of the canonical iterated integral
(when $2\leq p<3$ \cite{gubinelli2004controlling}) of two controlled paths a
very interesting result. In fact the existence of the iterated integral is
not solely about the one-form; it is the result of the interplay between the
one-form and the path via the intermediary of integration (Section \ref%
{Subsection existence of integral}). Indeed, based on Corollary \ref{Example
weakly controlled path}, for paths $\gamma ^{i}$, $i=1,2$, controlled by $%
g\in C^{p-var}\left( \left[ 0,T\right] ,\mathcal{G}_{\left[ p\right]
}\right) $, $2\leq p<3$, the path $\int_{0}^{\cdot }\gamma _{u}^{1}\otimes
d\gamma _{u}^{2}$ can be represented as the integral of a time-varying
cocyclic one-form against the group-valued path $\gamma ^{2}\oplus g$. As a
result, for a path $\gamma :\left[ 0,T\right] \rightarrow 
\mathbb{R}
^{e}$ controlled by $g\in C^{p-var}\left( \left[ 0,T\right] ,\mathcal{G}%
_{[p]}\right) $, $2\leq p<3$, there exists a canonical enhancement of $%
\gamma $ to a path taking values in the step-$2$ Butcher group over $%
\mathbb{R}
^{e}$:%
\begin{gather*}
\Gamma _{t}:=1+\tsum\nolimits_{\sigma \in \mathcal{P}_{2}}x\left( \sigma
\right) _{t}e_{\sigma }\text{ ,} \\
\text{with }x\left( \Lcdot_{i}\right) _{t}:=\gamma _{t}^{i}-\gamma _{0}^{i}%
\text{, \ }x\left( \Lcdot_{i}\Lcdot_{j}\right) _{t}:=x\left( \Lcdot%
_{i}\right) _{t}x\left( \Lcdot_{j}\right) _{t}\text{, }x\left( %
\raisebox{0.0012in}{$|$}\hspace{-0.0572in}_{\Lcdot i}^{\Lcdot j}\right)
_{t}:=\tint\nolimits_{0}^{t}x\left( \Lcdot_{j}\right) _{u}d\gamma _{u}^{i}%
\text{, }\forall \left( i,j\right) \in \left\{ 1,\dots ,e\right\}
^{2},\forall t\text{,}
\end{gather*}%
where $e_{\sigma }\in (%
\mathbb{R}
^{e})^{\otimes \left\vert \sigma \right\vert }$ is the basis coordinate
corresponding to the labelled forest $\sigma $. The set of paths dominated
by $\Gamma $ clearly includes $\gamma $. When $\gamma $ is dominated by $g$,
the set of paths dominated by $\Gamma $ is a subset of the paths dominated
by $g$ (Proposition \ref{Proposition Transitivity}). Intuitively, one could
split the space of controlled paths to subspaces of dominated paths
(dominated by a slightly perturbed group-valued path). Each subspace is a
linear space and an algebra, stable under iterated integration and
composition with regular functions. It is also possible to union finitely
many of these subspaces, which will be dominated by the joint signature of
these controlled paths. The canonical enhancement of a controlled path is
well-defined when the group is the step-$n$ nilpotent Lie group for $n\geq 1$
or the step-$2$ Butcher group. While it should be noted that for step-$n$
Butcher group, $n\geq 3$, it is difficult to construct the mapping $\mathcal{%
I}$ in Condition \ref{Condition g satisfies differential equation}, so it is
difficult to define the enhancement of a controlled path in that case.

Working with dominated paths has the benefit that basic operations are 
\textit{continuous} in operator norm in the space of one-forms (Section \ref%
{Section stableness of dominated paths}). For example, the one-form
associated with the group-valued enhancement of a dominated path (built by
using operations in Section \ref{Section stableness of dominated paths}, can
take values in e.g. nilpotent Lie group or Butcher group) is continuous in
operator norm w.r.t. the one-form associated with the base dominated path.

\section{Stability of dominated paths}

\label{Section stableness of dominated paths}

The set of dominated paths is a linear space, stable under iterated
integration, is an algebra and is stable under composition with regular
functions. Moreover, being a dominated path is a transitive property: if $%
\gamma $ is dominated by $g$ and we enhance $\gamma $ via iterated
integration to a group-valued path $\Gamma $, then those paths dominated by $%
\Gamma $ form a subset of the paths dominated by $g$. All of these
operations are continuous in associated one-forms, and explicit dependence
in operator norm is given.

Coefficients in this section may depend on the norm of the mappings $\sigma
_{1}\ast \cdots \ast \sigma _{k}$ (as in Condition \ref{Condition P is an
algebra}) and the norm of the mapping $\mathcal{I}$ (as in Condition \ref%
{Condition g satisfies differential equation}).

Let $\mathcal{U}$ be a Banach space and $\beta \in B\left( \mathcal{G}_{%
\left[ p\right] },\mathcal{U}\right) $. Using that $\beta \left( a,b\left(
c-1\right) \right) =\beta \left( ab,c\right) =\beta \left( ab,c-1\right) $, $%
\forall a,b,c\in \mathcal{G}_{\left[ p\right] }$, the linearity in $%
c-1=c-\sigma _{0}\left( c\right) $ and that $%
\mathbb{R}
\oplus \mathcal{V}\oplus \cdots \oplus \mathcal{V}^{\otimes \left[ p\right]
} $ is the linear span of $\mathcal{G}_{\left[ p\right] }$ (Condition \ref%
{Condition smallest Banach space}), we get 
\begin{equation}
\beta \left( a,bv\right) =\beta \left( ab,v\right) \text{, \ }\forall a,b\in 
\mathcal{G}_{\left[ p\right] },\forall v\in \mathcal{V}^{\otimes k},\text{ }%
k=1,\dots ,\left[ p\right] \text{.}  \label{equality between two one-forms}
\end{equation}

\subsection{Iterated integration}

Recall the mapping $\mathcal{I}\in L(T^{\left( \left[ p\right] \right)
}\left( \mathcal{V}\right) ,$ $T^{\left( \left[ p\right] \right) }\left( 
\mathcal{V}\right) ^{\otimes 2})$ in Condition \ref{Condition g satisfies
differential equation}.

\begin{proposition}[Iterated Integration]
\label{Proposition enhancement}\textit{Let }$g\in C^{p-var}\left( \left[ 0,T%
\right] ,\mathcal{G}_{\left[ p\right] }\right) $ and $\mathcal{U}^{i}$%
\textit{, }$i=1,2$,\textit{\ be Banach spaces. Suppose }$\int_{0}^{\cdot
}\beta _{u}^{i}\left( g_{u}\right) dg_{u}:\left[ 0,T\right] \rightarrow 
\mathcal{U}^{i}$, $i=1,2$, \textit{are two dominated paths satisfying }$%
\left\Vert \beta ^{i}\right\Vert _{\theta _{i}}^{\omega _{i}}<\infty $ for
control $\omega _{i}$ and $\theta _{i}>1$, $i=1,2$ (as defined in $\left( %
\ref{Definition of operator norm}\right) $). Then \textit{there exists a
dominated path }$\int_{0}^{\cdot }\beta _{u}\left( g_{u}\right) dg_{u}:\left[
0,T\right] \rightarrow \mathcal{U}^{1}\otimes \mathcal{U}^{2}$ such that
with control $\omega :=\omega _{1}+\omega _{2}+\left\Vert g\right\Vert
_{p-var}^{p}$ and $\theta :=\min \left( \theta _{1},\theta _{2}\right) $,%
\begin{equation}
\left\Vert \beta \right\Vert _{\theta }^{\omega }\leq C_{p,\omega \left(
0,T\right) }\left\Vert \beta ^{1}\right\Vert _{\theta _{1}}^{\omega
_{1}}\left\Vert \beta ^{2}\right\Vert _{\theta _{2}}^{\omega _{2}}\text{,}
\label{continuity of one-forms iterated integral}
\end{equation}%
and for $0\leq s\leq t\leq T$, 
\begin{equation}
\left\Vert \int_{s}^{t}\beta _{u}\left( g_{u}\right)
dg_{u}-\int_{0}^{s}\beta _{u}^{1}\left( g_{u}\right) dg_{u}\otimes \beta
_{s}^{2}\left( g_{s},g_{s,t}\right) -\beta _{s}^{1}\left( g_{s},\cdot
\right) \otimes \beta _{s}^{2}\left( g_{s},\cdot \right) \mathcal{I}\left(
g_{s,t}\right) \right\Vert \leq \omega \left( s,t\right) ^{\theta }\text{,}
\label{estimate of iterated integral}
\end{equation}%
\textit{where }$\beta _{s}^{1}\left( g_{s},\cdot \right) \otimes \beta
_{s}^{2}\left( g_{s},\cdot \right) $\textit{\ denotes the unique continuous
linear operator from }$T^{\left( \left[ p\right] \right) }\left( \mathcal{V}%
\right) ^{\otimes 2}$ to $\mathcal{U}^{1}\otimes \mathcal{U}^{2}$ satisfying 
$\left( \beta _{s}^{1}\left( g_{s},\cdot \right) \otimes \beta
_{s}^{2}\left( g_{s},\cdot \right) \right) \left( v_{1}\otimes v_{2}\right)
=\beta _{s}^{1}\left( g_{s},v_{1}\right) \otimes \beta _{s}^{2}\left(
g_{s},v_{2}\right) $\textit{, }$\forall v_{1},v_{2}\in T^{\left( \left[ p%
\right] \right) }\left( \mathcal{V}\right) $.
\end{proposition}

\begin{remark}
Let $\gamma _{\cdot }^{i}:=\int_{0}^{\cdot }\beta _{u}^{i}\left(
g_{u}\right) dg_{u}$, $i=1,2$. Then for $s<t$, 
\begin{eqnarray*}
&&\tiint\nolimits_{0<u_{1}<u_{2}<t}d\gamma _{u_{1}}^{1}\otimes d\gamma
_{u_{2}}^{2}-\tiint\nolimits_{0<u_{1}<u_{2}<s}d\gamma _{u_{1}}^{1}\otimes
d\gamma _{u_{2}}^{2} \\
&=&\left( \gamma _{s}^{1}-\gamma _{0}^{1}\right) \otimes \left( \gamma
_{t}^{2}-\gamma _{s}^{2}\right) +\tiint\nolimits_{s<u_{1}<u_{2}<t}d\gamma
_{u_{1}}^{1}\otimes d\gamma _{u_{2}}^{2}\text{,}
\end{eqnarray*}%
with $\left( \gamma _{s}^{1}-\gamma _{0}^{1}\right) \otimes \left( \gamma
_{t}^{2}-\gamma _{s}^{2}\right) $ and $\tiint\nolimits_{s<u_{1}<u_{2}<t}d%
\gamma _{u_{1}}^{1}\otimes d\gamma _{u_{2}}^{2}$ correspond to the two parts
in $\left( \ref{estimate of iterated integral}\right) $.
\end{remark}

\begin{remark}
\label{Remark integrating a weakly controlled path get a dominated path}When 
$\gamma ^{1}$ is a controlled path (Definition \ref{Definition weakly
controlled path}) and $\gamma ^{2}$ is a dominated path, same proof applies
and obtains that $\int_{0}^{\cdot }\gamma ^{1}\otimes d\gamma ^{2}$ is a
dominated path. In particular, when $\gamma ^{1}$ is a controlled path and $%
\gamma ^{2}=\pi _{1}\left( g\right) :=x$ (dominated by $g$ with $\beta
_{s}\left( a,b\right) =\pi _{1}\left( a\left( b-1\right) \right) $ for all $%
s $), the indefinite integral $\int_{0}^{\cdot }\gamma ^{1}\otimes dx$ is a
dominated path. This integral $\int_{0}^{\cdot }\gamma ^{1}\otimes dx$ is
well-defined when Condition 25' (instead of Condition 25) is satisfied (see
Lemma 11 \cite{lyons2015theory}).
\end{remark}

\begin{proof}
We check that $\beta $ satisfies the slowly-varying condition (Condition \ref%
{Condition of integrable beta}). Then $\left( \ref{estimate of iterated
integral}\right) $ follows from Theorem \ref{Theorem integrating slow
varying cocyclic one forms}. We define\textit{\ }$\beta ^{1,2}\in C(\left[
0,T\right] ,C(\mathcal{G}_{\left[ p\right] },L(T^{\left( \left[ p\right]
\right) }\left( \mathcal{V}\right) ,\mathcal{U}^{1}\otimes \mathcal{U}%
^{2}))) $\ by, for $a\in \mathcal{G}_{\left[ p\right] }$ and $v\in T^{\left( %
\left[ p\right] \right) }\left( \mathcal{V}\right) $,\textit{\ } 
\begin{equation}
\beta _{s}^{1,2}\left( a,v\right) :=\left( \beta _{s}^{1}\left( a,\cdot
\right) \otimes \beta _{s}^{2}\left( a,\cdot \right) \right) \mathcal{I}%
\left( v\right) \text{,}  \label{definition of beta1,2}
\end{equation}%
and define $\beta \in C\left( \left[ 0,T\right] ,B\left( \mathcal{G}_{\left[
p\right] },\mathcal{U}^{1}\otimes \mathcal{U}^{2}\right) \right) $ by, for $%
a,b\in \mathcal{G}_{\left[ p\right] }$ and $s\in \left[ 0,T\right] $,%
\begin{equation*}
\beta _{s}\left( a,b\right) :=\int_{0}^{s}\beta _{u}^{1}\left( g_{u}\right)
dg_{u}\otimes \beta _{s}^{2}\left( g_{s},g_{s}^{-1}a\left( b-1\right)
\right) +\beta _{s}^{1,2}\left( g_{s},g_{s}^{-1}a\left( b-1\right) \right) 
\text{.}
\end{equation*}

Recall the mapping $\mathcal{I}$ in Condition \ref{Condition g satisfies
differential equation} satisfy%
\begin{gather}
\mathcal{I}\left( 1\right) =\mathcal{I}\left( \mathcal{V}\right) =0\text{ \
and \ }\mathcal{I}\left( \mathcal{V}^{\otimes k}\right) \subseteq
\tsum\nolimits_{j_{1}+j_{2}=k,j_{i}\geq 1}\mathcal{V}^{\otimes j_{1}}\otimes 
\mathcal{V}^{\otimes j_{2}}\text{, }\ k=2,\dots ,\left[ p\right] \text{,}
\label{inner property of I} \\
\mathcal{I}\left( ab\right) =\mathcal{I}\left( a\right) +1_{\left[ p\right]
,2}\left( \left( a\otimes a\right) \mathcal{I}\left( b\right) \right) +1_{%
\left[ p\right] ,2}\left( \left( a-1\right) \otimes \left( a\left(
b-1\right) \right) \right) \text{, }\forall a,b\in \mathcal{G}_{\left[ p%
\right] }\text{,}  \label{inner property of I   II}
\end{gather}%
where $1_{\left[ p\right] ,2}$ denotes the projection of $T^{\left( \left[ p%
\right] \right) }\left( \mathcal{V}\right) ^{\otimes 2}$ to $%
\tsum\nolimits_{j_{1}+j_{2}\leq \left[ p\right] ,j_{i}\geq 1}\mathcal{V}%
^{\otimes j_{1}}\otimes \mathcal{V}^{\otimes j_{2}}$.

Fix $s<t$. Using $\left( \ref{inner property of I}\right) $, $||(\beta
_{t}^{1,2}-\beta _{s}^{1,2})\left( g_{t},\cdot \right) ||_{1}=0$.\ Using $%
\left( \ref{inner property of I II}\right) $, for $k=2,\dots ,\left[ p\right]
$ and $\theta =\min \left( \theta _{1},\theta _{2}\right) $, 
\begin{eqnarray}
&&\left\Vert \left( \beta _{t}^{1,2}-\beta _{s}^{1,2}\right) \left(
g_{t},\cdot \right) \right\Vert _{k}  \label{inner bound on beta1,2} \\
&\leq &C\tsum\nolimits_{j=1}^{k-1}\left( \left\Vert \left( \beta
_{t}^{1}-\beta _{s}^{1}\right) \left( g_{t},\cdot \right) \right\Vert
_{j}\left\Vert \beta _{t}^{2}\left( g_{t},\cdot \right) \right\Vert
_{k-j}+\left\Vert \beta _{s}^{1}\left( g_{t},\cdot \right) \right\Vert
_{j}\left\Vert \left( \beta _{t}^{2}-\beta _{s}^{2}\right) \left(
g_{t},\cdot \right) \right\Vert _{k-j}\right)  \notag \\
&\leq &C_{p,\omega \left( 0,T\right) }\left\Vert \beta ^{1}\right\Vert
_{\theta _{1}}^{\omega _{1}}\left\Vert \beta ^{2}\right\Vert _{\theta
_{2}}^{\omega _{2}}\omega \left( s,t\right) ^{\theta -\frac{k-1}{p}}\text{.}
\notag
\end{eqnarray}%
Moreover, by $\left( \ref{inner property of I II}\right) $, for $a,b,c\in 
\mathcal{G}_{\left[ p\right] }$ and $s\in \left[ 0,T\right] $, we have 
\begin{eqnarray}
\beta _{s}^{1,2}\left( a,bc\right) &=&\beta _{s}^{1,2}\left( a,b\right)
+\beta _{s}^{1,2}\left( ab,c\right) +\beta _{s}^{1}\left( a,b\right) \otimes
\beta _{s}^{2}\left( ab,c\right)
\label{beta12 group-valued cocyclic one-form} \\
&&-\tsum\nolimits_{\sigma _{i}\in \mathcal{P}_{\left[ p\right] },\left\vert
\sigma _{1}\right\vert +\left\vert \sigma _{2}\right\vert \geq \left[ p%
\right] +1}\beta _{s}^{1}\left( a,\sigma _{1}\left( b\right) \right) \otimes
\beta _{s}^{2}\left( a,\sigma _{2}\left( b\left( c-1\right) \right) \right) 
\notag \\
&&-\tsum\nolimits_{k=2}^{\left[ p\right] }\tsum\nolimits_{\sigma _{i}\in 
\mathcal{P}_{\left[ p\right] },\left\vert \sigma _{1}\right\vert +\left\vert
\sigma _{2}\right\vert \geq \left[ p\right] +1-k}\beta _{s}^{1}\left(
a,\sigma _{1}\left( b\right) \cdot \right) \otimes \beta _{s}^{2}\left(
a,\sigma _{2}\left( b\right) \cdot \right) \mathcal{I}\left( \pi _{k}\left(
c\right) \right) \text{,}  \notag
\end{eqnarray}%
where the extra terms are caused by the truncation in $\left( \ref{inner
property of I II}\right) $. Then it can be computed that, for $a\in \mathcal{%
G}_{\left[ p\right] }$,%
\begin{eqnarray}
&&\left( \beta _{t}-\beta _{s}\right) \left( g_{t},a\right)
\label{inner linear mapping iterated integral} \\
&=&\tint\nolimits_{0}^{t}\beta _{u}^{1}\left( g_{u}\right) dg_{u}\otimes
\left( \beta _{t}^{2}-\beta _{s}^{2}\right) \left( g_{t},a\right)  \notag \\
&&+\left( \tint\nolimits_{s}^{t}\beta _{u}^{1}\left( g_{u}\right)
dg_{u}-\beta _{s}^{1}\left( g_{s},g_{s,t}\right) \right) \otimes \beta
_{s}^{2}\left( g_{t},a\right) +\left( \beta _{t}^{1,2}-\beta
_{s}^{1,2}\right) \left( g_{t},a\right)  \notag \\
&&+\tsum\nolimits_{\sigma _{i}\in \mathcal{P}_{\left[ p\right] },\left\vert
\sigma _{1}\right\vert +\left\vert \sigma _{2}\right\vert \geq \left[ p%
\right] \ +1}\beta _{s}^{1}\left( g_{s},\sigma _{1}\left( g_{s,t}\right)
\right) \otimes \beta _{s}^{2}\left( g_{s},\sigma _{2}\left( g_{s,t}\left(
a-\sigma _{0}\left( a\right) \right) \right) \right)  \notag \\
&&+\tsum\nolimits_{k=2}^{\left[ p\right] }\tsum\nolimits_{\sigma _{i}\in 
\mathcal{P}_{\left[ p\right] },\left\vert \sigma _{1}\right\vert +\left\vert
\sigma _{2}\right\vert \geq \left[ p\right] \ +1-k}\beta _{s}^{1}\left(
g_{s},\sigma _{1}\left( g_{s,t}\right) \cdot \right) \otimes \beta
_{s}^{2}\left( g_{s},\sigma _{2}\left( g_{s,t}\right) \cdot \right) \mathcal{%
I}\left( \pi _{k}\left( a\right) \right) \text{.}  \notag
\end{eqnarray}%
Using Condition \ref{Condition smallest Banach space}, $\left( \ref{inner
linear mapping iterated integral}\right) $ can be extended to $T^{\left( %
\left[ p\right] \right) }\left( \mathcal{V}\right) $. Combining with $\left( %
\ref{estimate of integral in terms of one-forms}\right) $, $\left( \ref%
{Lipschitz continuity of the integral in terms of one-forms}\right) $, $%
\left( \ref{inner bound on beta1,2}\right) $ and $\left\Vert \sigma \left(
g_{s,t}\right) \right\Vert \leq \left\Vert g\right\Vert _{p-var,\left[ s,t%
\right] }^{\left\vert \sigma \right\vert }$, we have, with $\omega =\omega
_{1}+\omega _{2}+\left\Vert g\right\Vert _{p-var}^{p}$ and $\theta =\min
\left( \theta _{1},\theta _{2}\right) $,%
\begin{equation*}
\left\Vert \left( \beta _{t}-\beta _{s}\right) \left( g_{t},\cdot \right)
\right\Vert _{k}\leq C_{p,\omega \left( 0,T\right) }\left\Vert \beta
^{1}\right\Vert _{\theta _{1}}^{\omega _{1}}\left\Vert \beta ^{2}\right\Vert
_{\theta _{2}}^{\omega _{2}}\omega \left( s,t\right) ^{\theta -\frac{k}{p}}%
\text{.}
\end{equation*}%
Similarly we have $\left\Vert \beta _{s}\left( g_{s},\cdot \right)
\right\Vert \leq C_{p,\omega \left( 0,T\right) }\left\Vert \beta
^{1}\right\Vert _{\theta _{1}}^{\omega _{1}}\left\Vert \beta ^{2}\right\Vert
_{\theta _{2}}^{\omega _{2}}$, and the estimate $\left( \ref{continuity of
one-forms iterated integral}\right) $ holds.
\end{proof}

\subsection{Algebra}

Based on Stone-Weierstrass Theorem, being an algebra can be viewed as one of
the most important properties of dominated paths. When working with measures
on paths space, the algebra structure is compatible with the filtration
generated, and the space of previsible integrable one-forms forms an algebra.

\begin{proposition}[Algebra]
\label{Proposition Algebra}\textit{Let }$g\in C^{p-var}\left( \left[ 0,T%
\right] ,\mathcal{G}_{\left[ p\right] }\right) $\textit{. Suppose }$%
\int_{0}^{\cdot }\beta _{u}^{i}\left( g_{u}\right) dg_{u}:\left[ 0,T\right]
\rightarrow \mathcal{U}^{i}$, $i=1,2$, are two dominated paths satisfying $%
\left\Vert \beta ^{i}\right\Vert _{\theta _{i}}^{\omega _{i}}<\infty $ for
control $\omega _{i}$ and $\theta _{i}>1$, $i=1,2$ (as defined in $\left( %
\ref{Definition of operator norm}\right) $). Then \textit{there exists a
dominated path }$\int_{0}^{\cdot }\beta _{u}\left( g_{u}\right) dg_{u}:\left[
0,T\right] \rightarrow \mathcal{U}^{1}\otimes \mathcal{U}^{2}$ such that
with control $\omega :=\omega _{1}+\omega _{2}+\left\Vert g\right\Vert
_{p-var}^{p}$ and $\theta :=\min \left( \theta _{1},\theta _{2}\right) $,%
\begin{equation}
\left\Vert \beta \right\Vert _{\theta }^{\omega }\leq C_{p,\omega \left(
0,T\right) }\left\Vert \beta ^{1}\right\Vert _{\theta _{1}}^{\omega
_{1}}\left\Vert \beta ^{2}\right\Vert _{\theta _{2}}^{\omega _{2}}\text{,}
\label{continuity of one-forms in algebra}
\end{equation}%
\begin{equation*}
\text{and }\int_{0}^{t}\beta _{u}\left( g_{u}\right)
dg_{u}=\int_{0}^{t}\beta _{u}^{1}\left( g_{u}\right) dg_{u}\otimes
\int_{0}^{t}\beta _{u}^{2}\left( g_{u}\right) dg_{u}\text{, }\forall 0\leq
t\leq T\text{.}
\end{equation*}
\end{proposition}

\begin{remark}
The fact that dominated paths form an algebra does not necessarily follow
from the statement that the iterated integral of two dominated paths is
canonically defined. This depends on the definition of the formal integral $%
\iint_{0<u_{1}<u_{2}<T}\delta g_{0,u_{1}}\otimes \delta g_{0,u_{2}}$ (see
Condition \ref{Condition g satisfies differential equation} and the
discussion thereafter). The reason is that the integration by parts formula
may not hold:%
\begin{equation}
\tiint\nolimits_{0<u_{1}<u_{2}<T}\delta g_{0,u_{1}}\otimes \delta
g_{0,u_{2}}+\tiint\nolimits_{0<u_{2}<u_{1}<T}\delta g_{0,u_{1}}\otimes
\delta g_{0,u_{2}}\overset{?}{=}g_{0,T}\otimes g_{0,T}\text{, }\forall g\in
C\left( \left[ 0,T\right] ,\mathcal{G}\right) \text{.}
\label{relationship between iterated integral and algebra}
\end{equation}%
When $\mathcal{G}$ is the nilpotent Lie group, $\left( \ref{relationship
between iterated integral and algebra}\right) $ holds; when $\mathcal{G}$ is
the Butcher group, generally $\left( \ref{relationship between iterated
integral and algebra}\right) $ does not hold.
\end{remark}

\begin{proof}
\textit{We define }$\beta \in C\left( \left[ 0,T\right] ,B\left( \mathcal{G}%
_{\left[ p\right] },\mathcal{U}^{1}\otimes \mathcal{U}^{2}\right) \right) $
by, for $a,b\in \mathcal{G}_{\left[ p\right] }$ and $s\in \left[ 0,T\right] $%
, 
\begin{eqnarray}
\beta _{s}\left( a,b\right) &=&\beta _{s}^{1}\left( g_{s},g_{s}^{-1}a\left(
b-1\right) \right) \otimes \tint\nolimits_{0}^{s}\beta _{u}^{2}\left(
g_{u}\right) dg_{u}+\tint\nolimits_{0}^{s}\beta _{u}^{1}\left( g_{u}\right)
dg_{u}\otimes \beta _{s}^{2}\left( g_{s},g_{s}^{-1}a\left( b-1\right) \right)
\label{definition of beta in lemma for algebra} \\
&&+\tsum\limits_{\sigma _{i}\in \mathcal{P}_{\left[ p\right] },\left\vert
\sigma _{1}\right\vert +\left\vert \sigma _{2}\right\vert \leq \left[ p%
\right] }\beta _{s}^{1}\left( g_{s},\cdot \right) \otimes \beta
_{s}^{2}\left( g_{s},\cdot \right) \left( \sigma _{1}\ast \sigma _{2}\right)
\left( g_{s}^{-1}a\left( b-1\right) \right) \text{,}  \notag
\end{eqnarray}%
where $\sigma _{1}\ast \sigma _{2}\in L(\mathcal{V}^{\otimes (|\sigma
_{1}|+|\sigma _{2}|)},\mathcal{V}^{\otimes |\sigma _{1}|}\otimes \mathcal{V}%
^{\otimes |\sigma _{2}|})$ is defined in Condition \ref{Condition P is an
algebra} and $\beta _{s}^{1}\left( g_{s},\cdot \right) \otimes \beta
_{s}^{2}\left( g_{s},\cdot \right) $ denotes the unique continuous linear
mapping from $T^{\left( \left[ p\right] \right) }\left( \mathcal{V}\right)
^{\otimes 2}$ to $\mathcal{U}^{1}\otimes \mathcal{U}^{2}$ satisfying 
\begin{equation*}
\beta _{s}^{1}\left( g_{s},\cdot \right) \otimes \beta _{s}^{2}\left(
g_{s},\cdot \right) \left( v_{1}\otimes v_{2}\right) =\beta _{s}^{1}\left(
g_{s},v_{1}\right) \otimes \beta _{s}^{2}\left( g_{s},v_{2}\right) ,\forall
v_{1},v_{2}\in T^{\left( \left[ p\right] \right) }\left( \mathcal{V}\right) .
\end{equation*}%
Since there exists a unique multiplicative function associated with an
almost multiplicative function \cite{lyons1998differential}, we have 
\begin{equation}
\int_{0}^{t}\beta _{u}^{1}\left( g_{u}\right) dg_{u}\otimes
\int_{0}^{t}\beta _{u}^{2}\left( g_{u}\right) dg_{u}=\int_{0}^{t}\beta
_{u}\left( g_{u}\right) dg_{u}\text{, }\forall t\in \left[ 0,T\right] \text{.%
}  \label{Relation algebra}
\end{equation}

When proving that $\beta $ satisfies the slowly-varying condition, there is
an extra term caused by truncation up to level $\left[ p\right] $ in $\left( %
\ref{definition of beta in lemma for algebra}\right) $. The extra term can
be derived by using that for $\tau _{i},\sigma _{i}\in \mathcal{P}_{\left[ p%
\right] }$ and $a,b\in \mathcal{G}_{\left[ p\right] }$, ($\tau _{1}\otimes
\tau _{2}$ denotes the continuous linear mapping that satisfies $\left( \tau
_{1}\otimes \tau _{2}\right) \left( a\right) :=\tau _{1}\left( a\right)
\otimes \tau _{2}\left( a\right) $, $\forall a\in \mathcal{G}_{\left[ p%
\right] }$, and $\left( a_{1}\otimes a_{2}\right) \left( b_{1}\otimes
b_{2}\right) :=\left( a_{1}b_{1}\right) \otimes \left( a_{2}b_{2}\right) $)%
\begin{eqnarray*}
&&\left( \tau _{1}\otimes \tau _{2}\right) \left( a\left( b-1\right) \right)
\\
&=&\tsum\nolimits_{\sigma _{i}\in \mathcal{P}_{\left[ p\right] }}c\left(
\tau _{1},\rho _{1},\sigma _{1}\right) c\left( \tau _{2},\rho _{2},\sigma
_{2}\right) \left( \rho _{1}\left( a\right) \otimes \rho _{2}\left( a\right)
\right) \left( \left( \sigma _{1}\otimes \sigma _{2}\right) \left(
b-1\right) \right) \text{,} \\
&&\left( a\otimes a\right) \left( \left( \sigma _{1}\otimes \sigma
_{2}\right) \left( b-1\right) \right) \\
&=&\tsum\nolimits_{\tau _{i}\in \mathcal{P}_{\left[ p\right] }}c\left( \tau
_{1},\rho _{1},\sigma _{1}\right) c\left( \tau _{2},\rho _{2},\sigma
_{2}\right) \left( \rho _{1}\left( a\right) \otimes \rho _{2}\left( a\right)
\right) \left( \left( \sigma _{1}\otimes \sigma _{2}\right) \left(
b-1\right) \right) \text{,}
\end{eqnarray*}%
where the constant $c\left( \tau _{i},\rho _{i},\sigma _{i}\right) $, $i=1,2$%
, denotes the number of $\rho _{i}\otimes \sigma _{i}$ in the
comultiplication of $\tau _{i}$. Hence 
\begin{eqnarray*}
&&\tsum\nolimits_{\left\vert \tau _{1}\right\vert +\left\vert \tau
_{2}\right\vert \geq \left[ p\right] +1}\left( \tau _{1}\otimes \tau
_{2}\right) \left( a\left( b-1\right) \right) -\tsum\nolimits_{\left\vert
\sigma _{1}\right\vert +\left\vert \sigma _{2}\right\vert \geq \left[ p%
\right] +1}\left( a\otimes a\right) \left( \left( \sigma _{1}\otimes \sigma
_{2}\right) \left( b-1\right) \right) \\
&=&\tsum\nolimits_{\left\vert \tau _{1}\right\vert +\left\vert \tau
_{2}\right\vert \geq \left[ p\right] +1,\left\vert \sigma _{1}\right\vert
+\left\vert \sigma _{2}\right\vert \leq \left[ p\right] }c\left( \tau
_{1},\rho _{1},\sigma _{1}\right) c\left( \tau _{2},\rho _{2},\sigma
_{2}\right) \left( \rho _{1}\left( a\right) \otimes \rho _{2}\left( a\right)
\right) \left( \left( \sigma _{1}\otimes \sigma _{2}\right) \left(
b-1\right) \right) \text{.}
\end{eqnarray*}

Then, for $s<t$ and $a\in \mathcal{G}_{\left[ p\right] }$, it can be
computed that 
\begin{eqnarray}
&&\left( \beta _{t}-\beta _{s}\right) \left( g_{t},a\right)
\label{inner representation algebra} \\
&=&\left( \beta _{t}^{1}-\beta _{s}^{1}\right) \left( g_{t},a\right) \otimes
\tint\nolimits_{0}^{s}\beta _{v}^{2}\left( g_{v}\right)
dg_{v}+\tint\nolimits_{0}^{s}\beta _{v}^{1}\left( g_{v}\right) dg_{v}\otimes
\left( \beta _{t}^{2}-\beta _{s}^{2}\right) \left( g_{t},a\right)  \notag \\
&&+\beta _{t}^{1}\left( g_{t},a\right) \otimes \left(
\tint\nolimits_{s}^{t}\beta _{v}^{2}\left( g_{v}\right) dg_{v}-\beta
_{s}^{2}\left( g_{s},g_{s,t}\right) \right) +\left(
\tint\nolimits_{s}^{t}\beta _{v}^{1}\left( g_{v}\right) dg_{v}-\beta
_{s}^{1}\left( g_{s},g_{s,t}\right) \right) \otimes \beta _{t}^{2}\left(
g_{t},a\right)  \notag \\
&&+\left( \beta _{t}^{1}-\beta _{s}^{1}\right) \left( g_{t},a\right) \otimes
\beta _{s}^{2}\left( g_{s},g_{s,t}\right) +\beta _{s}^{1}\left(
g_{s},g_{s,t}\right) \otimes \left( \beta _{t}^{2}-\beta _{s}^{2}\right)
\left( g_{t},a\right)  \notag \\
&&+\tsum\nolimits_{\left\vert \sigma _{1}\right\vert +\left\vert \sigma
_{2}\right\vert \leq \left[ p\right] }\left( \beta _{t}^{1}\left(
g_{t},\cdot \right) \otimes \beta _{t}^{2}\left( g_{t},\cdot \right) -\beta
_{s}^{1}\left( g_{t},\cdot \right) \otimes \beta _{s}^{2}\left( g_{t},\cdot
\right) \right) \left( \left( \sigma _{1}\otimes \sigma _{2}\right) \left(
a-\sigma _{0}\left( a\right) \right) \right)  \notag \\
&&+\tsum\nolimits_{\substack{ \left\vert \tau _{1}\right\vert +\left\vert
\tau _{2}\right\vert \geq \left[ p\right] +1  \\ \left\vert \sigma
_{1}\right\vert +\left\vert \sigma _{2}\right\vert \leq \left[ p\right] }}%
c\left( \tau _{1},\rho _{1},\sigma _{1}\right) c\left( \tau _{2},\rho
_{2},\sigma _{2}\right) \beta _{s}^{1}\left( g_{s},\rho _{1}\left(
g_{s,t}\right) \cdot \right) \otimes \beta _{s}^{2}\left( g_{s},\rho
_{2}\left( g_{s,t}\right) \cdot \right) \left( \left( \sigma _{1}\otimes
\sigma _{2}\right) \left( a-\sigma _{0}\left( a\right) \right) \right) \text{%
.}  \notag
\end{eqnarray}%
We represent $\sigma _{1}\otimes \sigma _{2}$ as a continuous linear mapping 
$\sigma _{1}\ast \sigma _{2}$ by using Condition \ref{Condition P is an
algebra}. Similar as in the proof of Proposition \ref{Proposition
enhancement}, we have $\left( \ref{continuity of one-forms in algebra}%
\right) $ holds based on Condition \ref{Condition smallest Banach space}, $%
\left( \ref{inner representation algebra}\right) $, $\left( \ref{estimate of
integral in terms of one-forms}\right) $, $\left( \ref{Lipschitz continuity
of the integral in terms of one-forms}\right) $ and $\left\Vert \rho \left(
g_{s,t}\right) \right\Vert \leq \left\Vert g\right\Vert _{p-var,\left[ s,t%
\right] }^{\left\vert \rho \right\vert }$.
\end{proof}

\begin{remark}
In $\left( \ref{definition of beta in lemma for algebra}\right) $, $\beta $
is defined as the sum of three time-varying cocyclic one-forms (denoted by $%
(\eta ^{i})_{i=1,2,3}$). Although $\beta $ is slowly-varying as we proved
above, $(\eta ^{i})_{i=1,2,3}$ are generally not slowly-varying. Roughly
speaking, the difference between $\eta _{t}^{i}$ and $\eta _{s}^{i}$ for $%
s<t $ would be comparable to $\left\Vert g\right\Vert _{p-var,\left[ s,t%
\right] } $ that is not (slow) enough to integrate against $g$ when $p\geq 2$%
.
\end{remark}

\begin{remark}
Proposition \ref{Proposition Algebra} is a special (important) case of
Proposition \ref{Proposition Composition}.
\end{remark}

\subsection{Composition}

For $\gamma >0$, let $\lfloor \gamma \rfloor $ denote the largest integer
which is strictly less than $\gamma $. Let $\mathcal{U}$ and $\mathcal{W}$
be two Banach spaces. We say $f\in C^{\gamma }\left( \mathcal{U},\mathcal{W}%
\right) $, if $f:\mathcal{U}\rightarrow \mathcal{W}$ is $\lfloor \gamma
\rfloor $-times Fr\'{e}chet differentiable and%
\begin{equation}
\left\Vert f\right\Vert _{\limfunc{Lip}\left( \gamma \right)
,R}:=\sup_{x\neq y,\left\Vert x\right\Vert \vee \left\Vert y\right\Vert \leq
R}\frac{\left\Vert \left( D^{\lfloor \gamma \rfloor }f\right) \left(
x\right) -\left( D^{\lfloor \gamma \rfloor }f\right) \left( y\right)
\right\Vert }{\left\Vert x-y\right\Vert ^{\gamma -\lfloor \gamma \rfloor }}%
\leq C_{R}\text{, }\forall R>0\text{.}  \label{definition of Cgamma}
\end{equation}

\begin{proposition}[Composition]
\label{Proposition Composition}\textit{Let }$g\in C^{p-var}\left( \left[ 0,T%
\right] ,\mathcal{G}_{\left[ p\right] }\right) $\textit{, and suppose }$%
X_{\cdot }:=\int_{0}^{\cdot }\beta _{u}\left( g_{u}\right) dg_{u}:\left[ 0,T%
\right] \rightarrow \mathcal{U}$\textit{\ is a dominated path satisfying }$%
\left\Vert \beta \right\Vert _{\omega }^{\theta }<\infty $ for control $%
\omega $ and $\theta >1$\textit{. For Banach space }$\mathcal{W}$\textit{\
and }$f\in C^{\gamma }\left( \mathcal{U},\mathcal{W}\right) $\textit{, }$%
\gamma >p$\textit{,} \textit{there exists a dominated path }$\int_{0}^{\cdot
}\hat{\beta}\left( g\right) dg:\left[ 0,T\right] \rightarrow \mathcal{W}$%
\textit{\ such that with control }$\hat{\omega}=\omega +\left\Vert
g\right\Vert _{p-var}^{p}$ and $\hat{\theta}:=\min \left( \theta ,\frac{%
\gamma }{p},\frac{\left[ p\right] +1}{p}\right) $,%
\begin{equation}
\left\Vert \hat{\beta}\right\Vert _{\hat{\theta}}^{\hat{\omega}}\leq C_{p,%
\hat{\omega}\left( 0,T\right) }\left\Vert f\right\Vert _{\limfunc{Lip}\left(
\gamma \right) ,\left\Vert X\right\Vert _{\infty }}\max \left( \left\Vert
\beta \right\Vert _{\theta }^{\omega },\left( \left\Vert \beta \right\Vert
_{\theta }^{\omega }\right) ^{\left[ p\right] }\right)
\label{continuity of one-forms in composition}
\end{equation}%
where $\left\Vert X\right\Vert _{\infty }:=\sup\nolimits_{t\in \left[ 0,T%
\right] }\left\Vert X_{t}\right\Vert $, and 
\begin{equation*}
\int_{0}^{t}\hat{\beta}_{u}\left( g_{u}\right) dg_{u}=f\left( X_{t}\right)
-f\left( 0\right) ,\forall 0\leq t\leq T.
\end{equation*}
\end{proposition}

\begin{proof}
We rescale $f$ by $\left\Vert f\right\Vert _{\limfunc{Lip}\left( \gamma
\right) ,\left\Vert X\right\Vert _{\infty }}^{-1}$ and assume $\left\Vert
f\right\Vert _{\limfunc{Lip}\left( \gamma \right) ,\left\Vert X\right\Vert
_{\infty }}=1$.

Define $\hat{\beta}\in C(\left[ 0,T\right] ,B(\mathcal{G}_{\left[ p\right] },%
\mathcal{W}))$ by, for $a,b\in \mathcal{G}_{\left[ p\right] }$ and $s\in %
\left[ 0,T\right] $, 
\begin{equation}
\hat{\beta}_{s}(a,b):=\tsum\nolimits_{l=1}^{\left[ p\right] }\frac{1}{l!}%
\left( D^{l}f\right) \left( X_{s}\right) \tsum\limits_{\sigma _{i}\in 
\mathcal{P}_{\left[ p\right] },\left\vert \sigma _{1}\right\vert +\cdots
+\left\vert \sigma _{l}\right\vert \leq \left[ p\right] }\beta _{s}\left(
g_{s},\cdot \right) ^{\otimes l}\left( \sigma _{1}\ast \cdots \ast \sigma
_{l}\right) \left( g_{s}^{-1}a\left( b-1\right) \right) \text{,}
\label{inner definition of one-form in composition}
\end{equation}%
where $\sigma _{1}\ast \cdots \ast \sigma _{l}\in L(\mathcal{V}^{\otimes
(|\sigma _{1}|+\cdots +|\sigma _{l}|)},\mathcal{V}^{\otimes |\sigma
_{1}|}\otimes \cdots \otimes \mathcal{V}^{\otimes |\sigma _{l}|})$ is
defined in Condition \ref{Condition P is an algebra} and\ $\beta _{s}\left(
g_{s},\cdot \right) ^{\otimes l}$ denotes the unique continuous linear
mapping from $T^{\left( \left[ p\right] \right) }\left( \mathcal{V}\right)
^{\otimes l}$ to $\mathcal{U}^{\otimes l}$ satisfying 
\begin{equation*}
\beta _{s}\left( g_{s},\cdot \right) ^{\otimes l}\left( v_{1}\otimes \cdots
\otimes v_{l}\right) =\beta _{s}\left( g_{s},v_{1}\right) \otimes \cdots
\otimes \beta _{s}\left( g_{s},v_{l}\right) ,\text{ }\forall v_{i}\in
T^{\left( \left[ p\right] \right) }\left( \mathcal{V}\right) ,\text{ }%
i=1,\dots ,l\text{.}
\end{equation*}%
Since the multiplicative function associated with an almost multiplicative
function is unique \cite{lyons1998differential}, we have $f\left(
X_{t}\right) =f\left( 0\right) +\tint\nolimits_{0}^{t}\hat{\beta}_{u}\left(
g_{u}\right) dg_{u}$, $\forall t\in \left[ 0,T\right] $. Then we check that $%
\hat{\beta}$ satisfies the slowly-varying condition.\ As in the proof of
Proposition \ref{Proposition Algebra}, the extra term caused by truncation
up to level $\left[ p\right] $ in $\left( \ref{inner definition of one-form
in composition}\right) $\ can be estimated based on the formula that for $%
\tau _{i},\sigma _{i}\in \mathcal{P}_{\left[ p\right] }$ and $a,b\in 
\mathcal{G}_{\left[ p\right] }$,%
\begin{eqnarray*}
&&\tsum\nolimits_{\left\vert \tau _{1}\right\vert +\cdots +\left\vert \tau
_{l}\right\vert \geq \left[ p\right] +1}\left( \tau _{1}\otimes \cdots
\otimes \tau _{l}\right) \left( a\left( b-1\right) \right)
-\tsum\nolimits_{\left\vert \sigma _{1}\right\vert +\cdots +\left\vert
\sigma _{l}\right\vert \geq \left[ p\right] +1}a^{\otimes l}\left( \left(
\sigma _{1}\otimes \cdots \otimes \sigma _{l}\right) \left( b-1\right)
\right) \\
&=&\tsum\nolimits_{\substack{ \left\vert \tau _{1}\right\vert +\cdots
+\left\vert \tau _{l}\right\vert \geq \left[ p\right] +1  \\ \left\vert
\sigma _{1}\right\vert +\cdots +\left\vert \sigma _{l}\right\vert \leq \left[
p\right] }}c\left( \tau _{1},\rho _{1},\sigma _{1}\right) \cdots c\left(
\tau _{l},\rho _{l},\sigma _{l}\right) \left( \rho _{1}\left( a\right)
\otimes \cdots \otimes \rho _{l}\left( a\right) \right) \left( \left( \sigma
_{1}\otimes \cdots \otimes \sigma _{l}\right) \left( b-1\right) \right) 
\text{,}
\end{eqnarray*}%
where $c\left( \tau _{i},\rho _{i},\sigma _{i}\right) $ denotes the number
of $\rho _{i}\otimes \sigma _{i}$ in the coproduct of $\tau _{i}$.

Hence, for $s<t$ and $a\in \mathcal{G}_{\left[ p\right] }$, it can be
computed that 
\begin{eqnarray}
&&\left( \hat{\beta}_{t}-\hat{\beta}_{s}\right) \left( g_{t},a\right)
\label{inner composition linear} \\
&=&\tsum\nolimits_{l=1}^{\left[ p\right] }\frac{1}{l!}\left( \left(
D^{l}f\right) \left( X_{t}\right) -\tsum\nolimits_{j=0}^{\left[ p\right] -l}%
\frac{1}{j!}\left( D^{l+j}f\right) \left( X_{s}\right) \left(
X_{t}-X_{s}\right) ^{\otimes j}\right) \beta _{t}\left( g_{t},\cdot \right)
^{\otimes l}L^{l}\left( a\right)  \notag \\
&&+\tsum\nolimits_{l=1}^{\left[ p\right] }\tsum\nolimits_{j=0}^{\left[ p%
\right] -l}\frac{1}{l!}\frac{1}{j!}\left( D^{l+j}f\right) \left(
X_{s}\right) \left( \left( X_{t}-X_{s}\right) ^{\otimes j}-\left( \beta
_{s}\left( g_{s},g_{s,t}\right) \right) ^{\otimes j}\right) \beta _{t}\left(
g_{t},\cdot \right) ^{\otimes l}L^{l}\left( a\right)  \notag \\
&&+\tsum\nolimits_{l=1}^{\left[ p\right] }\frac{1}{l!}\left( D^{l}f\right)
\left( X_{s}\right) \left( \left( \beta _{s}\left( g_{s},g_{s,t}\right)
+\beta _{t}\left( g_{t},\cdot \right) \right) ^{\otimes l}-\left( \beta
_{s}\left( g_{s},g_{s,t}\right) +\beta _{s}\left( g_{t},\cdot \right)
\right) ^{\otimes l}\right) L^{l}\left( a\right)  \notag \\
&&+\tsum\nolimits_{l=1}^{\left[ p\right] }\frac{1}{l!}\left( D^{l}f\right)
\left( X_{s}\right) \beta _{s}\left( g_{s},\cdot \right) ^{\otimes
l}R^{l}\left( a\right) \text{,}  \notag
\end{eqnarray}%
where%
\begin{eqnarray*}
L^{l}\left( a\right) &:&=\tsum\nolimits_{\left\vert \sigma _{1}\right\vert
+\cdots +\left\vert \sigma _{l}\right\vert \leq \left[ p\right] }\left(
\sigma _{1}\otimes \cdots \otimes \sigma _{l}\right) \left( a\right) \text{,}
\\
R^{l}\left( a\right) &:&=\tsum\nolimits_{\substack{ \left\vert \tau
_{1}\right\vert +\cdots +\left\vert \tau _{l}\right\vert \geq \left[ p\right]
+1  \\ \left\vert \sigma _{1}\right\vert +\cdots +\left\vert \sigma
_{l}\right\vert \leq \left[ p\right] }}c\left( \tau _{1},\rho _{1},\sigma
_{1}\right) \cdots c\left( \tau _{l},\rho _{l},\sigma _{l}\right) \left(
\rho _{1}\left( g_{s,t}\right) \otimes \cdots \otimes \rho _{l}\left(
g_{s,t}\right) \right) \left( \left( \sigma _{1}\otimes \cdots \otimes
\sigma _{l}\right) \left( a-1\right) \right) \text{.}
\end{eqnarray*}%
Based on Condition \ref{Condition smallest Banach space} and Condition \ref%
{Condition P is an algebra}, we represent $L^{l}\left( \cdot \right) $ and $%
R^{l}\left( \cdot \right) $ as continuous linear mappings on $T^{\left( %
\left[ p\right] \right) }\left( \mathcal{V}\right) $ and extend $\left( \ref%
{inner composition linear}\right) $ from $\mathcal{G}_{\left[ p\right] }$ to 
$T^{\left( \left[ p\right] \right) }\left( \mathcal{V}\right) $. By using
Taylor's Theorem, estimate $\left( \ref{estimate of integral in terms of
one-forms}\right) $, $\left( \ref{Lipschitz continuity of the integral in
terms of one-forms}\right) $, the slow-varying property of $\beta $, and
that $\left\Vert \rho \left( g\right) \right\Vert \leq \left\Vert
g\right\Vert _{p-var,\left[ s,t\right] }^{\left\vert \rho \right\vert }$, we
have $\left( \ref{continuity of one-forms in composition}\right) $ holds.
\end{proof}

\subsection{Transitivity}

Let $\mathcal{U}$ be a Banach space. Suppose the multiplication in the
Banach algebra $T^{([p])}(\mathcal{U})$ is defined by (with $\pi _{k}$
denotes the projection to $\mathcal{U}^{\otimes k}$) $\pi _{k}\left(
ab\right) =\tsum\nolimits_{i=0}^{k}\pi _{i}\left( a\right) \otimes \pi
_{k-i}\left( b\right) $ for $k=0,1,\dots ,\left[ p\right] $ and $a,b\in
T^{([p])}(\mathcal{U})$. Let $1\oplus \mathcal{U}\oplus \cdots \oplus 
\mathcal{U}^{\otimes \left[ p\right] }$ denote the closed topological group
in $T^{([p])}(\mathcal{U})$ defined by%
\begin{equation}
1\oplus \mathcal{U}\oplus \cdots \oplus \mathcal{U}^{\otimes \left[ p\right]
}:=\left\{ a\in T^{([p])}(\mathcal{U})\text{ }\big|\text{ }\pi _{0}\left(
a\right) =1\right\} \text{.}  \label{definition of group 1+U+...+Up}
\end{equation}

\begin{proposition}[Transitivity]
\label{Proposition Transitivity}Let $\gamma _{\cdot }:=\int_{0}^{\cdot
}\beta _{u}\left( g_{u}\right) dg_{u}:\left[ 0,T\right] \rightarrow \mathcal{%
U}$ be a path dominated by $g$ satisfying $\left\Vert \beta \right\Vert
_{\theta }^{\omega }<\infty $ for control $\omega $ and $\theta >1$. Define $%
\Gamma :\left[ 0,T\right] \rightarrow 1\oplus \mathcal{U}\oplus \cdots
\oplus \mathcal{U}^{\otimes \left[ p\right] }$ by 
\begin{equation*}
\Gamma _{t}:=1+\tsum\nolimits_{n=1}^{\left[ p\right] }x_{t}^{n}\text{ with }%
x_{t}^{1}:=\gamma _{t}-\gamma _{0}\text{ and }x_{t}^{n}:=\tint%
\nolimits_{0}^{t}x_{u}^{n-1}\otimes d\gamma _{u}\text{, }n=2,\dots ,\left[ p%
\right]
\end{equation*}%
where the integrals are defined as in Proposition \ref{Proposition
enhancement}.

For Banach space $\mathcal{W}$, suppose $\int_{0}^{\cdot }\zeta _{u}\left(
\Gamma _{u}\right) d\Gamma _{u}:\left[ 0,T\right] \rightarrow \mathcal{W}$
is a path dominated by $\Gamma $ satisfying $\left\Vert \zeta \right\Vert
_{\kappa }^{\rho }<\infty $ for control $\rho $ and $\kappa >1$. Then there
exists a dominated path $\int_{0}^{\cdot }\widetilde{\beta }\left( g\right)
dg:\left[ 0,T\right] \rightarrow \mathcal{W}$ such that with control $\tilde{%
\omega}:=\omega +\rho +\left\Vert g\right\Vert _{p-var}^{p}$ and $\tilde{%
\theta}:=\min \left( \theta ,\kappa \right) >1$,%
\begin{gather}
\left\Vert \widetilde{\beta }\right\Vert _{\tilde{\theta}}^{\tilde{\omega}%
}\leq C_{p,\tilde{\omega}\left( 0,T\right) }\left\Vert \zeta \right\Vert
_{\kappa }^{\rho }\max \left( \left\Vert \beta \right\Vert _{\theta
}^{\omega },\left( \left\Vert \beta \right\Vert _{\theta }^{\omega }\right)
^{\left[ p\right] }\right) \text{,}
\label{continuity of one-forms in transitivity} \\
\text{and }\int_{0}^{t}\widetilde{\beta }\left( g\right)
dg=\int_{0}^{t}\zeta _{u}\left( \Gamma _{u}\right) d\Gamma _{u}\text{, }%
\forall 0\leq t\leq T\text{.}  \notag
\end{gather}
\end{proposition}

\begin{remark}
Whether $\Gamma $ takes values in the step-$\left[ p\right] $ nilpotent Lie
group or not will depend on the mapping $\mathcal{I}$ in Condition \ref%
{Condition g satisfies differential equation}.
\end{remark}

\begin{remark}
We can also define $\tilde{\Gamma}$, taking values in the `\thinspace
infinite-dimensional\thinspace ' Butcher group over\ $\mathcal{U}$, by%
\begin{equation*}
\tilde{\Gamma}_{t}:=1+\tsum\nolimits_{\sigma \in \mathcal{Q}_{\left[ p\right]
}}x_{t}^{\sigma }\text{ with }x_{t}^{\mbox{\thinspace}\Sbullet}:=\gamma
_{t}-\gamma _{0}\text{, }x_{t}^{\sigma _{1}\sigma _{2}}:=x_{t}^{\sigma
_{1}}\otimes x_{t}^{\sigma _{2}}\text{, }x_{t}^{\left[ \sigma _{1}\right]
}:=\tint\nolimits_{0}^{t}x_{u}^{\sigma _{1}}\otimes d\gamma _{u}\text{, }%
\forall \sigma _{1},\sigma _{2}\in \mathcal{Q}_{\left[ p\right] }\text{,}
\end{equation*}%
where $\mathcal{Q}_{\left[ p\right] }$ denotes the set of unlabelled ordered
forests of degree less or equal to $\left[ p\right] $ and $[\sigma _{1}]$
denotes the tree obtained by attaching the forest $\sigma _{1}$ to a new
root. $\tilde{\Gamma}$ is well-defined because the set of dominated paths is
stable under multiplication and iterated integration.
\end{remark}

\begin{proof}
We extend the definition of $\Gamma $ to $\left\{ \left( s,t\right) |0\leq
s\leq t\leq T\right\} $ to make the proof work. 
\begin{equation}
\Gamma _{s,t}:=1+\tsum\nolimits_{n=1}^{\left[ p\right] }x_{s,t}^{n}\text{
with }x_{s,t}^{1}:=\gamma _{t}-\gamma _{s}\text{ and }x_{s,t}^{n}:=\tint%
\nolimits_{s}^{t}x_{s,u}^{n-1}\otimes d\gamma _{u}\text{, }n=2,\dots ,\left[
p\right] \text{.}  \label{Definition of Gamma}
\end{equation}%
We view $\mathcal{U}\oplus \cdots \oplus \mathcal{U}^{\otimes \left[ p\right]
}$ as a Banach space (with norm $\sum_{k=1}^{\left[ p\right] }\left\Vert \pi
_{k}\left( \cdot \right) \right\Vert $), and define $B_{s,t}\in B\left( 
\mathcal{G}_{\left[ p\right] },\mathcal{U}\oplus \cdots \oplus \mathcal{U}%
^{\otimes \left[ p\right] }\right) $ for $s<t$ by%
\begin{equation}
B_{s,t}:=\tsum\nolimits_{n=1}^{\left[ p\right] }\beta _{s,t}^{n}\text{,}
\label{inner definition of Beta s,t}
\end{equation}%
where $\beta _{s,t}^{n}\in B\left( \mathcal{G}_{\left[ p\right] },\mathcal{U}%
^{\otimes n}\right) $ are defined by, for $a,b\in \mathcal{G}_{\left[ p%
\right] }$, (with $\mathcal{I}$ in Condition \ref{Condition g satisfies
differential equation})%
\begin{eqnarray}
\beta _{s,t}^{1}\left( a,b\right) &:&=\beta _{t}\left( a,b\right) \text{,} 
\notag \\
\beta _{s,t}^{n+1}\left( a,b\right) &:&=x_{s,t}^{n}\otimes \beta _{t}\left(
g_{t},g_{t}^{-1}a\left( b-1\right) \right) +\beta _{s,t}^{n}\left(
g_{t},\cdot \right) \otimes \beta _{t}\left( g_{t},\cdot \right) \mathcal{I}%
\left( g_{t}^{-1}a\left( b-1\right) \right) \text{, }n\geq 1\text{.}
\label{inner definition of beta s,u tau bracket}
\end{eqnarray}%
Based on Proposition \ref{Proposition enhancement} and $\Gamma $ defined in $%
\left( \ref{Definition of Gamma}\right) $, it can be proved inductively that 
$\left\Vert \beta _{s,\cdot }^{n}\right\Vert _{\theta }^{\hat{\omega}}\leq
C_{p,\hat{\omega}\left( 0,T\right) }\left( \left\Vert \beta \right\Vert
_{\theta }^{\omega }\right) ^{n}$ for $n=1,\dots ,\left[ p\right] $ with $%
\hat{\omega}=\omega +\left\Vert g\right\Vert _{p-var}^{p}$ for any $s\in %
\left[ 0,T\right] $. Hence,%
\begin{gather}
\left\Vert B_{s,\cdot }\right\Vert _{\theta }^{\hat{\omega}}\leq C_{p,\hat{%
\omega}\left( 0,T\right) }\max \left( \left\Vert \beta \right\Vert _{\theta
}^{\omega },\left( \left\Vert \beta \right\Vert _{\theta }^{\omega }\right)
^{\left[ p\right] }\right) \text{, }\forall s\in \left[ 0,T\right] \text{,}
\label{inner norm of Beta} \\
\Gamma _{s,t}=1+\tint\nolimits_{s}^{t}B_{s,u}\left( g_{u}\right) dg_{u}\text{%
, }\left\Vert \Gamma _{s,t}-1-B_{s,s}\left( g_{s},g_{s,t}\right) \right\Vert
\leq \left\Vert B_{s,\cdot }\right\Vert _{\theta }^{\hat{\omega}}\hat{\omega}%
\left( s,t\right) ^{\theta }\text{, }\forall 0\leq s\leq t\leq T\text{.}
\label{inner transitivity Eta is close to Gamma}
\end{gather}

Then we prove a simple property of $B_{t,t}$ that for $k=1,\dots ,\left[ p%
\right] $,%
\begin{equation}
B_{t,t}\left( g_{t},v\right) \in \mathcal{U}\oplus \cdots \oplus \mathcal{U}%
^{\otimes k}\text{, }\forall v\in \mathcal{V}^{\otimes k}\text{, }\forall t%
\text{.}  \label{inner transitivity property of Hs}
\end{equation}%
Equivalently, we prove that, for $n=1,\dots ,\left[ p\right] $,%
\begin{equation}
\text{ }\beta _{t,t}^{n}\left( g_{t},v\right) =0\text{, }\forall v\in 
\mathcal{V}^{\otimes k}\text{, }k=0,1,\dots ,n-1\text{, }\forall t\text{,}
\label{inner property eta 1}
\end{equation}%
which holds for $n=1$. Suppose $\left( \ref{inner property eta 1}\right) $
holds for some $n=1,\dots ,\left[ p\right] -1$. Then, by using the property
in Condition \ref{Condition g satisfies differential equation} that $%
\mathcal{I}\left( \mathcal{V}^{\otimes k}\right) \subseteq
\tsum\nolimits_{j_{i}\geq 1,j_{1}+j_{2}=k}\mathcal{V}^{\otimes j_{1}}\otimes 
\mathcal{V}^{\otimes j_{2}}$ and by using the inductive hypothesis, we have
that, if $\beta _{t,t}^{n+1}\left( g_{t},v\right) =\beta _{t,t}^{n}\left(
g_{t},\cdot \right) \otimes \beta _{t}\left( g_{t},\cdot \right) \mathcal{I}%
\left( v\right) \neq 0$ for some $v\in \mathcal{V}^{\otimes k}$, then $%
k=j_{1}+j_{2}$ for some $j_{1}\geq n$ and $j_{2}\geq 1$, which implies $%
k\geq n+1$ and the induction is complete.

Based on the definition of $\beta _{s,t}^{n+1}$ in $\left( \ref{inner
definition of beta s,u tau bracket}\right) $, it can be proved by induction
that%
\begin{equation}
\beta _{s,t}^{n}\left( g_{t},\cdot \right)
=\tsum\nolimits_{i=1}^{n}x_{s,t}^{n-i}\otimes \beta _{t,t}^{i}\left(
g_{t},\cdot \right) \text{, }\forall v\in T^{(\left[ p\right] )}\left( 
\mathcal{V}\right) \text{.}
\label{inner relationship between two set of one-forms}
\end{equation}%
Hence the relationship holds:%
\begin{equation}
B_{s,t}\left( g_{t},\cdot \right) =\Gamma _{s,t}B_{t,t}\left( g_{t},\cdot
\right) \text{, }\forall 0\leq s\leq t\leq T\text{,}
\label{inner transitivity relationship between Beta and Eta}
\end{equation}%
where the multiplication between $\Gamma _{s,t}$ and $B_{t,t}\left(
g_{t},\cdot \right) $ is in the algebra $T^{([p])}(\mathcal{U})$.

Then we prove 
\begin{equation}
\Gamma _{0,s}\Gamma _{s,t}=\Gamma _{0,t}\text{, }\forall 0\leq s\leq t\leq T%
\text{.}  \label{inner transitivity Gamma}
\end{equation}%
Equivalently, we prove,%
\begin{equation}
x_{0,t}^{n}=\tsum\nolimits_{i=0}^{n}x_{0,s}^{n-i}\otimes x_{s,t}^{i}\text{, }%
n=1,\dots ,\left[ p\right] \text{, }0\leq s\leq t\leq T\text{,}
\label{inner transitivity}
\end{equation}%
which holds when $n=1$. Suppose $\left( \ref{inner transitivity}\right) $
holds for some $n=1,\dots ,\left[ p\right] -1$. Based on $\left( \ref{inner
relationship between two set of one-forms}\right) $ and the inductive
hypothesis $\left( \ref{inner transitivity}\right) $, we have%
\begin{equation}
\beta _{0,t}^{n}\left( g_{t},\cdot \right)
=\tsum\nolimits_{l=1}^{n}x_{0,s}^{n-l}\otimes \beta _{s,t}^{l}\left(
g_{t},\cdot \right) \text{, }\forall 0\leq s\leq t\leq T\text{.}
\label{inner Gamma taking values in group 2}
\end{equation}%
Then $\left( \ref{inner transitivity}\right) $ holds for $n+1$ based on the
definition of the integral in Proposition \ref{Proposition enhancement}, the
inductive hypothesis $\left( \ref{inner transitivity}\right) $ and $\left( %
\ref{inner Gamma taking values in group 2}\right) $.

Then we prove that, if a path is dominated by $\Gamma $, then it is
dominated by $g$. Suppose $\tint\nolimits_{0}^{\cdot }\zeta \left( \Gamma
_{0,u}\right) d\Gamma _{0,u}:\left[ 0,T\right] \rightarrow \mathcal{W}$ is a
path dominated by $t\mapsto \Gamma _{0,t}$. Define $\widetilde{\beta }:\left[
0,T\right] \rightarrow B\left( \mathcal{G}_{\left[ p\right] },\mathcal{W}%
\right) $ by%
\begin{equation*}
\widetilde{\beta }_{s}\left( a,b\right) =\zeta _{s}\left( \Gamma
_{0,s},B_{s,s}\left( g_{s},g_{s}^{-1}a\left( b-1\right) \right) \right) 
\text{, }\forall a,b\in \mathcal{G}_{\left[ p\right] }\text{.}
\end{equation*}%
That $\int_{0}^{\cdot }\widetilde{\beta }_{u}\left( g_{u}\right)
dg_{u}=\int_{0}^{\cdot }\zeta _{u}\left( \Gamma _{0,u}\right) d\Gamma _{0,u}$
follows from the uniqueness of the multiplicative function and $\left( \ref%
{inner transitivity Eta is close to Gamma}\right) $. Since $B_{t,t}\left(
g_{t},\cdot \right) $ takes values in $\mathcal{U}\oplus \cdots \oplus 
\mathcal{U}^{\otimes \left[ p\right] }$, by using $\left( \ref{equality
between two one-forms}\right) $, $\left( \ref{inner transitivity
relationship between Beta and Eta}\right) $ and $\left( \ref{inner
transitivity Gamma}\right) $, we have $\zeta _{s}\left( \Gamma
_{0,t},B_{t,t}\left( g_{t},\cdot \right) \right) =\zeta _{s}\left( \Gamma
_{0,s}\Gamma _{s,t},B_{t,t}\left( g_{t},\cdot \right) \right) =\zeta
_{s}\left( \Gamma _{0,s},\Gamma _{s,t}B_{t,t}\left( g_{t},\cdot \right)
\right) =\zeta _{s}\left( \Gamma _{0,s},B_{s,t}\left( g_{t},\cdot \right)
\right) $. Hence, 
\begin{eqnarray*}
&&\left( \widetilde{\beta }_{t}-\widetilde{\beta }_{s}\right) \left(
g_{t},\cdot \right) \\
&=&\left( \zeta _{t}-\zeta _{s}\right) \left( \Gamma _{0,t},B_{t,t}\left(
g_{t},\cdot \right) \right) +\zeta _{s}\left( \Gamma _{0,t},B_{t,t}\left(
g_{t},\cdot \right) \right) -\zeta _{s}\left( \Gamma _{0,s},B_{s,s}\left(
g_{t},\cdot \right) \right) \\
&=&\left( \zeta _{t}-\zeta _{s}\right) \left( \Gamma _{0,t},B_{t,t}\left(
g_{t},\cdot \right) \right) +\zeta _{s}\left( \Gamma _{0,s},\left(
B_{s,t}-B_{s,s}\right) \left( g_{t},\cdot \right) \right) \text{.}
\end{eqnarray*}%
Then based on $\left( \ref{inner transitivity property of Hs}\right) $, the
slowly-varying property of $t\mapsto \zeta _{t}$ and $t\mapsto B_{s,t}$, and 
$\left( \ref{inner norm of Beta}\right) $, we have $\left( \ref{continuity
of one-forms in transitivity}\right) $ holds.
\end{proof}

\section{Rough integration}

Recall that $L\left( \mathcal{V},\mathcal{U}\right) $ denotes the set of
continuous linear mappings from $\mathcal{V}$ to $\mathcal{U}$, and $L\left( 
\mathcal{V},\mathcal{U}\right) $ becomes a Banach space when equipped with
the operator norm. Recall the group $1\oplus \mathcal{U}\oplus \cdots \oplus 
\mathcal{U}^{\otimes \left[ p\right] }$ defined in $\left( \ref{definition
of group 1+U+...+Up}\right) $.

\begin{corollary}[Rough Integration]
\label{Example rough integral}Let $\mathcal{G}_{\left[ p\right] }$ be the
step-$\left[ p\right] $ nilpotent Lie group over the Banach space $\mathcal{V%
}$. For $\gamma >p-1$ and Banach space\ $\mathcal{U}$, suppose $f\in
C^{\gamma }\left( \mathcal{V},L\left( \mathcal{V},\mathcal{U}\right) \right) 
$ (as defined in $\left( \ref{definition of Cgamma}\right) $). For $g\in
C^{p-var}\left( \left[ 0,T\right] ,\mathcal{G}_{\left[ p\right] }\right) $,
we define $\beta \in \left( \left[ 0,T\right] ,B\left( \mathcal{G}_{\left[ p%
\right] },\mathcal{U}\right) \right) $ by (with $x_{s}:=\pi _{1}\left(
g_{s}\right) $ and $\pi _{l}$ denotes the projection to $\mathcal{V}%
^{\otimes l}$)%
\begin{equation}
\beta _{s}\left( a,b\right) :=\sum_{l=0}^{\left[ p\right] -1}\frac{1}{l!}%
\left( D^{l}f\right) \left( x_{s}\right) \pi _{l+1}\left( g_{s}^{-1}a\left(
b-1\right) \right) \text{, \ }\forall a,b\in \mathcal{G}_{\left[ p\right] }%
\text{, \ }\forall s\in \left[ 0,T\right] \text{.}
\label{inner definition rough integral}
\end{equation}%
Then $\int_{0}^{\cdot }\beta _{u}\left( g_{u}\right) dg_{u}$ is a dominated
path. In addition, by using the mapping $\mathcal{I}$ in $\left( \ref%
{definition of I for nilpotent Lie group}\right) $ and the integration in
Proposition \ref{Proposition enhancement}, we define $Y:\left[ 0,T\right]
\rightarrow 1\oplus \mathcal{U}\oplus \cdots \oplus \mathcal{U}^{\otimes %
\left[ p\right] }$ by%
\begin{equation*}
Y_{t}=1+\tsum\nolimits_{n=1}^{\left[ p\right] }y_{t}^{n}\text{ with }%
y_{t}^{1}:=\tint\nolimits_{0}^{t}\beta _{u}\left( g_{u}\right) dg_{u}\text{
and }y_{t}^{n}:=\tint\nolimits_{0}^{t}y_{u}^{n-1}\otimes dy_{u}^{1}\text{, }%
n=2,\dots ,\left[ p\right] \text{, }\forall t\in \left[ 0,T\right] \text{.}
\end{equation*}%
Then $Y\ $coincides with the rough integral in \cite{lyons1998differential}.
\end{corollary}

\begin{remark}
\label{Remark integrating time-varying Lipschitz functions}Let $\gamma >p-1$%
. Suppose $F:\left[ 0,T\right] \rightarrow C^{\gamma }\left( \mathcal{V}%
,L\left( \mathcal{V},\mathcal{U}\right) \right) $ is a time-varying
Lipschitz one-form, satisfying, for some control $\omega $ and $\theta >1$, 
\begin{equation}
\left\Vert \left( \left( D^{l}F_{t}\right) -\left( D^{l}F_{s}\right) \right)
\left( x_{t}\right) \right\Vert \leq \omega \left( s,t\right) ^{\theta -%
\frac{l+1}{p}}\text{, }\forall 0\leq s\leq t\leq T\text{, }l=0,1,\dots ,%
\left[ p\right] -1\text{.}
\label{condition for time-varying Lipschitz function}
\end{equation}%
If we modify the definition of $\beta _{s}$ in $\left( \ref{inner definition
rough integral}\right) $ by replacing $f$ with $F_{s}$, then it can be
proved similarly that $\int_{0}^{\cdot }\beta _{u}\left( g_{u}\right) dg_{u}$
is a dominated path.
\end{remark}

\begin{remark}
\label{Remark rough integral with inhomogeneous degree}Let $X$ be a $p$%
-rough path, $H$ be a $q$-rough path, $p^{-1}+q^{-1}>1$, $p\geq q$, and let $%
\left( x,h\right) \mapsto \alpha \left( x,h\right) $ be a Lipschitz function
that is $Lip\left( \gamma \right) $ in $x$ and $Lip\left( \kappa \right) $
in $h$. Then the rough integral $\int \alpha \left( X,H\right) dX$ is well
defined when $\gamma >p-1$ and $\left( \kappa \wedge 1\right)
q^{-1}+p^{-1}>1 $, because $t\mapsto \alpha \left( \cdot ,h_{t}\right) $ is
a time-varying Lipschitz function satisfying $\left( \ref{condition for
time-varying Lipschitz function}\right) $.
\end{remark}

\begin{remark}
\label{Remark integration of weakly geometric rough paths}The integration in
Corollary \ref{Example rough integral} is a generalization of Lyons'
original integration \cite{lyons1998differential} in the sense that $g$ can
be a weakly geometric rough path (see also \cite{cass2015integration}).
\end{remark}

\begin{remark}
\label{Remark polynomial one-form Butcher group}When $g$ takes values in
Butcher group, the one-form can be defined by, 
\begin{equation*}
\beta _{s}\left( a,b\right) :=\tsum\nolimits_{l=0}^{\left[ p\right] -1}\frac{%
1}{l!}\left( D^{l}f\right) \left( x_{s}\right) \sigma _{l+1}\left(
g_{s}^{-1}a\left( b-1\right) \right) \text{, }\forall a,b\in \mathcal{G}_{%
\left[ p\right] }\text{, \ }\forall s\in \left[ 0,T\right] \text{,}
\end{equation*}%
where $\sigma _{l+1}\in \mathcal{P}_{\left[ p\right] }$ denotes the
unlabelled tree obtained by attaching $l$ branches with one node to a new
root, e.g. $\sigma _{1}=\Bcdot$, $\sigma _{2}=[\Bcdot]$, $\sigma _{3}=[\Bcdot%
\Bcdot]$. Then $\int \beta \left( g\right) dg$ is a dominated path and
coincides with the integral in \cite{gubinelli2010ramification}.
\end{remark}

\begin{proof}
Since $D^{l}f\in C^{\gamma -l}\left( \mathcal{V},L\left( \mathcal{V}%
^{\otimes l},L\left( \mathcal{V},\mathcal{U}\right) \right) \right) $ is
symmetric in $\mathcal{V}^{\otimes l}$ and the projection of $\pi _{i}\left(
a\right) $, $a\in \mathcal{G}_{\left[ p\right] }$, to the space of symmetric
tensors is $\left( i!\right) ^{-1}\left( \pi _{1}\left( a\right) \right)
^{\otimes i}$ (see \cite{lyons1998differential}), it can be computed that,
for $a,b\in \mathcal{G}_{\left[ p\right] }$, 
\begin{equation}
\beta _{s}\left( a,b\right) =\sum_{l=0}^{\left[ p\right] -1}\left(
\sum_{j=0}^{\left[ p\right] -1-l}\left( D^{l+j}f\right) \left( x_{s}\right) 
\frac{\left( \pi _{1}\left( a\right) -x_{s}\right) ^{\otimes j}}{j!}\right)
\otimes \pi _{l+1}\left( b\right) \text{ (since }\pi _{0}\left( b\right) =1%
\text{)).}  \label{inner rough integral}
\end{equation}

Then we check that $\beta $ satisfies the slowly-varying condition. Since $%
\beta _{s}\left( a,\cdot \right) $ is a continuous linear mapping, based on
Condition \ref{Condition smallest Banach space}, the equality $\left( \ref%
{inner rough integral}\right) $ holds when $b$ is replaced by $v\in \mathcal{%
V}^{\otimes k}$, $k=1,\dots ,\left[ p\right] $. Then, for $0\leq s\leq t\leq
T$ and $k=1,2,\dots ,\left[ p\right] $, we have%
\begin{equation}
\left( \beta _{t}-\beta _{s}\right) \left( g_{t},v\right) =\left( \left(
D^{k-1}f\right) \left( x_{t}\right) -\sum_{j=0}^{\left[ p\right] -k}\left(
D^{k+j-1}f\right) \left( x_{s}\right) \frac{\left( x_{t}-x_{s}\right)
^{\otimes j}}{j!}\right) \otimes v\text{, }\forall v\in \mathcal{V}^{\otimes
k}\text{.}  \label{inner rough integral 2}
\end{equation}%
Since $f\in C^{\gamma }\left( \mathcal{V},L\left( \mathcal{V},\mathcal{U}%
\right) \right) $ for $\gamma >p-1$, by using $\left( \ref{inner rough
integral 2}\right) $ and Taylor's theorem, we have%
\begin{equation*}
\left\Vert \left( \beta _{t}-\beta _{s}\right) \left( g_{t},\cdot \right)
\right\Vert _{k}\leq C\left\Vert x_{t}-x_{s}\right\Vert ^{\gamma -k+1}\leq
C(\left\Vert g\right\Vert _{p-var,\left[ s,t\right] }^{p})^{\frac{\gamma +1}{%
p}-\frac{k}{p}}\text{.}
\end{equation*}%
Since $\{D^{k}f\}_{k=0}^{\left[ p\right] -1}$ are bounded on bounded sets,
we have $\sup_{s\in \left[ 0,T\right] }\left\Vert \beta _{s}\left(
g_{s},\cdot \right) \right\Vert <\infty $ and $\beta $ satisfies the
slowly-varying condition.

Then, by working on the local expansion of $Y$, we check that $Y$ coincides
with the rough integral in \cite{lyons1998differential}.

For $0\leq s\leq t\leq T$, we denote 
\begin{equation*}
Y_{s,t}:=Y_{s}^{-1}Y_{t}\text{ \ and \ }y_{s,t}^{n}:=\pi _{n}\left(
Y_{s,t}\right) \text{, }n=0,1,\dots ,\left[ p\right] \text{.}
\end{equation*}%
We define $H_{s}\in B(\mathcal{G}_{\left[ p\right] },\mathcal{U\oplus \cdots
\oplus U}^{\otimes \left[ p\right] })$ for $s\in \left[ 0,T\right] $ by%
\begin{equation*}
H_{s}=\tsum\nolimits_{n=1}^{\left[ p\right] }\eta _{s}^{n}\text{, }
\end{equation*}%
where $\eta _{s}^{n}\in B(\mathcal{G}_{\left[ p\right] },\mathcal{U}%
^{\otimes n})$ are defined by 
\begin{equation*}
\eta _{s}^{1}\left( a,b\right) :=\beta _{s}\left( a,b\right) \text{, }\eta
_{s}^{n}\left( a,b\right) :=\eta _{s}^{n-1}\left( g_{s},\cdot \right)
\otimes \eta _{s}^{1}\left( g_{s},\cdot \right) \mathcal{I}\left(
g_{s}^{-1}a\left( b-1\right) \right) \text{, }\forall a,b\in \mathcal{G}_{%
\left[ p\right] }\text{.}
\end{equation*}%
Hence $H_{s}=B_{s,s}$ for $B_{s,s}$ defined at $\left( \ref{inner definition
of Beta s,t}\right) $ and $\eta _{s}^{n}=\beta _{s,s}^{n}$ for $\beta
_{s,s}^{n}$ defined at $\left( \ref{inner definition of beta s,u tau bracket}%
\right) $. Then based on $\left( \ref{inner transitivity Eta is close to
Gamma}\right) $, there exist control $\omega $ and $\theta >1$, s.t.%
\begin{equation}
\left\Vert Y_{s,t}-1-H_{s}\left( g_{s},g_{s,t}\right) \right\Vert \leq
\omega \left( s,t\right) ^{\theta }\text{, }\forall 0\leq s\leq t\leq T\text{%
.}  \label{inner rough integral estimate of Ys,t in term of Beta}
\end{equation}%
To write $H_{s}\left( g_{s},g_{s,t}\right) $ in a more explicit form, we
define the mappings $\mathcal{I}^{n}\in L(T^{([p])}\left( \mathcal{V}\right)
,(T^{([p])}(\mathcal{V}))^{\otimes \left( n+1\right) })$, $n=1,\dots ,\left[
p\right] -1$, by (with $\mathcal{I}$ in $\left( \ref{definition of I for
nilpotent Lie group}\right) $ and $I_{d}$ denotes the identity function on $%
T^{([p])}(\mathcal{V})$):%
\begin{equation}
\mathcal{I}^{1}:=\mathcal{I}\text{, \ }\mathcal{I}^{n}=\left( \mathcal{I}%
^{n-1}\otimes I_{d}\right) \circ \mathcal{I}^{1}\text{, }n=2,\dots ,\left[ p%
\right] -1\text{.}  \label{inner definition I^n}
\end{equation}%
Then, it can be proved inductively that%
\begin{equation*}
\eta _{s}^{1}\left( g_{s},g_{s,t}\right) =\beta _{s}\left(
g_{s},g_{s,t}\right) \text{ and }\eta _{s}^{n}\left( g_{s},g_{s,t}\right)
=\beta _{s}\left( g_{s},\cdot \right) ^{\otimes n}\mathcal{I}^{n-1}\left(
g_{s,t}\right) \text{, }n=2,\dots ,\left[ p\right] \text{.}
\end{equation*}%
Combined with $\left( \ref{inner rough integral estimate of Ys,t in term of
Beta}\right) $, we have that, there exist control $\omega $ and $\theta >1$
such that for any $0\leq s\leq t\leq T$,%
\begin{equation}
\left\Vert y_{s,t}^{1}-\beta _{s}\left( g_{s},g_{s,t}\right) \right\Vert
\leq \omega \left( s,t\right) ^{\theta }\text{, }\left\Vert
y_{s,t}^{n}-\beta _{s}\left( g_{s},\cdot \right) ^{\otimes n}\mathcal{I}%
^{n-1}\left( g_{s,t}\right) \right\Vert \leq \omega \left( s,t\right)
^{\theta }\text{, }n=2,\dots ,\left[ p\right] \text{.}
\label{inner rough integral 3}
\end{equation}

Then we check that 
\begin{equation}
\mathcal{I}^{n}\left( a\right) =\tsum\nolimits_{j_{1}+\cdots +j_{n+1}\leq %
\left[ p\right] ,j_{i}\geq 1}\tsum\nolimits_{\rho \in OS\left( j_{1},\dots
,j_{n+1}\right) }\rho ^{-1}\left( \pi _{j_{1}+\cdots +j_{n+1}}\left(
a\right) \right) \text{, }n=1,\dots ,\left[ p\right] -1\text{, }\forall a\in 
\mathcal{G}_{\left[ p\right] }\text{,}
\label{inner rough integral expression of I n-1}
\end{equation}%
where $OS\left( j_{1},\dots ,j_{n+1}\right) $ denotes the ordered shuffles
(p73 \cite{lyons2007differential}). Indeed, we first suppose $a$ is an
element in the step-$\left[ p\right] $ nilpotent Lie group over a finite
dimensional subspace of $\mathcal{V}$. Then based on Chow-Rashevskii
Connectivity Theorem, there exists some continuous bounded variation path $x$%
, which takes values in the finite dimensional subspace of $\mathcal{V}$ and
satisfies $S_{\left[ p\right] }\left( x\right) _{0,1}=a$. By comparing the
expression on the r.h.s.\ of $\left( \ref{inner rough integral expression of
I n-1}\right) $ with the expression $\left( 4.9\right) $ on p74 in \cite%
{lyons2007differential}, we can rewrite $\left( \ref{inner rough integral
expression of I n-1}\right) $ in the form (with $1_{\left[ p\right] ,n+1}$
denotes the projection of $(T^{([p])}(\mathcal{V}))^{\otimes \left(
n+1\right) }$ to $\sum_{j_{1}+\cdots +j_{n+1}\leq \left[ p\right] ,j_{i}\geq
1}\mathcal{V}^{\otimes j_{1}}\otimes \cdots \otimes \mathcal{V}^{\otimes
j_{n+1}}$)%
\begin{equation}
\mathcal{I}^{n}\left( a\right) =1_{\left[ p\right] ,n+1}\left(
\tidotsint\nolimits_{0<u_{1}<\cdots <u_{n+1}<1}dS_{\left[ p\right] }\left(
x\right) _{0,u_{1}}\otimes \cdots \otimes dS_{\left[ p\right] }\left(
x\right) _{0,u_{n+1}}\right) \text{.}
\label{inner rough integral expression of I in finite dimensional subspace}
\end{equation}%
Hence, based on the definition of $\mathcal{I}^{n}$ in $\left( \ref{inner
definition I^n}\right) $, it can be proved inductively that $\left( \ref%
{inner rough integral expression of I in finite dimensional subspace}\right) 
$ (so $\left( \ref{inner rough integral expression of I n-1}\right) $) holds
for all the elements in the step-$\left[ p\right] $ nilpotent Lie group over
any finite dimensional subspace of $\mathcal{V}$. Then by using continuity,
we have $\left( \ref{inner rough integral expression of I n-1}\right) $
holds for any $a\in \mathcal{G}_{\left[ p\right] }$.

Then we check that $Y$ coincides with the rough integral in \cite%
{lyons1998differential}. Indeed, based on $\left( \ref{inner rough integral
3}\right) $, if we define $X:\left\{ \left( s,t\right) |0\leq s\leq t\leq
T\right\} \rightarrow 1\oplus \mathcal{U}\oplus \cdots \oplus \mathcal{U}%
^{\otimes \left[ p\right] }$ by%
\begin{equation*}
X_{s,t}:=1+\beta _{s}\left( g_{s},g_{s,t}\right) +\tsum\nolimits_{n=2}^{ 
\left[ p\right] }\beta _{s}\left( g_{s},\cdot \right) ^{\otimes n}\mathcal{I}%
^{n-1}\left( g_{s,t}\right) \text{, }\forall 0\leq s\leq t\leq T\text{,}
\end{equation*}%
then $X$ is a almost multiplicative functional (Def 3.1.1 \cite%
{lyons1998differential}), and $Y$ is a $p$-rough path associated with $X$
(i.e. $Y$ is multiplicative and there exist control $\omega $ and $\theta >1$
such that $Y_{s,t}$ and $X_{s,t}$ are close up to an error bounded by $%
\omega \left( s,t\right) ^{\theta }$ for any $s<t$). On the other hand,
based on Def 3.2.2 and Theorem 3.2.1 \cite{lyons1998differential}, the rough
integral is another $p$-rough path associated with $X$. Since the $p$-rough
path associated with $X$ is unique (Theorem 3.3.1 \cite%
{lyons1998differential}), we have that $Y$ coincides with the rough integral
in \cite{lyons1998differential}.
\end{proof}

\section{Iterated integration for weakly controlled paths}

Recall in Section \ref{Section dominated path} that Condition \ref{Condition
smallest Banach space} states that $T^{(\left[ p\right] )}(\mathcal{V)}:%
\mathcal{=%
\mathbb{R}
\oplus V\oplus \cdots \oplus V}^{\otimes \left[ p\right] }$ is the closure
of the linear span of the topological group $\mathcal{G}_{\left[ p\right] }$%
; Condition \ref{Condition g satisfies differential equation} requires the
existence of a mapping $\mathcal{I}$ which expresses the formal integral $%
\iint_{0<u_{1}<u_{2}<T}\delta g_{0,u_{1}}\otimes \delta g_{0,u_{2}}$ for $%
g\in C\left( \left[ 0,T\right] ,\mathcal{G}_{\left[ p\right] }\right) $ as a
universal continuous linear mapping of $g_{0,T}$. In particular, Conditions %
\ref{Condition smallest Banach space} and \ref{Condition g satisfies
differential equation} are satisfied when $(T^{(\left[ p\right] )}(\mathcal{%
V)},\mathcal{G}_{\left[ p\right] },\mathcal{P}_{\left[ p\right] })$ is the
triple for the step-$\left[ p\right] $ free nilpotent Lie group $p\geq 1$ or
the step-$2$ Butcher group $2\leq p<3$.

\begin{corollary}[Iterated Integration for Weakly Controlled Paths]
\label{Example weakly controlled path}Suppose $(T^{(\left[ p\right] )}(%
\mathcal{V)},\mathcal{G}_{\left[ p\right] },\mathcal{P}_{\left[ p\right] })$
satisfies Conditions \ref{Condition smallest Banach space} and \ref%
{Condition g satisfies differential equation}, and $g\in C^{p-var}\left( %
\left[ 0,T\right] ,\mathcal{G}_{\left[ p\right] }\right) $ for some $p\geq 2$%
. Suppose $\mathcal{U}^{i}$, $i=1,2$, are two Banach spaces, and there
exist\ control $\omega $ and $\theta >1$ such that $\gamma ^{i}\in C\left( %
\left[ 0,T\right] ,\mathcal{U}^{i}\right) $ and $\beta ^{i}\in C\left( \left[
0,T\right] ,B\left( \mathcal{G}_{\left[ p\right] -1},\mathcal{U}^{i}\right)
\right) $, $i=1,2$, satisfy that%
\begin{gather}
M^{i}:=\sup_{0\leq t\leq T}\left\Vert \beta _{t}^{i}\left( g_{t},\cdot
\right) \right\Vert +\sup_{0\leq s<t\leq T}\max_{k=1,\dots ,\left[ p\right]
-1}\frac{\left\Vert \left( \beta _{t}^{i}-\beta _{s}^{i}\right) \left(
g_{t},\cdot \right) \right\Vert _{k}}{\omega \left( s,t\right) ^{\theta -%
\frac{k+1}{p}}}  \label{condition weakly controlled path} \\
+\sup_{0\leq s<t\leq T}\frac{\left\Vert \gamma _{t}^{i}-\gamma
_{s}^{i}-\beta _{s}^{i}\left( g_{s},g_{s,t}\right) \right\Vert }{\omega
\left( s,t\right) ^{\theta -\frac{1}{p}}}<\infty \text{, }i=1,2\text{.} 
\notag
\end{gather}%
Let us define $h\in C\left( \left[ 0,T\right] ,\mathcal{U}^{2}\oplus 
\mathcal{G}_{\left[ p\right] }\right) $ by $h:=\gamma ^{2}\oplus g$ with $%
h_{s,t}=\left( \gamma _{t}^{2}-\gamma _{s}^{2}\right) \oplus g_{s,t}$, $%
\forall 0\leq s\leq t\leq T$. Then there exists $\beta \in C\left( \left[ 0,T%
\right] ,B\left( \mathcal{U}^{2}\oplus \mathcal{G}_{\left[ p\right] },%
\mathcal{U}^{1}\otimes \mathcal{U}^{2}\right) \right) $ such that $\left(
\beta ,h\right) $ satisfies the integrable condition (Condition \ref%
{Condition integrable condition}), and with control $\hat{\omega}:=\omega
+\left\Vert g\right\Vert _{p-var}^{p}$ and $\hat{\theta}:=\min (\theta ,%
\frac{\left[ p\right] +1}{p})>1$, we have that ($\mathcal{I}$ in Condition %
\ref{Condition g satisfies differential equation})%
\begin{eqnarray*}
&&\left\Vert \int_{s}^{t}\beta _{u}\left( h_{u}\right) dh_{u}-\left( \gamma
_{s}^{1}-\gamma _{0}^{1}\right) \otimes \left( \gamma _{t}^{2}-\gamma
_{s}^{2}\right) -\beta _{s}^{1}\left( g_{s},\cdot \right) \otimes \beta
_{s}^{2}\left( g_{s},\cdot \right) \mathcal{I}\left( g_{s,t}\right)
\right\Vert \\
&\leq &C_{p,\hat{\omega}\left( 0,T\right) }M^{1}M^{2}\hat{\omega}\left(
s,t\right) ^{\hat{\theta}}\text{, }\forall 0\leq s<t\leq T\text{.}
\end{eqnarray*}
\end{corollary}

\begin{remark}
The integral path $\int_{0}^{\cdot }\beta \left( h\right) dh$ is continuous
in $p$-variation w.r.t. $M^{i}$ that is a norm involving both $\gamma ^{i}$
and $\beta ^{i}$. It is hard to derive the continuity in operator norm of
the one-form $\beta $ in terms of $\gamma ^{i}$ and $\beta ^{i}$ (comparing
with $\left( \ref{continuity of one-forms iterated integral}\right) $),
because the constraint $\left\Vert \gamma _{t}^{2}-\gamma _{s}^{2}-\beta
_{s}^{2}\left( g_{s},g_{s,t}\right) \right\Vert \leq M^{2}\omega \left(
s,t\right) ^{\theta -\frac{1}{p}}$ prescribes directions to evaluate $\beta $%
. The one-form $\beta $ would not vary slowly as a linear operator (that
induces strong continuity as demonstrated in Section \ref{Section stableness
of dominated paths}); it only varies slowly in (a neighborhood of) the
future direction of the path.
\end{remark}

\begin{proof}
With the mapping $\mathcal{I}\in L(T^{\left( \left[ p\right] \right) }\left( 
\mathcal{V}\right) ,T^{\left( \left[ p\right] \right) }\left( \mathcal{V}%
\right) ^{\otimes 2})$ in Condition \ref{Condition g satisfies differential
equation}, for $s\in \left[ 0,T\right] $, we define $\beta _{s}^{1,2}\in
C\left( \mathcal{G}_{\left[ p\right] },L\left( T^{\left( \left[ p\right]
\right) }\left( \mathcal{V}\right) ,\mathcal{U}^{1}\otimes \mathcal{U}%
^{2}\right) \right) $ by%
\begin{equation*}
\beta _{s}^{1,2}\left( a,v\right) :=\beta _{s}^{1}\left( a,\cdot \right)
\otimes \beta _{s}^{2}\left( a,\cdot \right) \mathcal{I}\left( v\right) 
\text{, }\forall a\in \mathcal{G}_{\left[ p\right] }\text{, }\forall v\in
T^{\left( \left[ p\right] \right) }\left( \mathcal{V}\right) \text{,}
\end{equation*}%
and define $\beta \in C\left( \left[ 0,T\right] ,B\left( \mathcal{U}%
^{2}\oplus \mathcal{G}_{\left[ p\right] },\mathcal{U}^{1}\otimes \mathcal{U}%
^{2}\right) \right) $ by%
\begin{equation*}
\beta _{s}\left( u\oplus a,v\oplus b\right) :=\left( \gamma _{s}^{1}-\gamma
_{0}^{1}\right) \otimes v+\beta _{s}^{1,2}\left( g_{s},g_{s}^{-1}a\left(
b-1\right) \right) \text{, }\forall \left( u\oplus a\right) ,\left( v\oplus
b\right) \in \mathcal{U}^{2}\oplus \mathcal{G}_{\left[ p\right] }\text{, }%
\forall s\in \left[ 0,T\right] \text{.}
\end{equation*}%
Then we check that $\left( \beta ,h\right) $ satisfies the integrable
condition. For $s<u<t$, similar to the argument used in the proof of the
iterated integration for dominated paths (see $\left( \ref{inner bound on
beta1,2}\right) $ on page \pageref{inner bound on beta1,2}), we have%
\begin{eqnarray*}
\left\Vert \left( \beta _{u}^{1,2}-\beta _{s}^{1,2}\right) \left(
g_{u},\cdot \right) \right\Vert _{k} &\leq &C_{p,\omega \left( 0,T\right)
}M^{1}M^{2}\omega \left( s,u\right) ^{\theta -\frac{1}{p}-\frac{k}{p}+\frac{1%
}{p}} \\
&=&C_{p,\omega \left( 0,T\right) }M^{1}M^{2}\omega \left( s,u\right)
^{\theta -\frac{k}{p}}\text{, }\forall s<u\text{, }k=1,\dots \left[ p\right] 
\text{.}
\end{eqnarray*}%
Then, 
\begin{equation}
\left\Vert \left( \beta _{u}^{1,2}-\beta _{s}^{1,2}\right) \left(
g_{u},g_{u,t}\right) \right\Vert \leq \tsum\limits_{\sigma \in \mathcal{P}_{%
\left[ p\right] }}\left\Vert \left( \beta _{u}^{1,2}-\beta _{s}^{1,2}\right)
\left( g_{u},\cdot \right) \right\Vert _{\left\vert \sigma \right\vert
}\left\Vert \sigma \left( g_{u,t}\right) \right\Vert \leq C_{p,\omega \left(
0,T\right) }M^{1}M^{2}\hat{\omega}\left( s,t\right) ^{\theta }\text{.}
\label{inner weakly controlled 1}
\end{equation}%
Based on $\left( \ref{beta12 group-valued cocyclic one-form}\right) $, for $%
s<u<t$,%
\begin{eqnarray*}
\beta _{s}^{1,2}\left( g_{s},g_{s,t}\right) &=&\beta _{s}^{1,2}\left(
g_{s},g_{s,u}\right) +\beta _{s}^{1,2}\left( g_{u},g_{u,t}\right) +\beta
_{s}^{1}\left( g_{s},g_{s,u}\right) \otimes \beta _{s}^{2}\left(
g_{u},g_{u,t}\right) \\
&&-\tsum\nolimits_{\sigma _{i}\in \mathcal{P}_{\left[ p\right] },\left\vert
\sigma _{1}\right\vert +\left\vert \sigma _{2}\right\vert \geq \left[ p%
\right] +1}\beta _{s}^{1}\left( g_{s},\sigma _{1}\left( g_{s,u}-1\right)
\right) \otimes \beta _{s}^{2}\left( g_{s},\sigma _{2}\left( g_{s,u}\left(
g_{u,t}-1\right) \right) \right) \\
&&-\tsum\nolimits_{k=2}^{\left[ p\right] }\tsum\nolimits_{\sigma _{i}\in 
\mathcal{P}_{\left[ p\right] },\left\vert \sigma _{1}\right\vert +\left\vert
\sigma _{2}\right\vert \geq \left[ p\right] +1-k}\beta _{s}^{1}\left(
g_{s},\sigma _{1}\left( g_{s,u}\right) \cdot \right) \otimes \beta
_{s}^{2}\left( g_{s},\sigma _{2}\left( g_{s,u}\right) \cdot \right) \mathcal{%
I}\left( \pi _{k}\left( g_{u,t}\right) \right) \text{.}
\end{eqnarray*}%
Then%
\begin{eqnarray*}
&&\left( \beta _{u}-\beta _{s}\right) \left( h_{u},h_{u,t}\right) \\
&=&\left( \gamma _{u}^{1}-\gamma _{s}^{1}-\beta _{s}^{1}\left(
g_{s},g_{s,u}\right) \right) \otimes \left( \gamma _{t}^{2}-\gamma
_{u}^{2}\right) +\beta _{s}^{1}\left( g_{s},g_{s,u}\right) \otimes \left(
\gamma _{t}^{2}-\gamma _{u}^{2}-\beta _{u}^{2}\left( g_{u},g_{u,t}\right)
\right) \\
&&+\beta _{s}^{1}\left( g_{s},g_{s,u}\right) \otimes \left( \beta
_{u}^{2}-\beta _{s}^{2}\right) \left( g_{u},g_{u,t}\right) +\left( \beta
_{u}^{1,2}-\beta _{s}^{1,2}\right) \left( g_{u},g_{u,t}\right) \\
&&+\tsum\nolimits_{\sigma _{i}\in \mathcal{P}_{\left[ p\right] },\left\vert
\sigma _{1}\right\vert +\left\vert \sigma _{2}\right\vert \geq \left[ p%
\right] +1}\beta _{s}^{1}\left( g_{s},\sigma _{1}\left( g_{s,u}\right)
\right) \otimes \beta _{s}^{2}\left( g_{s},\sigma _{2}\left( g_{s,u}\left(
g_{u,t}-1\right) \right) \right) \\
&&+\tsum\nolimits_{k=2}^{\left[ p\right] }\tsum\nolimits_{\sigma _{i}\in 
\mathcal{P}_{\left[ p\right] },\left\vert \sigma _{1}\right\vert +\left\vert
\sigma _{2}\right\vert \geq \left[ p\right] +1-k}\beta _{s}^{1}\left(
g_{s},\sigma _{1}\left( g_{s,u}\right) \cdot \right) \otimes \beta
_{s}^{2}\left( g_{s},\sigma _{2}\left( g_{s,u}\right) \cdot \right) \mathcal{%
I}\left( \pi _{k}\left( g_{u,t}\right) \right) \text{.}
\end{eqnarray*}%
Then combined with $\left( \ref{condition weakly controlled path}\right) $, $%
\left( \ref{inner weakly controlled 1}\right) $ and that $\left\Vert \sigma
\left( g_{s,t}\right) \right\Vert \leq \left\Vert g\right\Vert _{p-var,\left[
s,t\right] }^{\left\vert \sigma \right\vert }$, $\left( \beta ,h\right) $
satisfies the integrable condition (Condition \ref{Condition integrable
condition}), and the local estimate follows from Theorem \ref{Theorem
integrating slow varying cocyclic one forms}.
\end{proof}

\bigskip

\noindent \textbf{Acknowledgement} \ The authors would like to thank Horatio
Boedihardjo, Ilya Chevyrev and Vlad Margarint for valuable suggestions on
the paper.

\bibliographystyle{abbrv}
\bibliography{acompat,roughpath}

\newif\ifabfull\abfulltrue
\begin{thebibliography}{10}

\bibitem{boedihardjo2014signature}
H.~Boedihardjo, X.~Geng, T.~Lyons, and D.~Yang.
\newblock The signature of a rough path: Uniqueness.
\newblock {\em arXiv preprint arXiv:1406.7871}, 2014.

\bibitem{butcher1972algebraic}
J.~C. Butcher.
\newblock An algebraic theory of integration methods.
\newblock {\em Mathematics of Computation}, 26(117):79--106, 1972.

\bibitem{cass2015integration}
T.~Cass, B.~K. Driver, N.~Lim, and C.~Litterer.
\newblock On the integration of weakly geometric rough paths.
\newblock {\em Journal of the Mathematical Society of Japan}, 2015.

\bibitem{chen1957integration}
K.-T. Chen.
\newblock Integration of paths, geometric invariants and a generalized
  baker-hausdorff formula.
\newblock {\em Annals of Mathematics}, pages 163--178, 1957.

\bibitem{chen2001iterated}
K.-T. Chen.
\newblock Iterated integrals and exponential homomorphisms.
\newblock {\em Collected Papers of K.-T. Chen}, 3:54, 2001.

\bibitem{connes1998hopf}
A.~Connes and D.~Kreimer.
\newblock Hopf algebras, renormalization and noncommutative geometry.
\newblock {\em Communications in Mathematical Physics}, 199(1):203--242, 1998.

\bibitem{Davie07differentialequations}
A.~M. Davie.
\newblock Differential equations driven by rough paths: an approach via
  discrete approximation.
\newblock {\em Appl. Math. Res. Express}, pages 009--40, 2007.

\bibitem{feyel2006curvilinear}
D.~Feyel and A.~de~La~Pradelle.
\newblock Curvilinear integrals along enriched paths.
\newblock {\em Electron. J. Probab}, 11(34):860--892, 2006.

\bibitem{feyel2008non}
D.~Feyel, A.~de~La~Pradelle, and G.~Mokobodzki.
\newblock A non-commutative sewing lemma.
\newblock {\em Electron. Commun. Probab}, 13:24--34, 2008.

\bibitem{friz2014course}
P.~Friz and M.~Hairer.
\newblock {\em A course on rough paths, with an introduction to regularity
  structures}.
\newblock Springer, 2014.

\bibitem{friz2012doob}
P.~Friz and A.~Shekhar.
\newblock Doob--meyer for rough paths.
\newblock {\em Bulletin of the Institute of Mathematics Academia Sinica (New
  Series)}, 8(1):73--84, 2013.

\bibitem{friz2008euler}
P.~Friz and N.~Victoir.
\newblock Euler estimates for rough differential equations.
\newblock {\em Journal of Differential Equations}, 244(2):388--412, 2008.

\bibitem{friz2010multidimensional}
P.~K. Friz and N.~B. Victoir.
\newblock {\em Multidimensional stochastic processes as rough paths: theory and
  applications}, volume 120.
\newblock Cambridge University Press, 2010.

\bibitem{gubinelli2004controlling}
M.~Gubinelli.
\newblock Controlling rough paths.
\newblock {\em Journal of Functional Analysis}, 216(1):86--140, 2004.

\bibitem{gubinelli2010ramification}
M.~Gubinelli.
\newblock Ramification of rough paths.
\newblock {\em Journal of Differential Equations}, 248(4):693--721, 2010.

\bibitem{hairer2014theory}
M.~Hairer.
\newblock A theory of regularity structures.
\newblock {\em Invent. Math.}, 2014.

\bibitem{hairer2015geometric}
M.~Hairer and D.~Kelly.
\newblock Geometric versus non-geometric rough paths.
\newblock {\em Annales de l'Institut Henri Poincar{\'e}, Probabilit{\'e}s et
  Statistiques}, 51(1):207--251, 2015.

\bibitem{hambly2010uniqueness}
B.~Hambly and T.~Lyons.
\newblock Uniqueness for the signature of a path of bounded variation and the
  reduced path group.
\newblock {\em Ann. of Math.(2)}, 171(1):109--167, 2010.

\bibitem{hu2009rough}
Y.~Hu and D.~Nualart.
\newblock Rough path analysis via fractional calculus.
\newblock {\em Transactions of the American Mathematical Society},
  361(5):2689--2718, 2009.

\bibitem{lejay2003introduction}
A.~Lejay.
\newblock An introduction to rough paths.
\newblock In {\em S{\'e}minaire de probabilit{\'e}s XXXVII}, pages 1--59.
  Springer, 2003.

\bibitem{lejay2009yet}
A.~Lejay.
\newblock Yet another introduction to rough paths.
\newblock In {\em S{\'e}minaire de probabilit{\'e}s XLII}, pages 1--101.
  Springer, 2009.

\bibitem{lejay2006p}
A.~Lejay and N.~Victoir.
\newblock On (p, q)-rough paths.
\newblock {\em Journal of Differential Equations}, 225(1):103--133, 2006.

\bibitem{lyons2002system}
T.~Lyons and Z.~Qian.
\newblock {\em System control and rough paths}.
\newblock Oxford University Press, 2002.

\bibitem{lyons2007extension}
T.~Lyons and N.~Victoir.
\newblock An extension theorem to rough paths.
\newblock In {\em Annales de l'Institut Henri Poincare (C) Non Linear
  Analysis}, volume~24, pages 835--847. Elsevier, 2007.

\bibitem{lyons1998differential}
T.~J. Lyons.
\newblock Differential equations driven by rough signals.
\newblock {\em Rev. Mat. Iberoamericana}, 14(2), 1998.

\bibitem{lyons2007differential}
T.~J. Lyons, M.~Caruana, and T.~L{\'e}vy.
\newblock {\em Differential equations driven by rough paths}.
\newblock Springer, 2007.

\bibitem{lyons2015theory}
T.~J. Lyons and D.~Yang.
\newblock The theory of rough paths via one-forms and the extension of an
  argument of schwartz to rough differential equations.
\newblock {\em Journal of the Mathematical Society of Japan}, 67(4):1681--1703,
  2015.

\bibitem{reutenauer2003free}
C.~Reutenauer.
\newblock Free lie algebras.
\newblock {\em Handbook of algebra}, 3:887--903, 2003.

\bibitem{schwartz1989convergence}
L.~Schwartz.
\newblock La convergence de la s{\'e}rie de picard pour les eds (equations
  diff{\'e}rentielles stochastiques).
\newblock In {\em S{\'e}minaire de Probabilit{\'e}s XXIII}, pages 343--354.
  Springer, 1989.

\bibitem{stein1970singular}
E.~M. Stein.
\newblock {\em Singular integrals and differentiability properties of
  functions}, volume~2.
\newblock Princeton university press, 1970.

\bibitem{young1936inequality}
L.~C. Young.
\newblock An inequality of the h{\"o}lder type, connected with stieltjes
  integration.
\newblock {\em Acta Mathematica}, 67(1):251--282, 1936.

\end{thebibliography}

\end{document}